\documentclass{amsart}

\usepackage{hyperref,amssymb,graphicx,subcaption}

\def\Riem{\mathop{\rm Rm}}

\newtheorem{theorem}{Theorem}[section]
\newtheorem{proposition}[theorem]{Proposition}
\newtheorem{lemma}[theorem]{Lemma}
\newtheorem{corollary}[theorem]{Corollary}

\title[A Classification of toric instantons]{A classification of scalar-flat toric K\"ahler instantons in dimension 4}

\author{Brian Weber}

\begin{document}

\maketitle

\begin{abstract}
	We classify all scalar-flat toric K\"ahler 4-manifolds under either of two asymptotic conditions: that the action fields decay slowly (or at all), or that the curvature decay is quadratic; for example we fully classify instantons that have any of the ALE-F-G-H asymptotic types.
	The momentum functions satisfy a degenerate elliptic equation, and under either asymptotic condition the image of the moment map is closed.
	Using a recent Liouville theorem for degenerate-elliptic equations, we classify all possibilities for the momentum functions, and from this, all possible metrics.
\end{abstract}

\section{Introduction}

A K\"ahler manifold $(M^4,J,g)$ with a 2-torus action that preserves the symplectic form and the metric is said to be a toric K\"ahler 4-manifold.
This paper classifies all scalar-flat toric K\"ahler manifolds under an asymptotic condition that is satisfied, for instance, on manifolds with ALE, ALF, ALG or ALH ends, or on any manifold with $o(r^{-2})$ curvature decay.

Rather than a torus action we work within the equivalent but more flexible situation that $(M^4,J,g,\mathcal{X}^1,\mathcal{X}^2)$ is a K\"ahler manifold with symplectomorphic Killing fields $\mathcal{X}^1$, $\mathcal{X}^2$ that commute: $[\mathcal{X}^1,\mathcal{X}^2]=0$.
This allows us to take linear combinations of the generators without worrying whether or not they come from torus isomorphisms.
With the action of the fields being both symplectic and isometric, the symplectic reduction coincides with the Riemannian quotient, and either produces a 2-dimensional reduced manifold $(\Sigma^2,g_\Sigma)$.
From the Riemannian point of view, $\Sigma^2$ can be viewed as the leaf-space of the isometric $\mathcal{X}^1$-$\mathcal{X}^2$ action, and $\Sigma^2$ inherits a metric $g_\Sigma$.
From the symplectic point of view, one constructs action-angle coordinates $(\varphi^1,\varphi^2,\theta_1,\theta_2)$ consisting of {\it action} or {\it momentum} variables $\varphi^1$, $\varphi^2$ given (implicitly) by $d\varphi^i=i_{\mathcal{X}^i}\omega$ and {\it angle} variables $\theta_1$, $\theta_2$ given (implicitly) by $\frac{\partial}{\partial\theta_i}=\mathcal{X}^i$.
The reduction map $\Phi:M^4\rightarrow\Sigma^2\subseteq\mathbb{R}^2$ is $(\varphi^1,\varphi^2,\theta_1,\theta_2)\mapsto(\varphi^1,\varphi^2)$, and is called the {\it Arnold-Liouville reduction} \cite{Ar}.

The manifold $(\Sigma^2,g_\Sigma)$ encodes all topological, complex-analytic, symplectic, and differential-geometric information of the original manifold $(M^4,J,g)$, but since it is 2-dimensional it is easier to study.
In a wide class of the most natural examples, the manifold $\Sigma^2$ is a topologically closed convex polygon (Figure \ref{SubFigInitClosed}), but in some cases is neither closed (Figure \ref{SubFigInitTwoEnded}) nor even a polygon (Example \ref{SubsubsecNonpolygon}).
This paper uses a recent Liouville theorem for a class of boundary-degenerate PDE to classify all possibilities for the pair $(\Sigma^2,g_\Sigma)$ when $\Sigma^2$ is closed and the parent manifold $(M^4,g)$ is scalar flat.
\noindent\begin{figure}[h]
	\centering
	\noindent\begin{subfigure}[b]{0.45\textwidth} 
		\centerline{\includegraphics[scale=0.6]{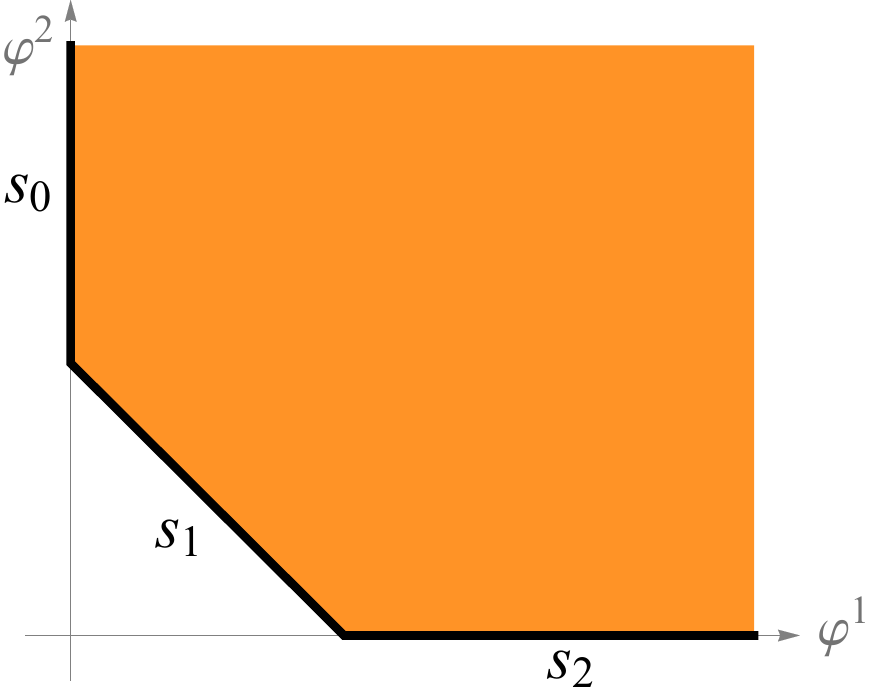}}
		\caption{Closed polygon.
			The polygon for the various ZSC toric K\"ahler metrics on $\mathcal{O}(-k)$, $k\ge1$, including the LeBrun metrics.
			The ``boundary conditions'' are the real numbers (not functions) $s_0$, $s_1$, $s_2$.
			See \S\S\ref{SecAnalytic} and \ref{SubsecBoundaryConds}, and the explicit computation in \S\ref{SubsubsecOk}.
			\label{SubFigInitClosed}}
	\end{subfigure}
	\hspace{0.25in}
	\begin{subfigure}[b]{0.45\textwidth} 
		\centerline{\includegraphics[scale=0.6]{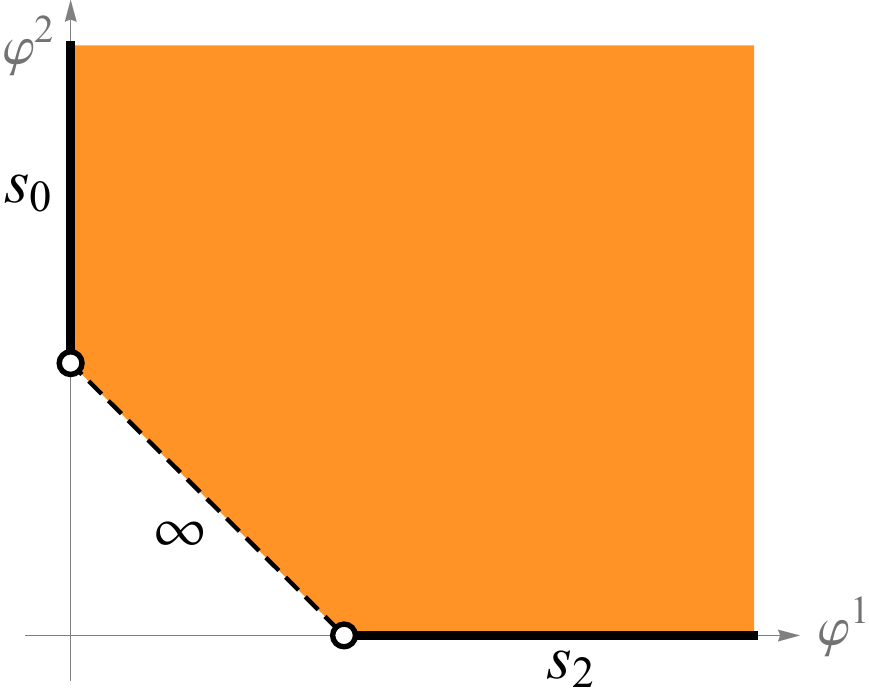}}
		\caption{Non-closed polygon.
			This is the polygon for certain 2-ended complete ZSC metrics on $\mathbb{C}^2\setminus\{0\}$.
			One end is ALE and the other resembles a sphere cross a pseudosphere; see \S\ref{SubsubsecTwoEnded}.
			The boundary condition, or label, on the missing segment is $\infty$.
			\label{SubFigInitTwoEnded}}
	\end{subfigure}
	
	\caption{\it Non-compact polygons, one closed and one not.}
	\label{FigTwoPolys}
\end{figure}

The closure condition on $\Sigma^2$ may seem technical, but it is a natural restriction on the kinds of ends\footnote{Recall that a {\it manifold end} of $M^4$ is an unbounded component of $M^4\setminus{}K$ where $K$ is a precompact domain.} the parent $M^4$ might have.
Polygon edges correspond to zeros of the Killing fields, and represent totally geodesic submanifolds of $M^4$.
The polygon's end corresponds to a manifold end.
A polygon ``missing'' segment may also correspond to a manifold end, but could potentially be pathological.
Assuming the parent satisfies either of the asymptotic conditions
\begin{itemize}
	\item[A1)] The action fields decay slowly (or not at all): $\sqrt{|\mathcal{X}^1|^2+|\mathcal{X}^2|^2}\;>\;C_1r^{-1+\epsilon}$ for some $\epsilon>0$ and $C_1>\infty$, as the distance function $r\rightarrow\infty$, or
	\item[A2)] Curvature decays quickly: $|\Riem|<(2-\epsilon)r^{-2}$, for any $\epsilon>0$, as $r\rightarrow\infty$
\end{itemize}
then its reduction $\Sigma^2$ is closed; see Proposition \ref{PropKillingDecay}.
Any manifold that happens to be ALE, ALF, ALG, ALH will certainly fit into category (A1), as such ends all have action fields that either grow linearly or else stabilize at a finite size.
However the common ALE-ALF-ALG-ALH schema (eg \cite{CK} \cite{Etesi}) is too confining, and our results extend beyond this.
This is true especially in the ALF case: many well-known ZSC toric K\"ahler manifolds have ends that are similar to ALF ends, but do not adhere to the usual ALF model; see Section \ref{SubSecInsufficiency}.
For this reason within category (A1) it is useful to create two subcategories: {\it asymptotically spheroidal} manifold ends, and {\it asymptotically toriodal} manifold ends.

{\bf Definition 1.}
A manifold end $M^4\setminus{}K$ is {\it asymptotically spheroidal} if it satsifies (A1),  the universal cover of $M^4\setminus{}K$ is diffeomorphic to $(R,\infty)\times\mathbb{S}^3$, and the lifts of the symmetry fields restrict to $\mathbb{S}^3$ where they produce a standard isometric torus action.

{\bf Definition 2.}
A manifold end $M^4\setminus{}K$ is {\it asymptotically toroidal} if it satisfies (A1), some cover of $M^4\setminus{}K$ is diffeomorphic to $(R,\infty)\times{}N^3$, and the lift of the end's symmetry fields restrict to $N^3$ where they produce an isometric free torus action on $N^3$.

We classify all scalar-flat toric K\"ahler manifolds with ends of type (A1) or (A2).
We show that, in the asymptotically toroidal cases including all ALG or ALH cases, such manifolds are flat.
ALE or ALF toric manifolds are clearly asymptotically spheroidal, and we classify these in the scalar flat case.

Finally we must say something about boundary conditions on $\Sigma^2$.
The boundary conditions come in an especially simple form: one positive real number for each edge, called its \textit{label}.
When $\Sigma^2$ is the reduction of some parent manifold, the labels are determined by $\mathcal{X}^1$ and $\mathcal{X}^2$.
We give two equivalent geometric explanations of the labels and how to compute them, the first in Section \ref{SubsecBoundaryConds} below and the second in Section \ref{SubsecOutlineMatching}.
In the scalar-flat case after labels are specified, we prove the metric is completely determined up to a 1- or 2-parameter family of possible variations.

If one does not impose \textit{any} restrictions on manifold ends, the reduction $(\Sigma^2,g_\Sigma)$ must still be convex, but we cannot say much else.
$\Sigma^2$ need not be closed or even be a polygon; see the examples in Section \ref{SecExamples}.
Some non-polygon examples are not pathological in the slightest, but some are very pathological.

\subsection{Statement of hypotheses}

We classify closed metric polygons $(\Sigma^2,g_\Sigma)$ that have a synthetic zero scalar curvature condition, whether $\Sigma^2$ comes from a larger manifold $M^4$ or not.
The classification is analytic, relying on a ``Liouville theorem'' for a certain degenerate-elliptic PDE.
This analysis takes place entirely on $(\Sigma^2,g_\Sigma)$ whether or not it is the K\"ahler reduction of some parent manifold, so it is useful to state our assumptions in terms of $(\Sigma^2,g_\Sigma)$ alone.
\begin{itemize}
	\item[A)] (Natural polygon condition)
	Within the $(\varphi^1,\varphi^2)$ coordinate plane $\Sigma^2$ is a convex closed polygon, which is not compact but has finitely many edges.
	\item[B)] (Boundary connectedness)
	The boundary $\partial\Sigma^2$ has one component or no components.
	\item[C)] (Polygon metric condition)
	The metric $g_\Sigma$ is positive definite, $C^\infty$ up to boundary segments, and Lipschitz at corners.
\end{itemize}
And the gradient fields $\nabla\varphi^1$, $\nabla\varphi^2$ satisfy
\begin{itemize}
	\item[D)] (Natural boundary conditions)
	The unit fields $\frac{\nabla\varphi^i}{|\nabla\varphi^i|}$ are smooth except at corner points, Lipschitz at corner points, and covariant-constant along boundary segments (in particular, boundary segments are totally geodesic).
	\item[E)] (Pseudo-toric conditions)
	The distribution $span\{\nabla\varphi^1,\nabla\varphi^1\}$ is 2-dimensional on the interior of $\Sigma^2$, $1$-dimensional on boundary segments, and $0$-dimensional on boundary vertices.
	Also $[\nabla\varphi^1,\nabla\varphi^2]=0$.
	\item[F)] (Pseudo-ZSC condition)
	Setting $\mathcal{V}\triangleq|\nabla\varphi^1|^2|\nabla\varphi^2|^2-\left<\nabla\varphi^1,\nabla\varphi^2\right>{}^2$, we have $\triangle_\Sigma\sqrt{\mathcal{V}}=0$.
\end{itemize}
Conditions (C), (D), and (E) are automatic when $(\Sigma^2,g_\Sigma)$ is the reduction of any smooth toric K\"ahler $M^4$.
They rule out pathologies that might exist on metric polygons but cannot exist on reductions of K\"ahler 4-manifolds.

Conditions (A) and (B) require a bit more discussion.
There do exist complete manifolds $(M^4,J,\omega,\mathcal{X}_1,\mathcal{X}_2)$ with associated polygons that violate them, although they are all unusual or pathological in one way or another.
If (A) is violated, the manifold might have infinite topological type, or it might have ends that are pathological or otherwise of unknown type.
Condition (B) is violated for one and only one manifold: $\mathbb{S}^2\times\mathbb{H}^2$ where one of the Killing fields gives a hyperbolic symmetry field on $\mathbb{H}^2$.
This is examined in \S\ref{SubsecSTimesH}.
This paper does not explore what the moduli of metrics could be in this case, mostly because our Proposition \ref{ThmIntoGlobalZ}---essential to our results---fails in this unique case.

Condition (F) is a synthetic scalar-flat condition, and is the point of the whole paper.
When $(\Sigma^2,g_\Sigma)$ in indeed the reduction of some $(M^4,J,g,\mathcal{X}^1,\mathcal{X}^2)$, condition (F) is equivalent to $(M^4,g)$ being scalar-flat.

\subsection{Statement of results}

First we state the only substantial result of this paper that does not use the ZSC-condition (F) but instead uses $\triangle_\Sigma\sqrt{\mathcal{V}}\le0$, a synthetic \textit{non-negative} scalar curvature condition.
\begin{theorem}[{\it cf.} Theorem \ref{ThmRTwoClassification}] \label{ThmIntroCompletePolygon}
	Assume $(\Sigma^2,g_\Sigma)$ is a geodesically complete 2-manifold in the $(\varphi^1,\varphi^2)$ plane that is convex but otherwise does not necessarily satisfy (A), does satisfy (B)-(E), and instead of (F) satisfies $\triangle_\Sigma\sqrt{\mathcal{V}}\le0$.
	Then the metric $g_\Sigma$ is flat.
\end{theorem}
\begin{corollary}[{\it cf.} Corollary \ref{CorRTwoClassification} and Proposition \ref{PropToroidalFlat}] \label{ThmIntroM4CompletePolygon}
	Assume $(M^4,J,g,\mathcal{X}^1,\mathcal{X}^2)$ is a complete toric K\"ahler manifold with scalar curvature $s\ge0$, and $span\{\mathcal{X}^1,\mathcal{X}^2\}$ has rank 2 at every point.
	Then $(M^4,g)$ is flat.
\end{corollary}
The proof of Theorem \ref{ThmIntroCompletePolygon} occupies Section \ref{SectionRTwoCase}, and uses methods different from the rest of the paper.
For a contrasting example---an instanton that obeys all hypotheses except $s\ge0$---see Example \ref{SubsubsecNonpolygon}.

All remaining results require Proposition \ref{ThmIntoGlobalZ}, which establishes a global holomorphic coordinate on $\Sigma^2$.
The synthetic ZSC condition $\triangle_\Sigma\sqrt{\mathcal{V}}=0$ allows us to create an analytic function $z:\Sigma^2\rightarrow\mathbb{C}$ first by setting $y=\sqrt{\mathcal{V}}$, then finding its harmonic conjugate by solving $dx=-J_\Sigma{}dy$ for $x$.
The resulting holomorphic function $z=x+\sqrt{-1}y$ we call the {\it volumetric normal function}, after the fact that $\mathcal{V}$ is a volume.
Proposition \ref{ThmIntoGlobalZ} asserts $z$ is unramified and maps $\Sigma^2$ bijectively onto the closed upper half-plane $\overline{H}{}^2\subset\mathbb{C}$.
See Section \ref{SecGlobalHarmonicCoordBehavior}.
\begin{proposition}[Volumetric normal coordinates, {\it cf.} Propositions \ref{ThmInjectivityOfZ} and \ref{ThmSurjectivityOfZ}] \label{ThmIntoGlobalZ}
	Assume $(\Sigma^2,g_\Sigma)$ satisfies (A)-(F), and let $z:\Sigma^2\rightarrow\overline{H}{}^2$ be the volumetric normal function.
	Then the analytic map $z:\Sigma^2\rightarrow\overline{H}{}^2$ is bijective.
\end{proposition}
This proposition requires condition (B).
If $\Sigma^2$ is closed but the boundary $\partial\Sigma^2$ is disconnected---this is the unique case that $\Sigma^2$ is an infinite strip---then $z$ is generically 2-to-1 and actually does have a ramification point.
See Example \ref{SubsecSTimesH}.
Now that we know $z$ is a {\it global} analytic coordinate $z:\Sigma^2\rightarrow\overline{H}{}^2$, we can study the global behavior of the $\varphi^1$ and $\varphi^2$ on the half-plane instead of on $\Sigma^2$ itself.
There, a result of Donaldson \cite{Do1} states the $\varphi^i$ satisfy the degenerate-elliptic equations
\begin{eqnarray}
	y\left(\varphi^1_{xx}+\varphi^1_{yy}\right)\,-\,\varphi^1_y\;=\;0
	\quad\text{and}\quad
	y\left(\varphi^2_{xx}+\varphi^2_{yy}\right)\,-\,\varphi^2_y\;=\;0.
	\label{EqnIntroDegEllipt}
\end{eqnarray}
Now the harmonic variables $x$, $y$ satisfy an elliptic system and the momentum variables $\varphi^1$, $\varphi^2$ satisfy a degenerate-elliptic system.
The interplay between these systems creates such strong controls that we can completely classify solutions.
\begin{theorem}[The Liouville theorem, Corollary 1.11 of \cite{Web2}]
	Assume $\varphi\in{}C^0(\overline{H}{}^2)\cap{}C^2(H^2)$ satisfies
	\begin{eqnarray}
		y(\varphi_{xx}+\varphi_{yy})\,-\,\varphi_y\;=\;0
	\end{eqnarray}
	in $H^2$, with boundary condition $\varphi(x,0)=0$ on $\{y=0\}$ and lower bound $\varphi\ge0$.
	
	Then $\varphi=C_1y^2$ for some constant $C_1\ge0$.
\end{theorem}
\noindent This theorem, quoted from the PDE literature, is slightly inadequate so in Section \ref{SubsecImprLiouville} we loosen the requirement that $\varphi\ge{}0$, and still achieve the Liouville theorem.
\begin{theorem}[The improved Liouville theorem, {\it cf.} Theorem \ref{ThmImprLiouville}] \label{ThmImprLiouIntro}
	Assume $\varphi\in{}C^0(\overline{H}{}^2)\cap{}C^2(H^2)$ satisfies
	\begin{eqnarray}
		y(\varphi_{xx}+\varphi_{yy})\,-\,\varphi_y\;=\;0 \label{EqnPDEInitialStatement}
	\end{eqnarray}
	in $H^2$, with the boundary condition $\varphi(x,0)=0$ on $\{y=0\}$.
	If $\varphi$ has lower bounds that are first order $x$ and sub-quadratic in $y$, specifically
	\begin{eqnarray}
		\varphi\;>\;-A\,-\,B|x|\,-\,Cy^\delta
	\end{eqnarray}
	for some $A,B,C\in\mathbb{R}^{\ge0}$, $\delta\in[0,2)$, then $\varphi\ge0$ and $\varphi=C_1y^2$ for some $C_1\ge0$.
\end{theorem}
Our classification results are corollaries of this Liouville theorem.
It is convenient to separate the classification into three cases; see Figure \ref{SubFigThreeCases}.

\begin{theorem}[Classification in the general case, {\it cf.} Theorem \ref{ThmGeneralCase}] \label{ThmIntroGeneralPoly}
	Let $(\Sigma^2,g_\Sigma)$ be a metric polygon satisfying (A)-(F), with $d$ many vertices.
	Assume $\Sigma^2$ has no parallel rays and is not the half-plane.
	Given $d+1$ labels $s_0,\dots,s_d$ on its edges, the metric $g_\Sigma$ is a member of a 2-parameter family of possible metrics.
	
	If labels are not specified, then $g_\Sigma$ is a member of a $(d+3)$-parameter family of possible metrics.
\end{theorem}
\begin{theorem}[Classification when $\Sigma^2$ has parallel rays, {\it cf.} Theorem \ref{ThmParallelRayCase}] \label{ThmIntroParallelPoly}
	Let $(\Sigma^2,g_\Sigma)$ be a metric polygon satisfying (A)-(F), with $d$ many vertices, and assume $\Sigma^2$ has parallel rays.
	Given $d+1$ labels $s_0,\dots,s_d$ on its edges, then the metric $g_\Sigma$ is a member of a 1-parameter family of possible metrics.
	
	If labels are not specified, then $g_\Sigma$ is a member of a $(d+2)$-parameter family of possible metrics.
\end{theorem}
\begin{theorem}[Classification when $\Sigma^2$ is the half-plane, {\it cf.} \ref{PropHalfPlaneFlat}] \label{ThmIntroHalfPlanePoly}
	Let $(\Sigma^2,g_\Sigma)$ be a metric polygon satisfying (A)-(F), and assume $\Sigma^2$ is the half-plane.
	Given a label on the bounding line, the metric $g_\Sigma$ is a member of a 3-parameter family of possible metrics.
	
	If no label is specified, $g_\Sigma$ is a member of a $4$-parameter family of possible metrics.
\end{theorem}
\noindent\begin{figure}[h]
	\centering
	\includegraphics[scale=0.5]{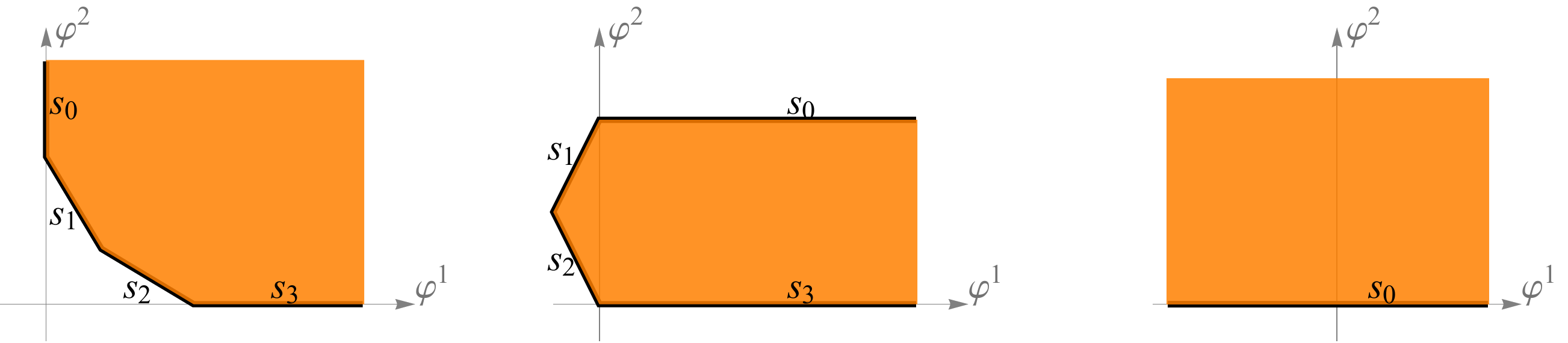}
	\caption{From left to right: the general case (Theorem \ref{ThmIntroGeneralPoly}), the case that $\Sigma^2$ has parallel rays (Theorem \ref{ThmIntroParallelPoly}), and the case that $\Sigma^2$ is the closed half-plane (Theorem \ref{ThmIntroHalfPlanePoly}).}
	\label{SubFigThreeCases}
\end{figure}

{\bf Remark}.
In Sections \ref{SubsecClassGen}, \ref{SubsecClassPar}, and \ref{SubsecClassHP} we produce the promised family of metrics $g_\Sigma$ quite explicitly.
Given a polygon with $d$ vertices and $d+1$ labels $s_0,\dots,s_d$ on its edges, the constructions of Section \ref{SubsecOutlineMatching} first produce a unique pair of comparison moment functions via a ``boundary-matching'' technique (as was done in \cite{AS}).
The values at the boundary matching up, the Liouville theorem states the true moment functions must differ from the comparison functions by at worst a term of the form $Cy^2$.
One such term for each moment variable produces either a 1- or 2-parameter variation for the pair.
From this, the metric $g_\Sigma$ is given by equation (\ref{EqnMetricConstruction}).

\begin{corollary}[Classification of one-ended, ZSC toric K\"ahler 4-manifolds; {\it cf.} Corollary \ref{CorOneEnded}] \label{CorIntroOneEnded}
	Assume $(M^4,J,g,\mathcal{X}^1,\mathcal{X}^2)$ is a scalar-flat toric K\"ahler manifold of finite topology that satisfies either of the asymptotic conditions (A1) or (A2), or otherwise has closed reduction $(\Sigma^2,g_\Sigma)$.
	
	Then $(\Sigma^2,g_\Sigma)$ is either an infinite closed strip or else satisfies conditions (A)-(F), and one of the following holds:
	\begin{itemize}
		\item[{\it{i}})] $(M^4,g)$ is flat,
		\item[{\it{ii}})] $\Sigma^2$ is an infinite closed strip in the $\varphi^1$-$\varphi^2$ plane,
		\item[{\it{iii}})] $(M^4,g)$ is the exceptional half-plane instanton,
		\item[{\it{iv}})] $\Sigma^2$ has parallel rays, and for given boundary values the metric belongs to a 1-parameter family of possibilities,
		\item[{\it{v}})] or $M^4$ is asymptotically spheroidal (never toroidal), and is
		\begin{itemize}
			\item[\textit{a})] Asymptotically locally Euclidean---and for any set of labels there is precisely one such metric---or
			\item[\textit{b})] Asymptotically spheriodal and asymptotically equivalent to a Taub-NUT, chiral Taub-NUT, or exceptional Taub-NUT; after labels are determined, the metric belongs to a 2-parameter family of possibilities.
		\end{itemize}
	\end{itemize}
\end{corollary}

\subsection{Interpretation of the boundary conditions} \label{SubsecBoundaryConds}

The ``boundary conditions'' or ``labels'' come in the form of a single positive number given to each segment or ray of a closed polygon.
Because the fields $\mathcal{X}^1$, $\mathcal{X}^2$ on $M^4$ need not be {\it standard} generators of a torus action, the Arnold-Liouville reduction might not produce the same results as the momentum construction of symplectic geometry.
The polygon need not be Delzant, and so the polygon itself does not suffice to reconstruct the parent manifold.
The labels fix this problem.
\noindent\begin{figure}[h!]
	\includegraphics[scale={0.4}]{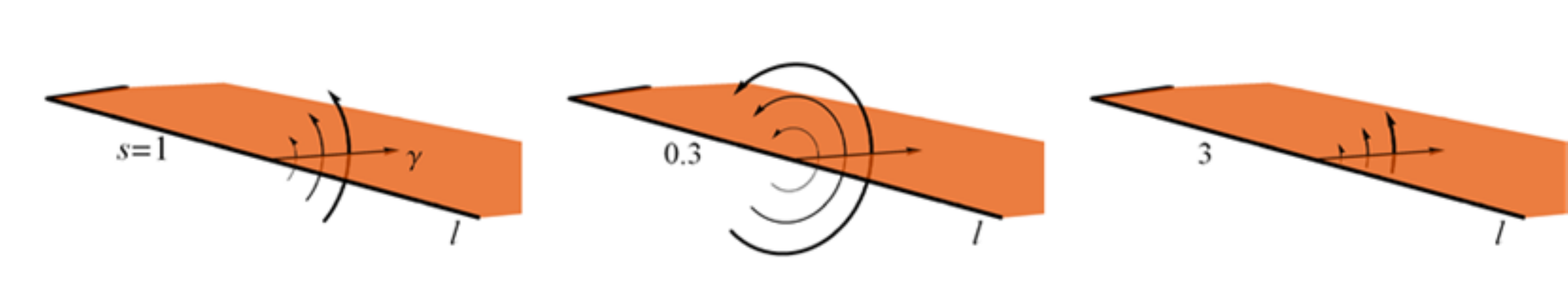}
	\caption{Geometric interpretation of the boundary conditions.}
	\label{FigEdgePar}
\end{figure}

To explain how they work, let $l_i$ be a boundary edge or ray with label $s_i$.
The edge $l_i\in\Sigma^2$ represents an embedded totally-geodesic submanifold $L_i^2\subset{}M^4$ on which $\{\mathcal{X}^1,\mathcal{X}^2\}$ spans a 1-dimensional instead of a 2-dimensional distribution.
Thus some linear combination $\mathcal{X}=c_1\mathcal{X}^1+c_2\mathcal{X}^2$ is a Killing field with zeros along $L_i^2$.
Let $\gamma(t)$ be a unit-speed geodesic perpendicular to this submanifold, and consider the vector field $\mathcal{X}(t)\triangleq\mathcal{X}_{\gamma(t)}$ along this path.
The label $s_i$ is precisely the quantity
\begin{eqnarray}
	s_i\;=\;\left(\left.{\frac{d}{dt}}^+\right|_{t=0}|\mathcal{X}(t)|\right)^{-1}.
	\label{EqnInterpretationOfMarkings}
\end{eqnarray}
Compare with Section \ref{SecAnalytic}, particularly equation (\ref{EqnMarkings}), where we give our second (equivalent but slightly more technical) interpretation of the labels.
The fact that $s_i$ is a constant along $l_i$ is simply the fact that $\mathcal{X}$ generates a circle action, and this circle must close off in the same way everywhere along the edge $l_i$.
Letting $\mathcal{X}$ be the Killing field that vanishes along $l_i$, after choosing a transversal the action of $\mathcal{X}$ creates a variable $\theta$ on $M^4$ which closes off at $\theta=\theta_i$, so $\theta\in[0,\theta_i)$.
Then the cone angle along $l_i$ is $\theta_i/s_i$---this might produce an orbifold or conifold along the edge $l_i$; see Figure \ref{FigOrbifolds}, and compare with Example \ref{SubsubsecOk} the discussion in Section \ref{SecAnalytic}.

When reconstructing $M^4$ from $\Sigma^2$, the fields $\mathcal{X}^1$, $\mathcal{X}^2$ can be considered images of generators of the lie algebra $\mathfrak{g}$ of an action torus $G$; choosing $G$, if it is not already known, produces coordinate ranges $[0,\theta_i)$ for the two angle variables.
Such choices might produce orbifold or conifold points along the edges.
One attempts to find a quotient of the torus by some discrete subgroup that simultaneously resolves all orbifolds.
If such a choice is possible, it is uniquely determined by the values $s_i$.
See Example \ref{SubsubsecOk} to see this done explicitly.

These labels are similar, but not identical, to the Lerman-Tolman labels \cite{LeTo} on symplectic orbifolds.
The difference is that our polygons need not be Delzant, and need not come from any specific torus action, so our labels have a somewhat different interpretation, for the reason that we use action fields rather than action tori.
Our labels are positive real numbers rather than positive integers, and can accommodate conifolds in addition to orbifolds and manifolds.

\subsection{The half-plane and quarter-plane metrics} \label{SubsecHPAndTaubNUT}

The half-plane and quarter-plane cases are exceptional because the polygons $\Sigma^2$ themselves are scale-invariant in the $\varphi^1$-$\varphi^2$ plane.
This reduces the number of degrees of freedom the metrics may take in these cases.
Given any polygon, we may, if we wish, perform a constant-coefficient linear recombination of the fields $\mathcal{X}^1$, $\mathcal{X}^2$ without changing anything essential.
This performs an affine mapping of the $\varphi^1$-$\varphi^2$ plane, alters how the polygon sits within the plane, and alters the labels.

Alternatively, one may create the volumetric coordinates $x$, $y$ first, and then make an affine recombination of $\varphi^1$, $\varphi^2$ without changing $x$ and $y$.
This creates a homothetic transformation of the metric $g_\Sigma$.
This is due to the formula (\ref{EqnMetricConstruction}) which expresses $g_\Sigma$ in terms of the transition between the two coordinate systems:
\begin{eqnarray}
	g_\Sigma\;=\;y^{-1}det(A)\left(dx\otimes{}dx\,+\,dy\otimes{}dy\right) \label{EqnGIntoExpression}
\end{eqnarray}
where $A$ is the coordinate transition matrix $A^i_j=\frac{\partial\varphi^i}{\partial{}x^j}$.
As in Figure \ref{FigTwoTransforms}, we can map any ``wedge'' to the first quadrant, and any half-plane to the upper half-plane.
\noindent\begin{figure}[h!]
	\includegraphics[scale={0.45}]{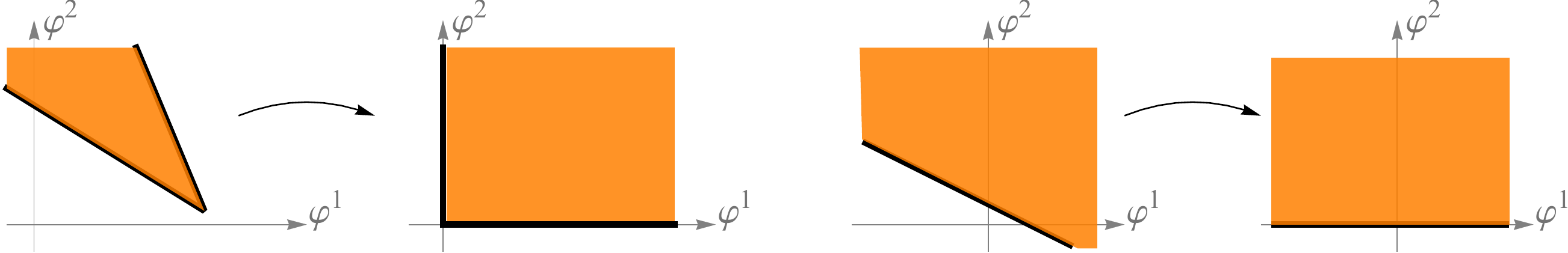}
	\caption{Affine transformations bringing wedges and half-planes into the quarter-plane and the upper half-plane.}
	\label{FigTwoTransforms}
\end{figure}

When $\Sigma^2$ is the half-plane, our classification states the momentum functions are
\begin{eqnarray}
	\varphi^1=C_1x+M_1xy^2+C_3y^2, \quad
	\varphi^2\;=\;C_2y^2.
\end{eqnarray} 
This is (\ref{EqnsHalfPlaneTransitions2}), and we see the specific 4-parameter family of variation the pair $\varphi^1,\varphi^2$ may take.
Leaving $x$, $y$ unchanged and making the linear transformation
\begin{eqnarray}
	\left(\begin{array}{c}
		\varphi^1 \\ \varphi^2
	\end{array}\right)
	\;\longmapsto\;
	\left(\begin{array}{cc}
		1/C_1 & -C_3/C_1C_2 \\
		0 & 1/2C_2
	\end{array}
	\right)
	\left(\begin{array}{c}
		\varphi^1 \\ \varphi^2
	\end{array}\right) \label{EqnHPCoordChange}
\end{eqnarray}
we obtain new functions $\varphi^1=x+Mxy^2$, $\varphi^2=\frac12y^2$ where $M=M_1/2C_1$, and only a single parameter $M$ remains.
However $M$ another homothetic parameter, as can be seen by making the transformation $x\mapsto{}x/M$ and $y\mapsto{}y/\sqrt{M}$.
Now the moment functions are $\varphi^1=\frac{1}{M}\left(x+xy^2\right)$ and $\varphi^2=\frac{1}{2M}y^2$, so multiplying these functions by $M$ while fixing $x$ and $y$---a homothetic transformation---produces the functions
\begin{eqnarray}
	\varphi^1=x+xy^2, \quad
	\varphi^2=\frac12y^2 \label{EqnHPStandard}
\end{eqnarray}
and therefore (\ref{EqnHPStandard}) is, up to homothety, the {\it only} possibility for moment variables when the polygon is a half-plane.

Next we consider polygons with one vertex.
After a possible affine transformation we can assume the polygon is the quarter-plane (see Figure \ref{FigTwoTransforms}).
Let $s_0,s_1>0$ be its two boundary labels.
The classification from Section \ref{SubsecClassGen} shows the $\varphi^1$-$\varphi^2$ pair must be among the 2-parameter family of variations
\begin{eqnarray}
	\begin{aligned}
	\varphi^1&=\frac{s_0}{2}\left(-x+\sqrt{x^2+y^2}\right)+C_1y^2 \\
	\varphi^2&=\frac{s_1}{2}\left(x+\sqrt{x^2+y^2}\right)+C_2y^2
	\end{aligned} \label{Eqn4ParameterQuarterPlane}
\end{eqnarray}
where $C_1,C_2\ge0$ are arbitrary constants.
The transformation $\varphi^1\mapsto\frac{\sqrt{2}}{s_0}\varphi^1$, $\varphi^2\mapsto\frac{\sqrt{2}}{s_1}\varphi^2$ produces moment functions
\begin{eqnarray}
	\begin{aligned}
	\varphi^1&=\frac{1}{\sqrt{2}}\left(-x+\sqrt{x^2+y^2}\right)
	+\frac{M(1+k)}{\sqrt{2}}y^2 \\
	\varphi^2&=\frac{1}{\sqrt{2}}\left(x+\sqrt{x^2+y^2}\right)
	+\frac{M(1-k)}{\sqrt{2}}y^2
	\end{aligned} \label{Eqn2ParameterQuarterPlane}
\end{eqnarray}
where the constants $M=\frac14\left(\frac{C_1}{s_0}+\frac{C_2}{s_1}\right)$, $k=(\frac{C_1}{s_0}-\frac{C_2}{s_1})/(\frac{C_1}{s_0}+\frac{C_2}{s_1})$ obey $M\ge0$ and $k\in[-1,1]$.
We have reduced the 4-parameter family of (\ref{Eqn4ParameterQuarterPlane}) to the 2-parameter family of (\ref{Eqn2ParameterQuarterPlane}).
But we can reduce this even further, as again $M$ is a homothetic parameter.
To see this, make the transformation $x\mapsto{}x/M$, $y\mapsto{}y/M$ and then multiply through by $M$ to obtain the moment functions
\begin{eqnarray}
	\begin{aligned}
	\varphi^1&=\frac{1}{\sqrt{2}}\left(-x+\sqrt{x^2+y^2}\right)
	+\frac{1}{\sqrt{2}}(1-k)y^2 \\
	\varphi^2&=\frac{1}{\sqrt{2}}\left(x+\sqrt{x^2+y^2}\right)
	+\frac{1}{\sqrt{2}}(1+k)y^2
	\end{aligned} \label{EqnMomTaubNut}
\end{eqnarray}
and we have reduced the 4 degrees of freedom to just one, that of a single constant $k$.
One might wonder if $k\in[-1,1]$ is also a homothetic variable, or if it is something different.
Indeed it is different, as formulas (\ref{EqnMetricConstruction}) and (\ref{EqnCompKSigma}) give
\begin{eqnarray}
	\begin{aligned}
	&g_\Sigma\;=\;\frac{1+2\left(kx+\sqrt{x^2+y^2}\right)}{\sqrt{x^2+y^2}}\big(dx\otimes{}dx+dy\otimes{}dy\big), \\
	&K_\Sigma
	\;=\;\frac{-1+2k\left(x+k\sqrt{x^2+y^2} \right)}{\left(1+2\left(kx+\sqrt{x^2+y^2}\right)\right)^3} \label{EqnFirstGKtaubnut}
	\end{aligned}
\end{eqnarray}
where $K_\Sigma$ is the polygon's Guassian curvature.
Now it is clear that changing $k$ makes qualitative changes to the metric.
For instance when $k=0$ we see cubic curvature falloff (at least as measured in coordinates) whereas when $k=-1$ we see a ray along which there is no curvature falloff whatsoever (the ray along the positive $x$-axis).
The number $k\in[-1,1]$ is called the instanton's {\it chirality number}, and there are three critical values: when $k=0$ the metric is the standard Taub-NUT metric, which is Ricci-flat.
When $k=1$ or $k=-1$ the metric is called the {\it exceptional Taub-NUT} metric---these two cases are isometric by exchanging the $\varphi^1$ and $\varphi^2$ variables.
In all other cases where the chirality number is in $(-1,0)\cup(0,1)$, we have the {\it generalized Taub-NUT} metrics of Donaldson's \cite{Do2}.
These are scalar-flat metrics with cubic volume growth and quadratic curvature decay.
They are not Ricci-flat.
The generalized Taub-NUTs, the exceptional Taub-NUT, and the exceptional half-plane instantons were studied in detail in \cite{Web3}.

These considerations establish the following corollary.
\begin{corollary}
	Assume $(M^4,J,g,\mathcal{X}^1,\mathcal{X}^2)$ is a scalar-flat toric 4-manifold satisfying either asymptotic condition (A1) or (A2), or otherwise has closed reduction $\Sigma^2$.
	Assume either that $\mathcal{X}^1$ and $\mathcal{X}^2$ have a single common zero (meaning its polygon has a single vertex), or no common zeros.
	Then, up to affine transformation of $\mathcal{X}^1$, $\mathcal{X}^2$ and a scale factor, $M^4$ is either
	\begin{itemize}
		\item[{\it{i}})] flat,
		\item[{\it{ii}})] the Taub-NUT instanton,
		\item[{\it{iii}})] one of the generalized Taub-NUT instantons,
		\item[{\it{iv}})] the exceptional Taub-NUT instanton,
		\item[{\it{v}})] or is the exceptional half-plane instanton.
	\end{itemize}
\end{corollary}

\subsection{Organization}

In Section \ref{SectionReduction} we recall the Arnold-Liouville reduction and the basics of K\"ahler toric 4-manifolds---this is mostly standard theory but there is new material in Section \ref{SubSecConfChangeOfMetric}.
Section \ref{SectionRTwoCase} proves Theorem \ref{ThmIntroCompletePolygon}.

Sections \ref{SecGlobalHarmonicCoordBehavior} and \ref{SecAnalytic} are the heart of the paper.
Section \ref{SecGlobalHarmonicCoordBehavior} proves the volumetric normal function $z$ is an unramified, global analytic coordinate.
In Section \ref{SecAnalytic} we use the Liouville theorem to find that, after a boundary-matching procedure, each momentum function is determined up to (at most) a one-parameter variation.

In Section \ref{SecAsymptotics} we consider the asymptotic conditions (A1) and (A2), and prove that metrics which meet either criteria are captured within the classification of this paper.
We show asymptotically toroidal metrics are flat.
In \S\ref{SubSecInsufficiency} we show why the ALE-ALF-ALG-ALH schema is inadequate for scalar-flat instantons.

Section \ref{SecExamples} details some pathological and non-pathological examples that illustrate the results of this paper.
We show the details of the ZSC LeBrun metrics \cite{LebOK} on the total spaces of the $O(-k)$ bundles over $\mathbb{P}^1$, and compute the labels explicitly.

{\bf Remark.} To expand this work to cases where $\Sigma^2$ need not be closed, one would have to contend with the pathologies of \ref{SubsecPathExamples} and with non-polygons.
We conclude with two conjectures.

{\bf Conjecture 1.}
If $(M^4,J,g,\mathcal{X}^1,\mathcal{X}^2)$ is scalar-flat, then $\Sigma^2$ includes its two terminal rays, and its closure in the $(\varphi^1,\varphi^2)$-plane is a polygon.

{\bf Conjecture 2.}
Assume $(M^4,J,g,\mathcal{X}^1,\mathcal{X}^2)$ is scalar-flat, and assume its reduction $\Sigma^2$ is not the strip.
After assigning boundary values, including ``$\infty$'' to missing segments, then $g$ is defined up to at most a 2-parameter family of possibilities.

\section{The K\"ahler Reduction} \label{SectionReduction}

The relationship between a toric K\"ahler instanton $(M^4,J,\omega,\mathcal{X}{}^1,\mathcal{X}{}^2)$ and its metric reduction $(\Sigma^2,g_\Sigma)$ has been developed by a number of authors; see for instance \cite{G} \cite{Ab1} \cite{Do1} \cite{Do2} \cite{AS} \cite{CDG} and references therein.
This section is for setting notation and reader convenience, as we will cite this material frequently.
Only Section \ref{SubSecConfChangeOfMetric} contains anything new.

We indicate this section's milestones.
In \S\ref{SecFundamentals} we perform the Arnold-Liouville reduction \cite{Ar}, and relate the symplectic coordinates $(\varphi^1,\varphi^2,\theta_1,\theta_2)$ to the instanton's complex-analytic coordinates $(z_1,z_2)$.
The holomorphic volume form in $(z_1,z_2)$ coordinates form is a multiple of the parallelotope volume, which is
\begin{eqnarray}
	\mathcal{V}
	\;\triangleq\;|\mathcal{X}{}^1|^2|\mathcal{X}{}^2|^2
	-\left<\mathcal{X}{}^1,\mathcal{X}{}^2\right>^2
	\;=\;|\nabla\varphi^1|^2|\nabla\varphi^2|^2
	-\left<\nabla\varphi^1,\nabla\varphi^2\right>^2. \label{EqnDefoOfV}
\end{eqnarray}
In \S\ref{SubSectionReductionToPolytope} we show in the inherited metric $g_\Sigma$ that $\triangle_\Sigma\mathcal{V}^{\frac12}+\frac12s\mathcal{V}^{\frac12}=0$ (this is the Trudinger-Wang reduction \cite{TW} and can also be considered a version of the Abreu equation, equation (10) of \cite{Ab1}).
When $M^4$ is scalar-flat, $\sqrt{\mathcal{V}}$ is harmonic on $(\Sigma^2,g_\Sigma)$, so we create the volumetric normal coordinates by setting $y=\sqrt{\mathcal{V}}$ and letting $x$ be its harmonic conjugate.
In \S\ref{SubSectionReductionToPolytope} we create degenerate-elliptic equations on $\Sigma^2$ for the moment functions $\varphi^1$ and $\varphi^2$.
This completes the loop: the harmonic functions $x$, $y$ satisfy an elliptic condition in $\varphi^1$ and $\varphi^2$, and the momentum functions $\varphi^1$, $\varphi^2$ satisfy a degenerate-elliptic equation in $x$ and $y$.

In \S\ref{SubSectionReconstructionOfTheMetric} we show how to reconstruct the metrics $g_\Sigma$ and $g$ simply by knowing $\varphi^1$ and $\varphi^2$ as functions of $x$ and $y$.
We compute the Gaussian curvature $K_\Sigma$ of $(\Sigma^2,g_\Sigma)$.
Lastly in \S\ref{SubSecConfChangeOfMetric} we calculate $K_\Sigma$ a second way, and perform a conformal change of metric to $\widetilde{g}{}_\Sigma=\mathcal{V}^{\frac12}g_\Sigma$ and show the Gaussian curvature of $\tilde{g}_\Sigma$ is non-negative.

\subsection{Fundamentals} \label{SecFundamentals}

We have a simply connected K\"ahler 4-manifold $(M^4,J,\omega)$ with commuting symplectomorphic Killing fields $\mathcal{X}^1$ and $\mathcal{X}^2$.
Because $\mathcal{L}_{\mathcal{X}^i}\omega=0$ we have $di_{\mathcal{X}^i}\omega=0$, so there are functions $\varphi^i$ with $d\varphi^i=\omega(\mathcal{X}^i,\cdot)$, traditionally called {\it momentum variables} or {\it action variables}.
To complete the coordinate system, two additional functions $\theta_1$, $\theta_2$, called {\it cyclic variables} or {\it angle variables}, are defined by taking a transversal to the $\{\mathcal{X}^1,\mathcal{X}^2\}$ distribution and then pushing the natural $\mathbb{R}^2$ variables forward along the action.
The values of $\theta_1$, $\theta_2$ are not canonical due to the choice of a transversal, but their fields $\frac{\partial}{\partial\theta_1}$, $\frac{\partial}{\partial\theta_2}$ are canonical, and are equal to $\mathcal{X}^1$, $\mathcal{X}^2$, respectively.
This construction yields the {\it action-angle} system $\{\varphi^1,\varphi^2,\theta_1,\theta_2\}$; see \cite{Ar}.
The coordinate fields are $\frac{\partial}{\partial\varphi^i}=G_{ij}\nabla\varphi^i$ and $\frac{\partial}{\partial\theta_i}=\mathcal{X}^i$, where $G^{ij}=\left<\nabla\varphi^i,\nabla\varphi^j\right>=\big<\frac{\partial}{\partial\theta_i},\frac{\partial}{\partial\theta_j}\big>$.
In the ordered frame $ \frac{\partial}{\partial\varphi^1},\frac{\partial}{\partial\varphi^2},\frac{\partial}{\partial\theta_1},\frac{\partial}{\partial\theta_2}$ we have metric, complex structure, and symplectic form
\begin{eqnarray}
\quad
	g\;=\;\left(\begin{array}{c|c}
		G_{ij} & 0 \\
		\hline
		0 & G^{ij}
		\end{array}\right), 
	\;
	J\;=\;\left(\begin{array}{c|c}
		0 & -G^{ij} \\
		\hline
		G_{ij} & 0
	\end{array}\right), 
	\;
	\omega\;=\;\left(\begin{array}{c|c}
		0 & -Id \\
		\hline
		Id\; & 0
	\end{array}\right). \label{EqnsGJOmegaM}
\end{eqnarray}

{\bf Remark}.
Compare (\ref{EqnsGJOmegaM}) to (4.8) of \cite{G} or (2.2), (2.3) of \cite{Ab2}.
A common construction shows $G_{ij}=\frac{\partial^2{\bf{u}}}{\partial\varphi^i\partial\varphi^j}$ for a function ${\bf{u}}$ called the symplectic potential, but this is less useful at present.
In our formulation, it may be objected that expressing the metric in terms of inner products is redundant.
But the specific forms which $g$, $J$, and $\omega$ take will be important below.

\begin{lemma}[Symplectic and holomorphic coordinates on $M^4$] \label{LemmaHoloSymplCoordRelation}
	The angle coordinates are pluriharmonic, meaning $dJd\theta_i=0$.
	These lead to local holomorphic coordinates $(z_1,z_2):M^4\rightarrow\mathbb{C}^2$ of the form
	\begin{eqnarray*}
		\begin{aligned}
			z_1\;=\;f_1(\varphi^1,\,\varphi^2)\,+\,\sqrt{-1}\,\theta_1, \quad
			z_2\;=\;f_2(\varphi^1,\,\varphi^2)\,+\,\sqrt{-1}\,\theta_2
		\end{aligned}
	\end{eqnarray*}
	where $f_1$, $f_2$ are functions with $df_i=Jd\theta_i$.
\end{lemma}
{\bf Remark}. Although we shall not need this fact, from \cite{G} the transition between $(\varphi^1,\varphi^2)$ and $(f_1,f_2)$ is the Legendre transform on $\Sigma^2$ by the convex function ${\bf{u}}$.
Indeed $f_i=\frac{\partial{\bf{u}}}{\partial\varphi^i}$ where ${\bf{u}}$ is the symplectic potential mentioned above.
\begin{proof}
The fields $\mathcal{X}^i$ preserve $g$, $\omega$, and $J$, and $L_{\mathcal{X}^i}J=0$ is equivalent to $dJd\theta_i=0$.
Using $\bar\partial=\frac12\left(d+\sqrt{-1}Jd\right)$, one directly verifies that $\bar\partial{z}_i=0$.
\end{proof}

One easily determines the holomorphic frame and coframe to be
\begin{eqnarray}
	\begin{aligned}
	\frac{\partial}{\partial{z}_1}
	&\;=\;\frac12\left(\nabla\varphi^1-\sqrt{-1}\,\mathcal{X}^1\right), \quad
	\frac{\partial}{\partial{z}_2}
	\;=\;\frac12\left(\nabla\varphi^2-\sqrt{-1}\,\mathcal{X}^2\right) \\
	dz_1&\;=\;Jd\theta_1 +\sqrt{-1}\,d\theta_1, \quad\quad\quad
	dz_2\;=\;Jd\theta_2+\sqrt{-1}\,d\theta_2
\end{aligned} \label{EqnCxCoords}
\end{eqnarray}
so in the holomorphic frame the Hermitian metric and volume element are
\begin{eqnarray}
	\begin{aligned}
	&h^{i\bar\jmath}\;=\;
	\frac12
	\left(\begin{array}{cc}
	|\mathcal{X}_1|^2 & \left<\mathcal{X}_1,\,\mathcal{X}_2\right> \\
	\left<\mathcal{X}_1,\,\mathcal{X}_2\right> & |\mathcal{X}_2|^2
	\end{array}\right)
	\;=\;\frac12G^{-1}, \quad\quad
	\det\,h^{i\bar\jmath}\;=\;\frac14\mathcal{V}. \label{EqnHermitianM}
	\end{aligned}
\end{eqnarray}
where we have used the convention $h^{i\bar\jmath}=h\left(\frac{\partial}{\partial{z}_i},\frac{\partial}{\partial\overline{z_j}}\right)$.
\begin{proposition} \label{PropRicciFormAndScalarOnM}
	The Ricci form and scalar curvature of $(M^4,J,\omega)$ are
	\begin{eqnarray}
	\begin{aligned}
	\rho\;=\;-\sqrt{-1}\partial\bar\partial\log\,\mathcal{V},
	\quad\quad
	s\;=\;-\triangle\log\,\mathcal{V}.
	\end{aligned} \label{EqnsRicScal}
	\end{eqnarray}
\end{proposition}
\begin{proof}
	Using (\ref{EqnHermitianM}), these are textbook formulas.
\end{proof}

\subsection{Reduction of $M^4$ to its metric polygon} \label{SubSectionReductionToPolytope}

The map $\Phi:M^4\rightarrow\mathbb{R}^2$ given in coordinates by $\Phi(\varphi^1,\varphi^2,\theta_1,\theta_2)=(\varphi^1,\varphi^2)$ is called the \textit{Arnold-Liouville reduction}, or sometimes the \textit{moment map} (although this abuses the term); the image of $\Phi$ is called $\Sigma^2$.
Supposing the image of $\Phi$ is topologically closed---this is always true if $M^4$ is compact but not always true if $(M^4,g)$ is complete, see Example \ref{SubsecSTimesH}---then this image is known to be a polygon \cite{Delzant}.
Because the action of $\mathcal{X}^1$, $\mathcal{X}^2$ is isometric, the metric passes down along Arnold-Liouville map and produces a metric $g_\Sigma$ on $\Sigma^2$.
Indeed just $g_\Sigma(\frac{\partial}{\partial\varphi^i},\frac{\partial}{\partial\varphi^i})=g(\frac{\partial}{\partial\varphi^i},\frac{\partial}{\partial\varphi^i})$ so that $g_{\Sigma,ij}=g_{ij}=G_{ij}$ as in (\ref{EqnsGJOmegaM}).

For clarity, objects on $\Sigma^2$ will be indicated with a subscript, so for instance $s_{\Sigma}$ and $s$ indicate respectively the scalar curvatures on $(\Sigma^2,g_{\Sigma})$ and $(M^4,J,g)$.
We note that $(\Sigma^2,g_\Sigma)$ has a natural complex structure---this is {\it not} inherited from $(M^4,J)$ but rather is the (dual of the) Hodge-$*$ of $(\Sigma^2,g_\Sigma)$.
One computes that
\begin{equation}
	J_\Sigma\;=\;\mathcal{V}^{-\frac12}\left(\begin{array}{cc}
	\left<\mathcal{X}{}^1,\,\mathcal{X}{}^2\right> & -|\mathcal{X}{}^1|^2 \\
	|\mathcal{X}{}^2|^2 & -\left<\mathcal{X}{}^1,\,\mathcal{X}{}^2\right>
	\end{array}\right). \label{EqnJSigma}
\end{equation}
\begin{proposition}[$\Sigma^2$ and $M^4$ Laplacians] \label{PropPojectionLaplacian}
	If $f:M^4\rightarrow\mathbb{C}$ is any $\mathcal{X}^1$-$\mathcal{X}^2$ invariant function, then $\triangle{f}$ is $\mathcal{X}^1$-$\mathcal{X}^2$ invariant, so $f$ and $\triangle{f}$ pass to functions on the leaf space $\Sigma^2$.
	On $(\Sigma^2,g_\Sigma)$ the function $\triangle{}f$ and the Laplacian $\triangle_\Sigma{}f$ are related by
	\begin{eqnarray}
	\triangle{}f
	&=&\triangle_\Sigma{f}\,+\,\left<\nabla_{\Sigma}\log{\mathcal{V}}^{\frac12},\,\nabla_{\Sigma}f\right>_\Sigma.
	\end{eqnarray}
\end{proposition}
{\it Proof}. A function $f:M^4\rightarrow\mathbb{C}$ is invariant under $\mathcal{X}_1$, $\mathcal{X}_2$ if and only if it is a function of $\varphi^1$, $\varphi^2$ only.
Noting that $\det(g)=1$ and $\det(g_\Sigma)=\mathcal{V}^{-1}$ we have
\begin{eqnarray}
	\begin{aligned}
	\triangle{f}
	&\;=\;\frac{\partial}{\partial\varphi^i}\left(g^{ij}\frac{\partial{f}}{\partial\varphi^j}\right)
	\;=\;\frac{\partial}{\partial\varphi^i}\left(g_\Sigma^{ij}\frac{\partial{f}}{\partial\varphi^j}\right) \\
	&\;=\;\mathcal{V}^{\frac12}\frac{\partial}{\partial\varphi^i}\left(g_\Sigma^{ij}\mathcal{V}^{-\frac12}\frac{\partial{f}}{\partial\varphi^j}\right)
	\;+\;\,g_\Sigma^{ij}\frac{\partial\log{\mathcal{V}}^{\frac12}}{\partial\varphi^i}\frac{\partial{f}}{\partial\varphi^j} \\
	&\;=\;\triangle_\Sigma{f}\,+\,\frac12\left<\nabla_\Sigma\log\mathcal{V},\,\nabla_\Sigma{f}\right>_\Sigma.
	\end{aligned}
\end{eqnarray}
\qed

\begin{corollary} \label{CorReducedScalar}
	The scalar curvature $s$ on $M^4$ passes to a function on $\Sigma^2$. There,
	\begin{eqnarray}
	\triangle_\Sigma\mathcal{V}^{\frac12}\,+\,\frac12s\,\mathcal{V}^{\frac12}\;=\;0. \label{EllipticEqnForV}
	\end{eqnarray}
\end{corollary}
\begin{proof}
	This follows from $s=-\triangle\log\mathcal{V}$ and Proposition \ref{PropPojectionLaplacian}.
\end{proof}

{\bf Remark}. In the scalar-flat case Corollary \ref{CorReducedScalar} is the Trudinger-Wang reduction, equation (1.4) of \cite{TW}, from their study of fourth-order nonlinear elliptic equations.

\begin{proposition}[The $\varphi^i$ Elliptic Equations] \label{PropDivOfVarphi}
	On $\Sigma^2$, $d\left(\mathcal{V}^{-\frac12}J_\Sigma{d}\varphi^i\right)=0$.
\end{proposition} 
\begin{proof}
	Combining (\ref{EqnsGJOmegaM}) and (\ref{EqnJSigma}), from $df_i=Jd\theta_i$ we compute $df_i=\mathcal{V}^{-\frac12}J_\Sigma{}d\varphi^i$, where the $f_i$ are the functions from Lemma \ref{LemmaHoloSymplCoordRelation}.
	We conclude that $0=d(\mathcal{V}^{-\frac12}J_\Sigma{}d\varphi^i)$.
	The projection along the Arnold-Liouville map $\Phi$ preserves closure, so we retain $0=d(\mathcal{V}^{-\frac12}J_\Sigma{}d\varphi^i)$ on $\Sigma^2$.
\end{proof}

Wherever the functions $x$ and $y$ form non-singular coordinate, which is to say where $z=x+\sqrt{-1}y$ is unramified and single-valued, the equation $d(\mathcal{V}^{-\frac12}J_\Sigma{}d\varphi)=0$ is precisely the degenerate-elliptic equation $y(\varphi_{xx}+\varphi_{yy})-\varphi_y=0$.
This follows easily from the fact that $\mathcal{V}^{-\frac12}=y^{-1}$.

The following simple theorem illustrates a contrast between the compact case, where $s\le0$ is forbidden, and the open case, where $s\le0$ is mandatory (at one point at least; see Theorem \ref{ThmIntroCompletePolygon}).
\begin{corollary} \label{CorCompactFlatPolytope}
	If $(\Sigma^2,g_\Sigma)$ is compact, it is impossible that $\triangle_\Sigma\sqrt{\mathcal{V}}\ge0$.
\end{corollary}
{\it Proof}.
	With $\mathcal{V}=0$ on $\partial\Sigma^2$ the maximum principle gives $\mathcal{V}\equiv0$.
	But $\mathcal{V}\equiv0$ means $\mathcal{X}_1$, $\mathcal{X}_2$ are co-linear throughout, an impossibility.
\qed

\begin{corollary} \label{CorCompactFlatFourManifold}
	If $(M^4,J,\omega,\mathcal{X}{}^1,\mathcal{X}{}^2)$ is compact, it is impossible that $s\le0$.
\end{corollary}
{\it Proof}.
	If $s\le0$ then $\triangle_\Sigma\sqrt{V}\ge0$, so the previous corollary applies.
\qed

\subsection{Reconstruction of the $M^4$ metric, assuming $s=0$} \label{SubSectionReconstructionOfTheMetric}

We show how to reconstruct the metrics on $(M^4,J,\omega)$ and $(\Sigma^2,g_\Sigma)$ from the relationship between the momentum and the volumetric normal coordinates on $\Sigma^2$.
When $s=0$ Corollary \ref{CorReducedScalar} states $\triangle_\Sigma{}\sqrt{\mathcal{V}}=0$.
Defining $y=\sqrt{\mathcal{V}}$ and letting $x$ be an harmonic dual (meaning $dx=J_\Sigma{}dy$), the pair $(x,y)$ are isothermal coordinates on $\Sigma^2$, called the {\it volumetric normal coordinates}.
The formula (\ref{EqnJSigma}) for $J_\Sigma$ is
\begin{eqnarray}
	\begin{aligned}
	&\frac{1}{y}\left<\nabla\varphi^i,\,\nabla\varphi^j\right>\,d\varphi^1\wedge{}d\varphi^2\;=\;-d\varphi^i\wedge{}J_\Sigma{}d\varphi^j.
	\end{aligned} \label{EqnFirstInnerProdEqn}
\end{eqnarray}
Changing variables to $(x,y)$, we have
\begin{eqnarray}
	\begin{aligned}
	d\varphi^1\wedge{}d\varphi^2
	&\;=\;\left(
	\frac{d\varphi^1}{dx}\frac{d\varphi^2}{dy}\,-\,\frac{d\varphi^2}{dx}\frac{d\varphi^1}{dy}
	\right)\,dx\wedge{}dy, \\
	-d\varphi^i\wedge{}J_{\Sigma}d\varphi^j
	&\;=\;\left(\frac{d\varphi^i}{dx}\frac{d\varphi^j}{dx}\,+\,
	\frac{d\varphi^i}{dy}\frac{d\varphi^j}{dy}\right)dx\wedge{}dy,
	\end{aligned}
\end{eqnarray}
so from (\ref{EqnFirstInnerProdEqn}) we have
\begin{eqnarray}
	g_{\Sigma}^{ij}&=&\left<\nabla\varphi^i,\,\nabla\varphi^j\right>
	\;=\;
	y\,\cdot\,\frac{\frac{d\varphi^i}{dx}\frac{d\varphi^j}{dx}+
	\frac{d\varphi^i}{dy}\frac{d\varphi^j}{dy}}
	{\frac{d\varphi^1}{dx}\frac{d\varphi^2}{dy}-\frac{d\varphi^2}{dx}\frac{d\varphi^1}{dy}}. \label{EqnFirstSigmaMetricExpresssion}
\end{eqnarray}
Letting
$A=
\left(\frac{\partial\{\varphi^1,\varphi^2\}}{\partial\{x,y\}}\right)$,
$B=
\left(\frac{\partial\{x,y\}}{\partial\{\varphi^1,\varphi^2\}}\right)$ be the coordinate transitions, this is
\begin{eqnarray}
	\left(g_{\Sigma}{}^{ij}\right)\;=\;\frac{y}{\det(A)}\,A^T\,A, \quad\quad
	\left(g_{\Sigma,ij}\right)\;=\;\frac{y^{-1}}{\det(B)}\,B\,B^T.
\end{eqnarray}
In components,
\begin{eqnarray}
	\begin{aligned}
	\;\;\;\;
	g_\Sigma
	=\frac{1}{y\det(B)}\delta_{kl}B^k_iB^l_j\,d\varphi^i\otimes{}d\varphi^j
	=y^{-1}\det(A)\left(dx\otimes{}dx+dy\otimes{}dy\right).
	\end{aligned} \label{EqnMetricConstruction}
\end{eqnarray}
In particular the coordinate transition matrices fully determine $g$ and $g_\Sigma$.
A simple expression for the Gaussian curvature of $(\Sigma^2,g_\Sigma)$ is
\begin{eqnarray}
	\begin{aligned}
	K_\Sigma
	&\;=\;-\frac{y}{\det(A)}\left(\left(\frac{\partial}{\partial{x}}\right)^2+\left(\frac{\partial}{\partial{y}}\right)^2\right)\log\sqrt{y^{-1}\det(A)}.
	\end{aligned} \label{EqnCompKSigma}
\end{eqnarray}

\subsection{A conformal change of the metric} \label{SubSecConfChangeOfMetric}

Here we make a second computation of the curvature of $(\Sigma^2,g_\Sigma)$, and show the effect of a certain conformal change of the metric on the Gaussian curvature.

The metric $g_\Sigma$ has the ``pseudo-K\"ahler'' property
\begin{eqnarray}
\frac{\partial{g_\Sigma}{}_{ij}}{\partial\varphi^k}\;=\;\frac{\partial{g}_\Sigma{}_{ik}}{\partial\varphi^j}. \label{EqnPseudoKahler}
\end{eqnarray}
This was noted in \cite{AS}; it is equivalent to both $[\nabla\varphi^i,\nabla\varphi^j]=0$ and $d(\mathcal{V}^{-\frac12}J_\Sigma{d}\varphi^i)=0$.
The Christoffel symbols are
\begin{eqnarray}
	\begin{aligned}
	\Gamma_\Sigma{}_{ij}^k\;=\;\frac12\frac{\partial{g}_\Sigma{}_{ij}}{\partial\varphi^s}g_\Sigma{}^{sk}, \quad
	\Gamma_\Sigma{}^k
	\;\triangleq\;g_\Sigma{}^{ij}\Gamma_\Sigma{}_{ij}^k
	\;=\;-g_\Sigma{}^{ks}\frac{\partial}{\partial\varphi^s}\log\mathcal{V}^{\frac12}.
	\end{aligned} \label{EqnChristoSymbs}
\end{eqnarray}
The usual formula for scalar curvature in terms of Christoffel symbols is
\begin{eqnarray}
s_{\Sigma}=
g_\Sigma{}^{ij}\frac{\partial\Gamma_\Sigma{}_{ij}^s}{\partial\varphi^s}
-g_\Sigma{}^{ij}\frac{\partial\Gamma_\Sigma{}_{sj}^s}{\partial\varphi^i}
+g_\Sigma{}_{st}\Gamma_\Sigma{}^s\Gamma_\Sigma{}^t
-g_\Sigma{}^{is}g_\Sigma{}^{jt}g_\Sigma{}_{kl}\Gamma_\Sigma{}_{ij}^k\Gamma_\Sigma{}_{st}^l.
\end{eqnarray}
The pseudo-K\"ahler relation implies $\frac{\partial}{\partial\varphi^s}\left(g_\Sigma{}^{ij}\Gamma_\Sigma{}_{ij}^s\right)=\frac{\partial}{\partial\varphi^i}\left(g_\Sigma{}^{ij}\Gamma_\Sigma{}_{sj}^s\right)$.
Using this,
\begin{eqnarray}
	\begin{aligned}
	s_{\Sigma}
	&\;=\;
	g_\Sigma{}^{is}g_\Sigma{}^{jt}g_\Sigma{}_{kl}\Gamma_\Sigma{}_{ij}^k\Gamma_\Sigma{}_{st}^l-g_\Sigma{}_{st}\Gamma_\Sigma{}^s\Gamma_\Sigma{}^t \\
	&\;=\;
	\left|\Gamma_\Sigma{}_{ij}^k\right|^2\,-\,\left|\Gamma_\Sigma{}^k\right|^2
	\;=\;
	\left|\Gamma_\Sigma{}_{ij}^k\right|^2\,-\,\left|\nabla_\Sigma\log\mathcal{V}^{\frac12}\right|^2.
	\end{aligned} \label{EqnCoordSigmaScalarComp}
\end{eqnarray}
Modify the metric by $\widetilde{g}_{\Sigma}=\mathcal{V}^{\frac12}{g}_\Sigma$.
The usual conformal-change formula gives
\begin{eqnarray}
	\begin{aligned}
	\widetilde{s}_\Sigma
	&\;=\;\mathcal{V}^{-\frac12}\left(s_\Sigma\,-\,\triangle_\Sigma\log\mathcal{V}^{\frac12}\right) \\
	&\;=\;\mathcal{V}^{-\frac12}\left(
	\left|\Gamma_\Sigma{}_{ij}^k\right|^2\,-\,\left|\nabla_\Sigma{}\log\mathcal{V}^{\frac12}\right|^2
	\,-\,\triangle_\Sigma\log\mathcal{V}^{\frac12}\right).
	\end{aligned} \label{EqnConfChangeScalar}
\end{eqnarray}
But $\triangle_\Sigma\log\mathcal{V}^{\frac12}+|\nabla_\Sigma{}\log\mathcal{V}^{\frac12}|^2=\mathcal{V}^{-\frac12}\triangle_\Sigma\mathcal{V}^{\frac12}$.
Therefore
\begin{eqnarray}
	\begin{array}{ll}
	\widetilde{s}_\Sigma
	&\;=\;\mathcal{V}^{-\frac12}\left(\left|\Gamma_\Sigma{}_{ij}^k\right|^2\,-\,\mathcal{V}^{-\frac12}\triangle_\Sigma\mathcal{V}^{\frac12}\right)
	\;=\;\mathcal{V}^{-\frac12}\left(\left|\Gamma_\Sigma{}_{ij}^k\right|^2\,+\,\frac12s\right).
	\end{array} \label{EqnConformalScalar}
\end{eqnarray}
In particular, if $s\ge0$ then $\widetilde{s}_\Sigma\ge0$.

\section{The case that the polygon has no edges} \label{SectionRTwoCase}

Here we prove Theorem \ref{ThmIntroCompletePolygon}, that if $\triangle_\Sigma\sqrt{\mathcal{V}}\le0$ and the polygon $\Sigma^2$ is geodesically complete (not necessarily coordinate-complete), then it is flat.
This works because the conformal change $\tilde{g}_\Sigma=\sqrt{\mathcal{V}}g_\Sigma$ has two vital properties: $\tilde{g}_\Sigma$ remains complete, and its sectional curvature $\widetilde{K}_\Sigma$ is non-negative by (\ref{EqnConformalScalar}).
Below we speak of biholomorphisms between $\Sigma^2$ and $\mathbb{C}$; for $\Sigma^2$ we always use the complex structure $J_\Sigma:T\Sigma^2\rightarrow{}T\Sigma^2$ given by dualizing the Hodge-star $*{}_\Sigma:\bigwedge{}^1_\Sigma\rightarrow\bigwedge{}^1_\Sigma$.
It is well-known that this is integrable and gives $(\Sigma^2,g_\Sigma,J_\Sigma)$ a K\"ahler structure.
We proceed in steps.
First, if we somehow know $\mathcal{V}$ is constant, then $g_\Sigma$ is flat.
\begin{lemma} \label{LemmaTConstFlat}
	Assume $(\Sigma^2,g_\Sigma)$ is geodesically complete.
	If $\mathcal{V}$ is constant, then $(\Sigma^2,g_\Sigma)$ is biholomorphic to $\mathbb{C}$.
	Further, $g_\Sigma$ is flat, and in fact $g_\Sigma$ has constant coefficients when expressed in $\varphi^1$, $\varphi^2$ coordinates.
\end{lemma}
\begin{proof}
	If $\mathcal{V}$ is constant, then Proposition \ref{PropDivOfVarphi} gives $d(J_\Sigma{}d\varphi^k)=0$.
	Therefore $\varphi^1$ and $\varphi^2$ are harmonic, and each determines a holomorphic function: $z=\varphi^1+\sqrt{-1}\eta^1$, $w=\varphi^2+\sqrt{-1}\eta^2$ (where $d\eta^k=-J_\Sigma{}d\varphi^k$).
	Away from possible ramification points, $z$ and $w$ are each a holomorphic coordinate on $\Sigma^2$, with respective coordinate fields
	\begin{eqnarray}
		\begin{aligned}
		\frac{d}{dz}&\;=\;\frac{1}{2|\nabla\varphi^1|^2}\left(\nabla\varphi^1\,+\,\sqrt{-1}J_\Sigma\nabla\varphi^1\right) \\
		\frac{d}{dw}&\;=\;\frac{1}{2|\nabla\varphi^2|^2}\left(\nabla\varphi^2\,+\,\sqrt{-1}J_\Sigma\nabla\varphi^2\right).
		\end{aligned}
	\end{eqnarray}
	We easily compute the transition function:
	\begin{eqnarray}
		\frac{dw}{dz}
		\;=\;g\left(\frac{d}{dz},\,\overline{\nabla}{w}\right)
		\;=\;\frac{\left<\nabla\varphi^1,\,\nabla\varphi^2\right>}{|\nabla\varphi^1|^2}\,-\,\sqrt{-1}\frac{\sqrt\mathcal{V}}{|\nabla\varphi^1|^2}.
	\end{eqnarray}
	But $\frac{dw}{dz}$ is holomorphic, so in particular its imaginary part is harmonic, so $|\nabla\varphi^1|^{-2}$ is harmonic.
	In the $z$-coordinate the Hermitian metric is $h_\Sigma=\left|\frac{d}{dz}\right|^2=\frac12|\nabla\varphi^1|^{-2}$, so $h_\Sigma$ itself is an harmonic function.
	Then using $\triangle_\Sigma{}h_\Sigma=0$, we find the Gaussian curvature to be
	\begin{eqnarray}
		K_\Sigma
		\;=\;-h_\Sigma^{-1}\triangle_\Sigma\log{h}_\Sigma\;=\;8|\nabla|\nabla\varphi^1||^2
	\end{eqnarray}
	which is non-negative, forcing the complete manifold $(\Sigma^2,g_\Sigma)$ to be parabolic (this is due to the Cheng-Yau condition for parabolicity; see \cite{CY} or the remark below).
	This is equivalent to the complete, simply connected manifold $(\Sigma^2,J_{\Sigma})$ being biholomorphic to $\mathbb{C}$.
	Thus the classical Liouville theorem says the non-negative harmonic function $|\nabla\varphi^1|^{-2}$ is constant.
	Similarly $|\nabla\varphi^2|^{-2}$ and $\left<\nabla\varphi^1,\nabla\varphi^2\right>$ are constant.
	
	Because $g_\Sigma{}^{ij}=\left<\nabla\varphi^i,\nabla\varphi^j\right>$ we have that all components of the metric are constants when measured in $(\varphi^1,\varphi^2)$ coordinates.
	In particular $K_\Sigma=0$.
\end{proof}

Now we begin to use $\triangle_\Sigma\sqrt{\mathcal{V}}\le0$.
If, in addition to this, we somehow know $\Sigma^2$ is biholomorphic to $\mathbb{C}$, then $\sqrt{\mathcal{V}}$ is constant and $g_\Sigma$ is flat.

\begin{lemma} \label{LemmaBiholoToCFlat}
	Assume $\Sigma^2$ is geodesically complete, biholomorphic to $\mathbb{C}$, and  $\triangle_\Sigma\sqrt{\mathcal{V}}\le0$.
	Then $g_\Sigma$ is flat, and has constant coefficients in $\varphi^1$, $\varphi^2$ coordinates.
\end{lemma}
\begin{proof}
	Since the function $\mathcal{V}^{\frac12}$ is superharmonic and bounded from below on $\mathbb{C}$, it is constant, and Lemma \ref{LemmaTConstFlat} provides the result.
\end{proof}

Later we prove that $\triangle_\Sigma\sqrt{\mathcal{V}}\le0$ alone implies $\Sigma^2\approx\mathbb{C}$, which will utilize the conformal change $\tilde{g}_\Sigma=\sqrt{\mathcal{V}}g_\Sigma$.
First we verify this conformal change does not make a complete metric into an incomplete metric.

\begin{lemma} \label{LemmaConformalComplete}
	Assume $(\Sigma^2,g_\Sigma)$ is geodesically complete and $\triangle_\Sigma\sqrt{\mathcal{V}}\le0$.
	Setting $\widetilde{g}_\Sigma=\mathcal{V}^{\frac12}g_\Sigma$, then $(\Sigma^2,\widetilde{g}_\Sigma)$ is also complete.
\end{lemma}
\begin{proof}
	From a point $p\in\Sigma^2$ let $\gamma:[0,R]\rightarrow\Sigma^2$ be a shortest unit-speed geodesic in the $\widetilde{g}_\Sigma$ metric that is inextendable.
	For a contradiction, we show that $\gamma$ continues to have finite length in the $g$ metric, so the endpoint $\gamma(R)$ remains in the interior of $\Sigma^2$, where the metric $g_\Sigma$ is still smooth.
	It is therefore extendable in both the $g_\Sigma$ and $\tilde{g}_\Sigma$ metrics.
	
	Our main technical step is use Laplacian comparison to create a lower bound for $\mathcal{V}$ in the $\tilde{g}_\Sigma$-ball $B_p(R)$.	
	Let $\tilde{r}=dist_{\tilde{g}_\Sigma}(p,\cdot)$.
	Because $\widetilde{K}_\Sigma\ge0$ by (\ref{EqnConformalScalar}), the standard Bochner-style Laplacian comparison (eg. \cite{SY}) allows us to compare the Laplacian of the distance function to its flat-space counterpart, and we obtain $\widetilde\triangle_\Sigma\tilde{r}\le\tilde{r}{}^{-1}$ (in the barrier sense) and therefore $\widetilde\triangle_\Sigma\tilde{r}{}^{-1}\ge\tilde{r}{}^{-3}$.
	Let $\mathcal{A}$ be the annulus $\mathcal{A}=B_p(R)\setminus{}B_p(R/2)$---the metric $\tilde{g}_\Sigma$ is still smooth in the interior of $B_p(R)$, because of how we chose $R$.
	On the inner boundary $\partial{B}_p(R/2)$ we have $\mathcal{V}>0$ so there is some $\epsilon>0$ with $\mathcal{V}^\frac12\ge2\epsilon{}R^{-1}$, and on the outer boundary $\mathcal{V}^\frac12\ge0$.
	This means
	\begin{eqnarray}
		\begin{aligned}
		&\mathcal{V}^{\frac12}\,-\,\epsilon\left(\tilde{r}{}^{-1}\,-\,R^{-1}\right)
		 \;\ge\;0 \quad \text{on $\partial\mathcal{A}$, and} \\
		&\widetilde{\triangle}_\Sigma\left(\mathcal{V}^{\frac12}
		\,-\,\epsilon\left(\tilde{r}{}^{-1}\,-\,R^{-1}\right)\right)
		\;\le\;-\epsilon\tilde{r}{}^{-3}\;<\;0 \quad \text{on $\mathcal{A}$} \label{EqnBochnerBarrierCreation}
		\end{aligned}
	\end{eqnarray}
	so by the maximum principle $\mathcal{V}^{\frac12}\ge\epsilon\left(\tilde{r}{}^{-1}\,-\,R^{-1}\right)$ on the closure of $\mathcal{A}$.
	Because the path $\gamma$ lies within the annulus $\mathcal{A}$ for $t\in[R/2,R]$ we have
	\begin{eqnarray}
		\begin{aligned}
		\left(\mathcal{V}{}^{\frac12}\circ\gamma\right)(t)
		&\;\ge\;\epsilon\left(t^{-1}-R^{-1}\right),
		\quad \text{on $t\in[R/2,R]$, which is} \\
		\left(\mathcal{V}{}^{-\frac14}\circ\gamma\right)(t)
		&\;\le\;\epsilon^{-\frac12}\left(t^{-1}-R^{-1}\right)^{-\frac12}
		\;=\;\epsilon^{-\frac12}(tR)^{\frac12}\left(R-t\right)^{-\frac12}.
		\end{aligned}
	\end{eqnarray}
	This allows us to estimate the $g_\Sigma$-length of $\gamma$ from its $\tilde{g}_\Sigma$-length:
	\begin{eqnarray}
		\begin{aligned}
			|\dot\gamma|_{g_\Sigma}
			&\;=\;\left(\mathcal{V}{}^{-\frac14}\circ\gamma\right)\cdot|\dot\gamma|_{\tilde{g}_\Sigma}
			\;\le\;\epsilon^{-\frac12}(tR)^{\frac12}\left(R-t\right)^{-\frac12}
		\end{aligned}
	\end{eqnarray}
	on $t\in[R/2,R]$, where we used that $\gamma$ is a $\tilde{g}_\Sigma$-geodesic so $|\dot\gamma|_{\tilde{g}_\Sigma}=1$.
	We get
	\begin{eqnarray}
		\begin{aligned}
			Len_{g_\Sigma}(\gamma)
			&=\int_{R/2}^R|\dot\gamma|_{g_\Sigma}dt
			\le
			\epsilon^{-\frac12}R^{\frac12}\int_{R/2}^Rt^{\frac12}(R-t)^{-\frac12}dt
			=
			\frac{2+\pi}{4\sqrt{\epsilon}}R^{\frac32}.
		\end{aligned}
	\end{eqnarray}
	Thus $\gamma$ has finite length in the $g_\Sigma$ metric, the sought-for contradiction.
\end{proof}

\begin{theorem}[{\it cf.} Theorem \ref{ThmIntroCompletePolygon}] \label{ThmRTwoClassification}
	Assume $(\Sigma^2,g_\Sigma)$ is geodesically complete and $\triangle_\Sigma\sqrt{\mathcal{V}}\le0$.
	Then $(\Sigma^2,g_\Sigma)$ and $(\Sigma^2,\widetilde{g}_\Sigma)$ are both flat Riemannian manifolds.
\end{theorem}
\begin{proof}
	By Lemma \ref{LemmaConformalComplete}, $(\Sigma^2,\widetilde{g}_\Sigma)$ is also complete, and by (\ref{EqnConformalScalar}) $\widetilde{K}_\Sigma\ge0$.
	The Cheng-Yau criterion for parabolicity states that $(\Sigma^2,\widetilde{g}_\Sigma)$ is parabolic and therefore biholomorphic to $\mathbb{C}$.
	Then $\mathcal{V}^{\frac12}$ is a positive superharmonic function on $\mathbb{C}$, so it is constant by the classical Liouville theorem.
	Lemma \ref{LemmaTConstFlat} now gives the conclusion.
\end{proof}

\begin{corollary} [{\it cf.} Corollary \ref{ThmIntroM4CompletePolygon}] \label{CorRTwoClassification}
	Assume $(M^4,J,\omega,\mathcal{X}_1,\mathcal{X}_2)$ has $s\ge0$, and assume the distribution $\{\mathcal{X}_1,\mathcal{X}_2\}$ is everywhere rank 2 (meaning $\mathcal{V}$ is nowhere zero).
	Then $M^4$ is a flat Riemannian manifold.
\end{corollary}
\begin{proof}
	Because the distribution $\{\mathcal{X}_1,\mathcal{X}_2\}$ has rank $2$, $\{d\varphi^1,d\varphi^2\}$ also has rank 2 and so the Arnold Liouville projection $(\varphi^1,\varphi^2,\theta_1,\theta_2)\rightarrow(\varphi^1,\varphi^2)$ is a Riemannian submersion.
	Therefore the metric polygon $(\Sigma^2,g_\Sigma)$ is complete.
	By Corollary \ref{CorReducedScalar} we have $\triangle_\Sigma\sqrt{\mathcal{V}}\le0$, so
	Theorem \ref{ThmRTwoClassification} implies $g_\Sigma$ is flat, and $g_\Sigma$ has constant coefficients in $\varphi^1$-$\varphi^2$.
	From equations (\ref{EqnsGJOmegaM}) we see $G$ is a constant matrix, so $g$ and $J$ are constant on $M^4$.
	Therefore $M^4$ is flat.
\end{proof}

{\bf Remark.}
Crucial to the proofs of Lemma \ref{LemmaTConstFlat} and Theorem \ref{ThmRTwoClassification} is the fact that a complete, simply connected $\Sigma^2$ with $K_\Sigma\ge0$ is biholomorphic to $\mathbb{C}$.
This is a simple consequence of Cheng-Yau criterion for parabolicity; see proposition 3 and corollary 1 of \cite{CY}.
The assertion that a simply connected, complete Riemann surface is parabolic if and only if it is actually $\mathbb{C}$ is a consequence of uniformization.
The study of parabolicity of Riemannian manifolds has received a great deal of attention; for a tiny sampling of this large subject see for example \cite{LT1} \cite{LT2} \cite{HK} \cite{GM} \cite{Web1}.

\section{Global behavior of the analytic coordinate system} \label{SecGlobalHarmonicCoordBehavior}

To establish our classification we want to use off-the-shelf analytic results from \cite{Web2}, but before this becomes available, a good coordinate system is required against which the momentum functions can be measured.
This is done by showing the volumetric normal function $z:\Sigma^2\rightarrow\overline{H}{}^2$ is a global coordinate, in particular that the analytic function $z=x+\sqrt{-1}y$ is unramified and surjective.
\begin{proposition}[Injectivity of $z$, {\it cf.} \S\ref{SubsecInjectivity}] \label{ThmInjectivityOfZ}
	Assume the metric polygon satisfies (A)-(F).
	Then the map $z:\Sigma^2\rightarrow\overline{H}{}^2$ is injective.
\end{proposition}
\begin{proposition}[Surjectivity of $z$, {\it cf.} \S\ref{SubsecSurjectivity}] \label{ThmSurjectivityOfZ}
	Assume the metric polygon satisfies (A)-(F).
	Then the map $z:\Sigma^2\rightarrow\overline{H}{}^2$ is surjective.
\end{proposition}
The proofs work by exploiting the interplay between the coordinate systems $(x,y)$ and $(\varphi^1,\varphi^2)$ on $\Sigma^2$, each of which satisfy and elliptic system in terms of the other.
Recall that $y\triangleq\sqrt{\mathcal{V}}\ge0$, and that $x$ is defined as its harmonic conjugate: $dx=J_\Sigma{}dy$.
Three immediate facts about the analytic map $z$ are
\begin{itemize}
	\item[{\it{i}})] $z$ has no poles
	\item[{\it{ii}})] $z$ maps boundary to boundary: $z:\partial\Sigma^2\rightarrow\{y=0\}$
	\item[{\it{iii}})] $z$ maps interior to interior:
	 $z:Int(\Sigma^2)\rightarrow{}H^2$.
\end{itemize}
Briefly, ({\it{i}}) is true because $y=\sqrt{\mathcal{V}}$ is never infinite, and ({\it{ii}}) and ({\it{iii}}) are the same as condition (E), although ({\it{iii}}) is also a consequence of the open mapping theorem.

The interplay between the coordinate systems is mediated by the barrier functions $\psi_{\mathcal{R}}$ and $\psi^1_{\mathcal{D}_{y'}}$ built in Section \ref{SubsecTheBarriers}.
The difficulty in creating these barriers is that the elliptic operator $y\triangle-\partial_y$ degenerates at the boundary $\{y=0\}$, so it is not clear whether we can assign boundary values there.
We show we can.
Then we use these barriers to prove our ``one-component'' lemmas, which states certain kinds of regions in $\overline{H}{}^2$ have just one component under $z^{-1}$.
This establishes injectivity.

For surjectivity, if coordinate-infinity in $\Sigma^2$ maps to a finite location in the $z$-plane, the image of the momentum functions must resemble an isolated pole.
But poles are much too structured for this to occur.
For example they have a winding number which can be determined within the pre-image, and the pre-image in this case (being $\Sigma^2$) does not allow this---there can be no ``loop'' around coordinate-infinity.
The main technical result expressing this idea is the ``Disk Lemma'' \ref{LemmaDiskAroundPBPoint}.

But the crucial first step, before the powerful ``one-component'' lemmas can be proved, is proving $z:\Sigma^2\rightarrow\overline{H}{}^2$ is bijective when restricted to the boundary itself, $z:\partial\Sigma^2\rightarrow\{y=0\}$.
Then using some classical elliptic theory we move outwards just a little and show there is a neighborhood $\Omega$ of the boundary $\partial\Sigma^2$ on which $z:\Omega\rightarrow{}z(\Omega)$ remains bijective.
We call this the ``bijectivity zone'' for $z$; see \S\ref{EqnBijectivityZone}.

\subsection{Construction of the barrier functions} \label{SubsecTheBarriers}

The two coordinate systems interact through the elliptic equations they satisfy.
The functions $\varphi^1$, $\varphi^2$ satisfy
\begin{equation}
	d\left(\mathcal{V}^{-1/2}J_\Sigma{}d\varphi^i\right)=0 \label{EqnMomentGlobalDegEq}
\end{equation}
which is Proposition \ref{PropDivOfVarphi}.
This is a degenerate-elliptic equation; the degeneration occurs where $\mathcal{V}=0$, which is at polygon edges.
The functions $x$, $y$ satisfy $dJ_\Sigma{}dx=dJ_{\Sigma}{}dy=0$ by Corollary \ref{CorReducedScalar}.
On any simply connected subdomain on which $z=x+\sqrt{-1}y$ is unramified, (\ref{EqnMomentGlobalDegEq}) is
\begin{eqnarray}
	\begin{aligned}
	&y\left(\frac{\partial^2\varphi^i}{\partial{}x^2}+\frac{\partial^2\varphi^i}{\partial{}y^2}\right)\,-\,\frac{\partial\varphi^i}{\partial{y}}\;=\;0.
	\end{aligned}
\end{eqnarray}
Our tool for establishing global relationships between the $(x,y)$ and $(\varphi^1,\varphi^2)$ systems is the use of barrier functions.
The first barrier we construct has support within the rectangle
\begin{eqnarray}
	\mathcal{R}_{x_0,y_0}
	=
	\left\{(x,y)\in\overline{H}{}^2\;\big|\;x\in[-x_0,x_0],\,y\in[0,y_0] \right\}
	\label{EqnBoxDef}
\end{eqnarray}
and has the explicit definition
\begin{eqnarray}
	\begin{aligned}
	&\psi_{x_0,y_0}(x,\,y)
	=
	\frac{1}{C_1}
	\cos\left(\frac{\pi}{2x_0}x\right)y\,K_1\left(\frac{\pi}{2x_0}y\right)
	-\frac{y_0}{C_1}\,K_1\left(\frac{\pi{}y_0}{2x_0}\right)
	\end{aligned} \label{EqnDefZeroBarrier}
\end{eqnarray}
where $C_1=\frac{2x_0}{\pi}-y_0K_1\left(\frac{\pi{}y_0}{2x_0}\right)$ and $K_\nu$ is the familiar modified Bessel function of the second kind.
We define $\psi_{x_0,y_0}$ to be zero whenever the expression in (\ref{EqnDefZeroBarrier}) is negative or when $(x,y)$ is outside the box $\mathcal{R}_{x_0,y_0}$.
Figure \ref{SubFig2} depicts $\psi_{x_0,y_0}$.


\begin{lemma}[Use of $\psi_{x_0,y_0}$ as a barrier] \label{LemmaBarrierProps}
	The function $\psi_{x_0,y_0}$ is $C^\infty$ on the interior of its support and satisfies
	\begin{eqnarray}
		y\left(\psi_{xx}+\psi_{yy}\right)\,-\,\psi_y\;=\;0.
	\end{eqnarray}
	It is $C^{1,\alpha}$ on the boundary $\{y=0\}$, and achieves a maximum of $1$ at $(0,0)$.
	
	Assume $f\ge0$ is any bounded function on $\mathcal{R}_{x_0,y_0}$ with $y\left(f_{xx}+f_{yy}\right)-f_y=0$.	
	If $\psi_{x_0,y_0}<f$ on the edge
	$\left\{(x,0)\,\big|\,x\in[-x_0,x_0]\right\}$,
	then $\psi_{x_0,y_0}<{}f$ on $\mathcal{R}_{x_0,y_0}$.
\end{lemma}
\begin{proof}
	This is a consequence of the maximum principle, along with the fact that $\psi_{x_0,y_0}\le0$ on three sides of the rectangle $\mathcal{R}_{x_0,y_0}$.
\end{proof}

In addition to the barrier $\psi_{x_0,y_0}$ on the box $\mathcal{R}_{x_0,y_0}$, we require similar barriers on more general domains (see Figures \ref{SubFig3} and \ref{SubFig4}).
Let $\mathcal{R}\subset\overline{H}{}^2$ be any domain on the upper half-plane with the following three characteristics:
\begin{itemize}
	\item[{\it{1}})] $\mathcal{R}\subset\overline{H}{}^2$ is precompact,
	\item[{\it{2}})] $\mathcal{R}$ is contractible, and
	\item[{\it{3}})] the closure $\overline{\mathcal{R}}$ intersects the boundary $\{y=0\}$ in a single line segment.
\end{itemize}
\begin{lemma}[Construction of a barrier $\psi_\mathcal{R}$ subordinate to a region $\mathcal{R}$] \label{LemmaBarrierR}
	Assuming $\mathcal{R}$ satisfies (1)-(3) above, there exists a function $\psi_{\mathcal{R}}$ with the following properties:
	\begin{itemize}
		\item[{\it{i}})] $y(\psi_{xx}+\psi_{yy})-\psi_y=0$ in $\mathcal{R}\setminus\{y=0\}$
		\item[{\it{ii}})] $\psi=0$ on the non-degenerate boundary $\partial\mathcal{R}\setminus\{y=0\}$
		\item[{\it{iii}})] $\psi=1$ on the degenerate boundary $\partial\mathcal{R}\cap\{y=0\}$
	\end{itemize}
\end{lemma}
\begin{proof}
	The ellipticity of the operator $y\triangle-\partial_y$ breaks down near the degenerate boundary $\{y=0\}$, so the classic Dirichlet theory can't be quoted.
	We can proceed using a domain-clipping method: for each $\delta>0$ let $\mathcal{R}_\delta=\mathcal{R}\cap\{y>\delta\}$, which is just clipping away a small strip from $\mathcal{R}$ near the degenerate boundary.
	In particular the operator $y\triangle-\partial_y$ is uniformly elliptic on $\mathcal{R}_\delta$.
	Let $\psi_\delta$ be the solution to the Dirichlet problem $y\triangle\psi_\delta-\partial_y\psi_\delta=0$ on $\mathcal{R}_\delta$ with boundary values $\psi_\delta=1$ on $\partial\mathcal{R}_\delta\cap\{y=\delta\}$ and $\psi_\delta=0$ on $\partial\mathcal{R}_\delta\setminus\{y=\delta\}$.
	Set $\psi_\mathcal{R}\;=\;\lim_{\delta\searrow0}\psi_\delta$.
	
	Due to ellipticity away from $\{y=0\}$, we certainly have $\psi_\mathcal{R}=0$ on $\partial{\mathcal{R}}\setminus\{y=0\}$ and $0\le\psi_{\mathcal{R}}\le1$.
	However we don't know that $\psi_\mathcal{R}$ is continuous at the degenerate boundary, or, if continuous, what its boundary values are.
	These issues are rectified if we can find a lower barrier $\Psi$ that equals $1$ at the degenerate boundary.
	To finish the lemma, we construct such a $\Psi$.
	
	Let $x'$ be any value so that the point $(x',0)$ is in the degenerate boundary $\partial\mathcal{R}\cap\{y=0\}$.
	Then there is some $\epsilon=\epsilon(x')>0$ so that the box $\{(x,y)\,|\,x\in(x'-\epsilon,x'+\epsilon),\,y\in(0,\epsilon)\}$ lies within $\mathcal{R}$.
	Subordinate to each such box is the function given by $\psi_{x'}(x,y)=\psi_{\epsilon,\epsilon}(x-x',y)$ where $\psi_{\epsilon,\epsilon}$ is the barrier of Lemma \ref{LemmaBarrierProps}.
	By construction, we have the support $supp\,\psi_{x'}$ within $\mathcal{R}$, but also that $\psi_{x'}$ is continuous at the boundary and in fact $\psi_{x'}(x',0)=1$.
	
	Then define $\Psi$ to be the supremum of all such barriers.
	That is:
	\begin{eqnarray}
		\Psi(x,y)\;=\;\sup_{x'\in\partial\mathcal{R}\cap\{y=0\}}\left\{\psi_{x'}(x,y),\,0\right\}. \label{EqnLowerSubBarrierDef}
	\end{eqnarray}
	Then $\Psi$ a lower semicontinuous subfunction, and $0\le\Psi\le1$.
	Its support is a neighborhood of the degenerate boundary portion within $\mathcal{R}$.
	Using any $\psi_{x'}$ as a lower barrier for $\Psi$, clearly $\Psi=1$ at the degenerate boundary.
	By construction of $\psi_\delta$ we have $\Psi<\psi_\delta\le1$ (on their common support), so sending $\delta\rightarrow0$ gives $\Psi\le\psi\le1$.
	This produces a sandwich at the degenerate boundary, so $\psi_{\mathcal{R}}$ is continuous and equals 1 there.
\end{proof}

\noindent\begin{figure}[h!]
	\centering
	\vspace{-0.15in}
	\hspace{-0.05in}
	\noindent\begin{subfigure}[b]{0.4\textwidth} 
		\centerline{\includegraphics[scale=0.25]{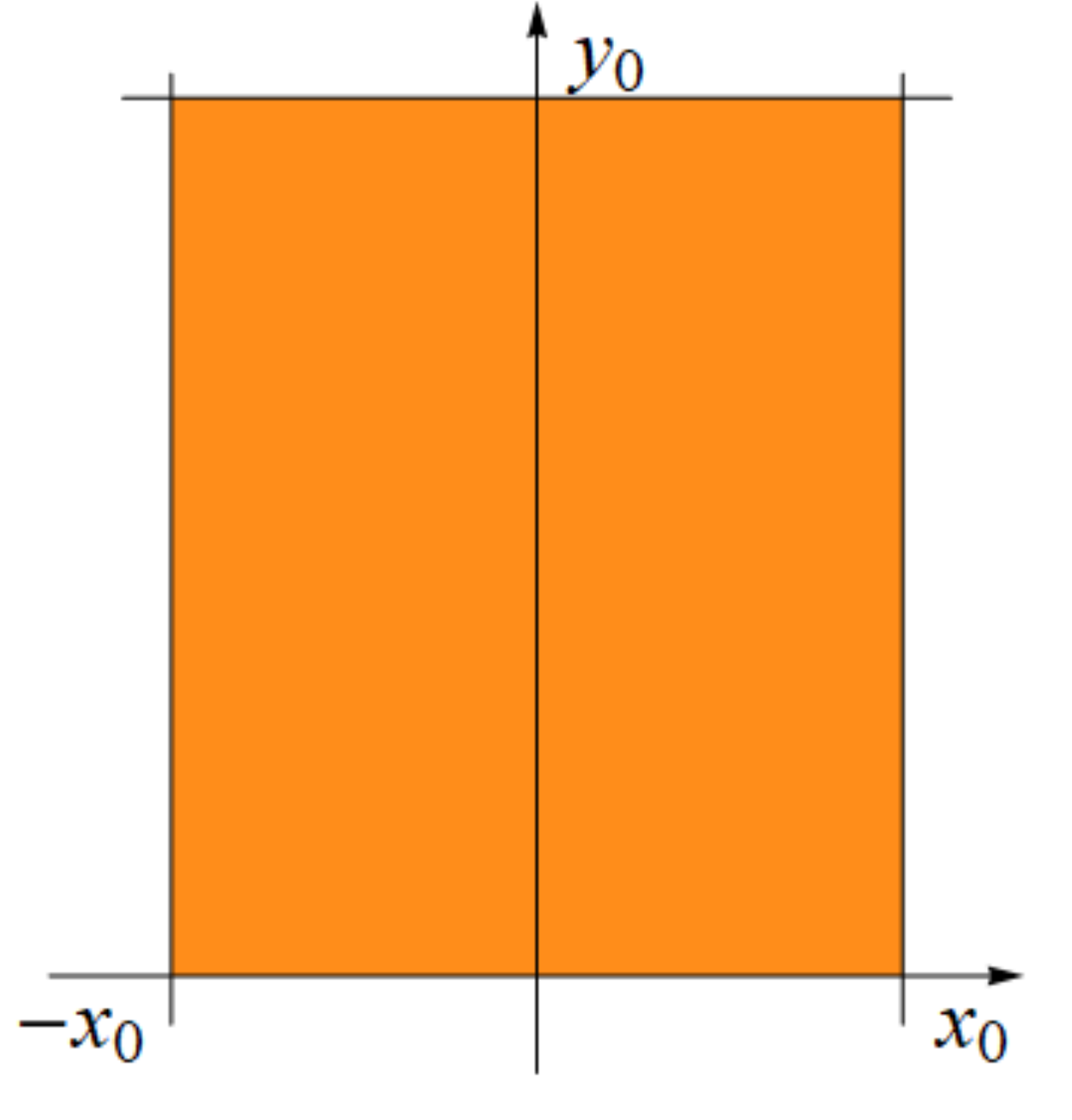}}
		\caption{The box $\mathcal{R}_{x_0,y_0}$.
			\label{SubFig1}}
	\end{subfigure}
	\hspace{0.05in}
	\begin{subfigure}[b]{0.55\textwidth} 
		\centerline{\includegraphics[scale=0.55]{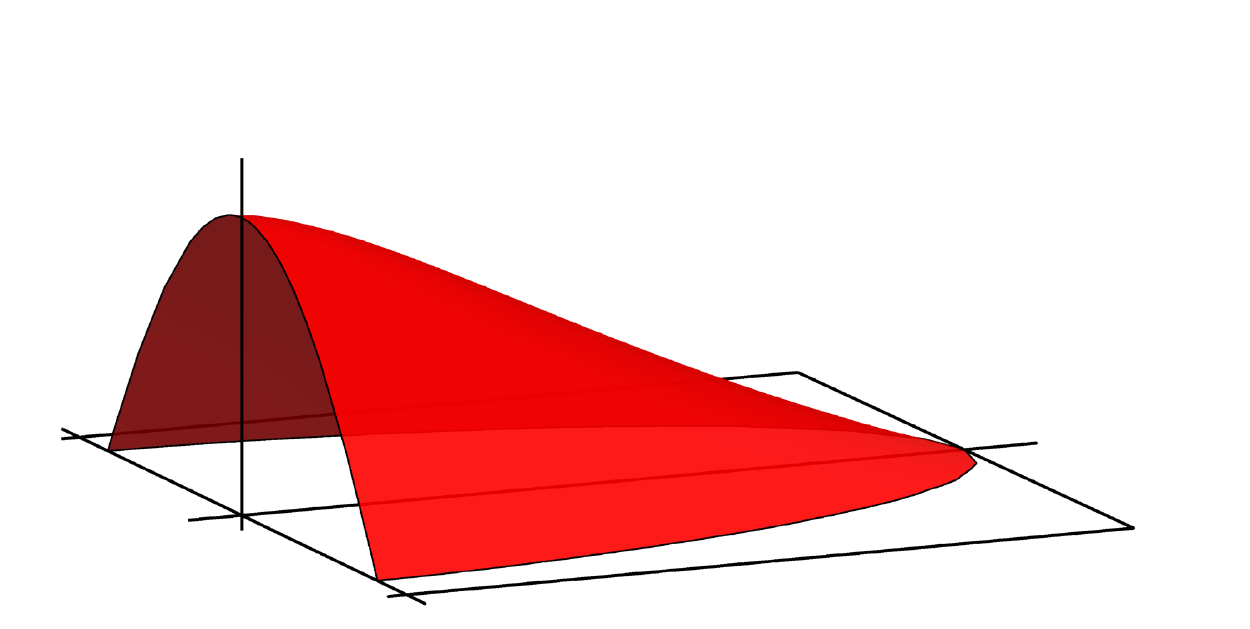}}
		\caption{Solution $\psi_{x_0,y_0}$ subordinate to $\mathcal{R}_{x_0,y_0}$
			\label{SubFig2}}
	\end{subfigure}
	
	\begin{subfigure}[b]{0.4\textwidth} 
		\centerline{\includegraphics[scale=0.25]{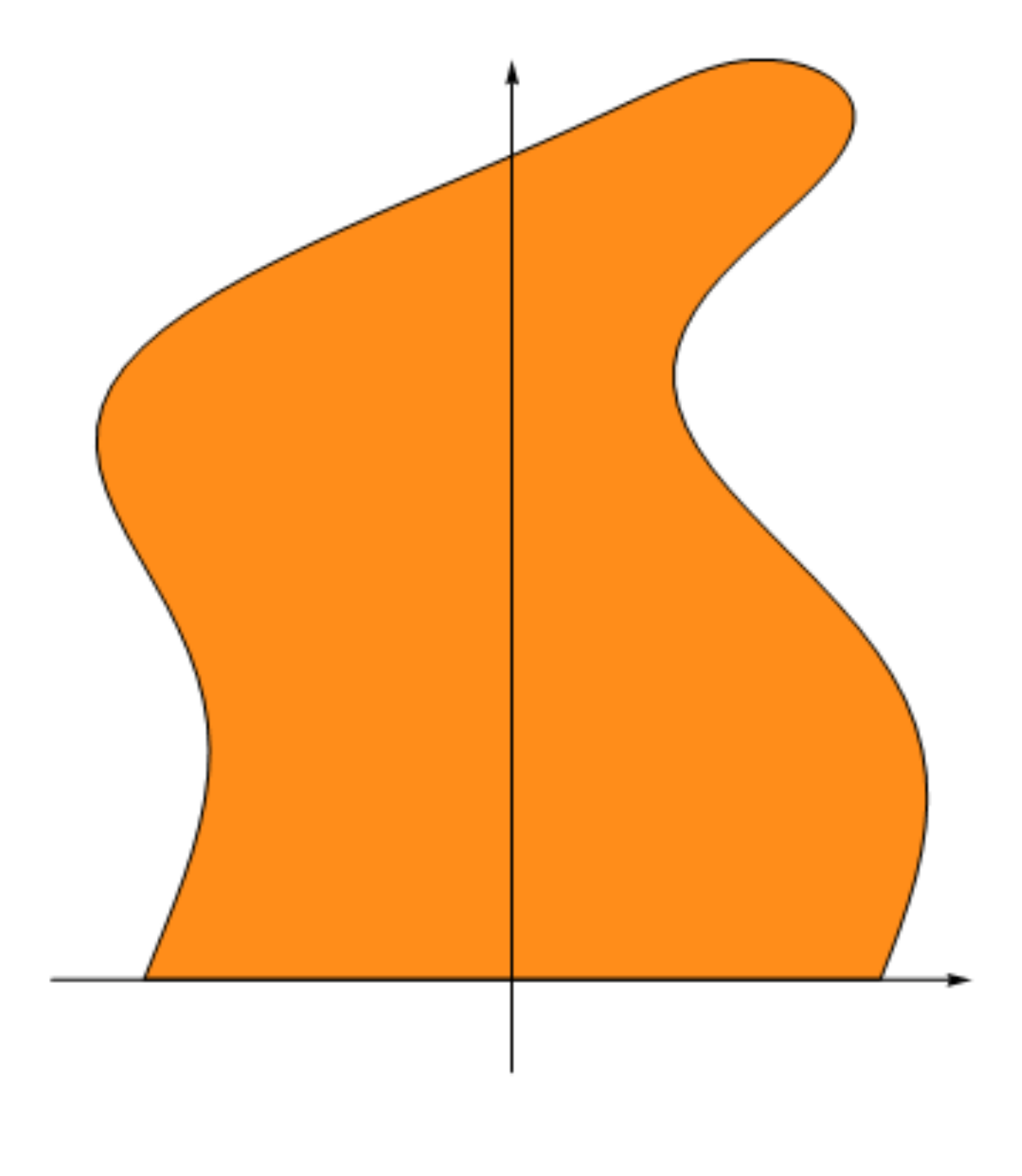}}
		\caption{A region $\mathcal{R}$ satisfying (1)-(3). \label{SubFig3}}
	\end{subfigure}
	\hspace{0.05in}
	\begin{subfigure}[b]{0.55\textwidth} 
		\centerline{\includegraphics[scale=0.45]{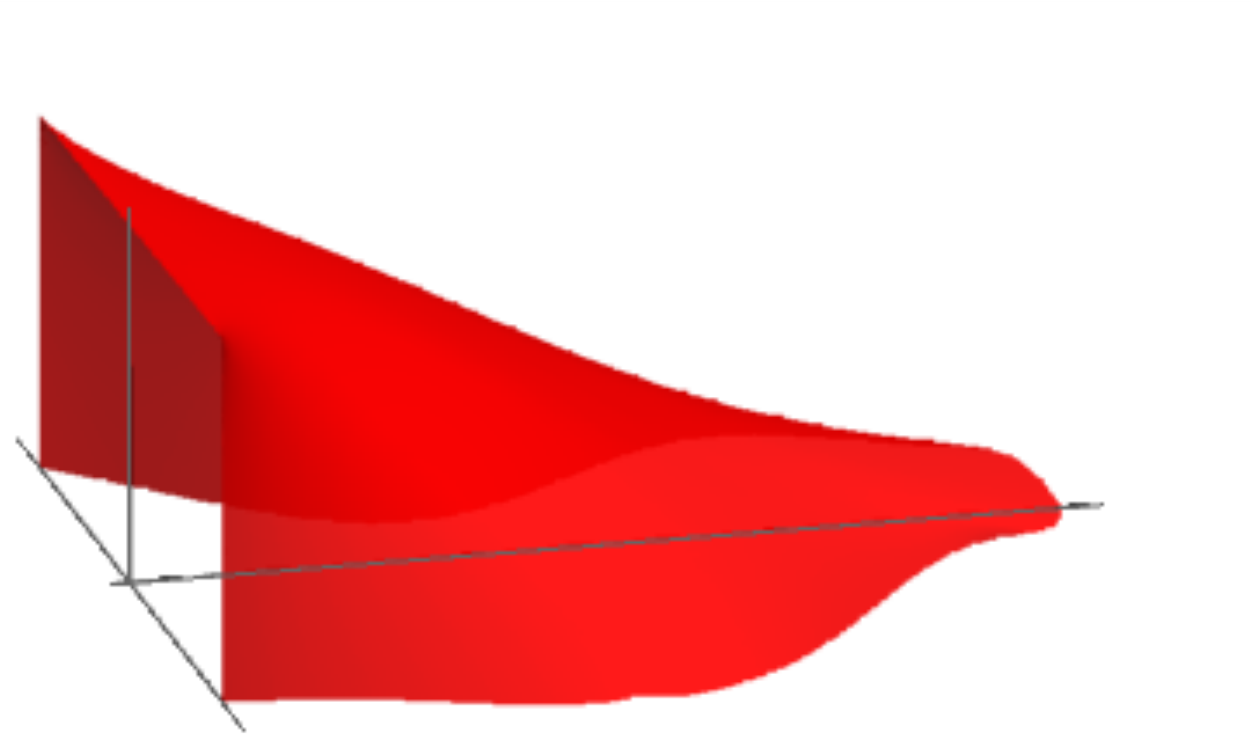}}
		\caption{Solution $\psi_{\mathcal{R}}$ subordinate to $\mathcal{R}$.
			\label{SubFig4}}
	\end{subfigure}
	
	\begin{subfigure}[b]{0.4\textwidth} 
		\centerline{\includegraphics[scale=0.25]{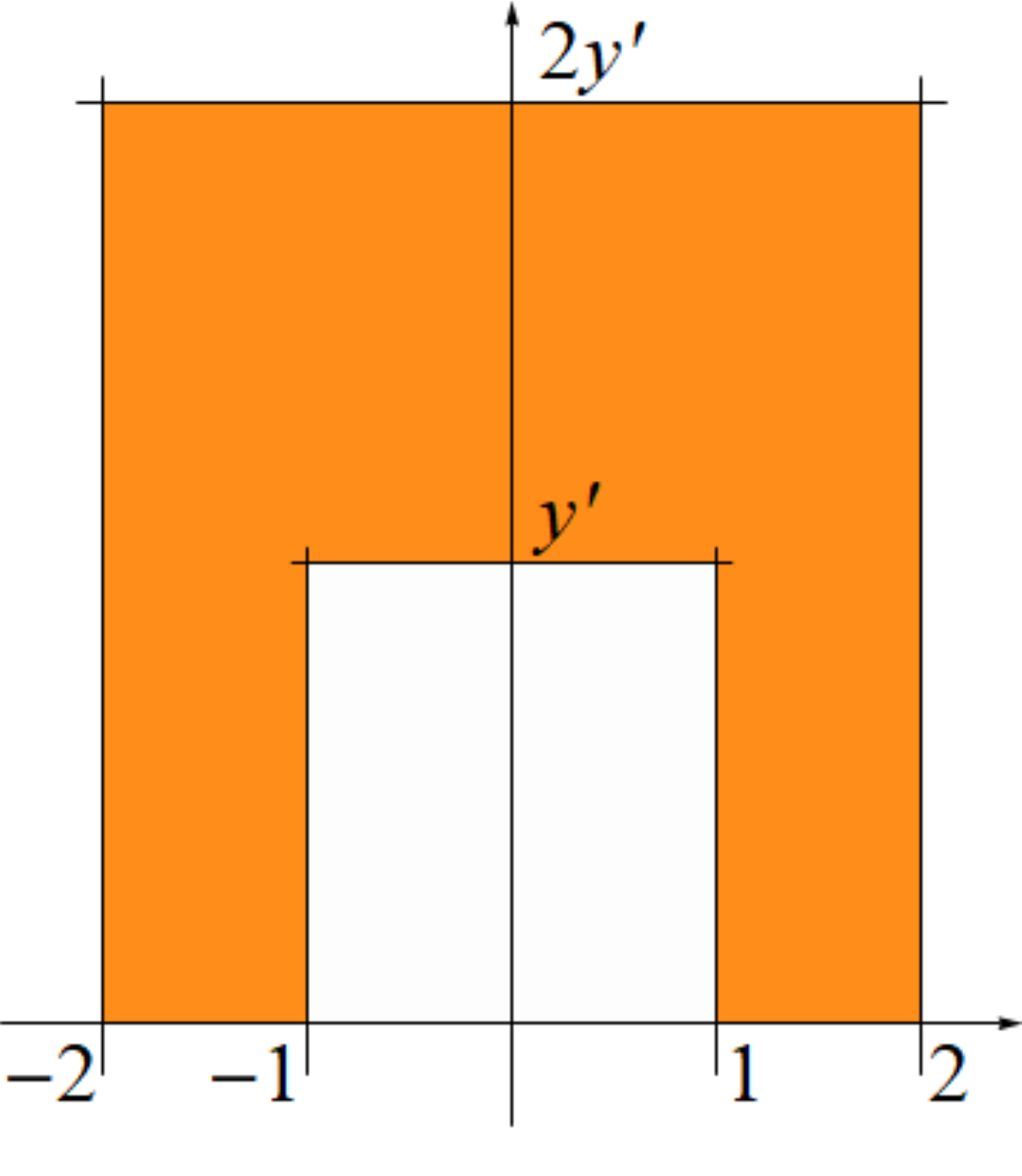}}
		\caption{The domain $\mathcal{D}_{y'}$.
			\label{SubFigRectDom}}
	\end{subfigure}
	\hspace{0.05in}
	\begin{subfigure}[b]{0.55\textwidth} 
		\centerline{\includegraphics[scale=0.5]{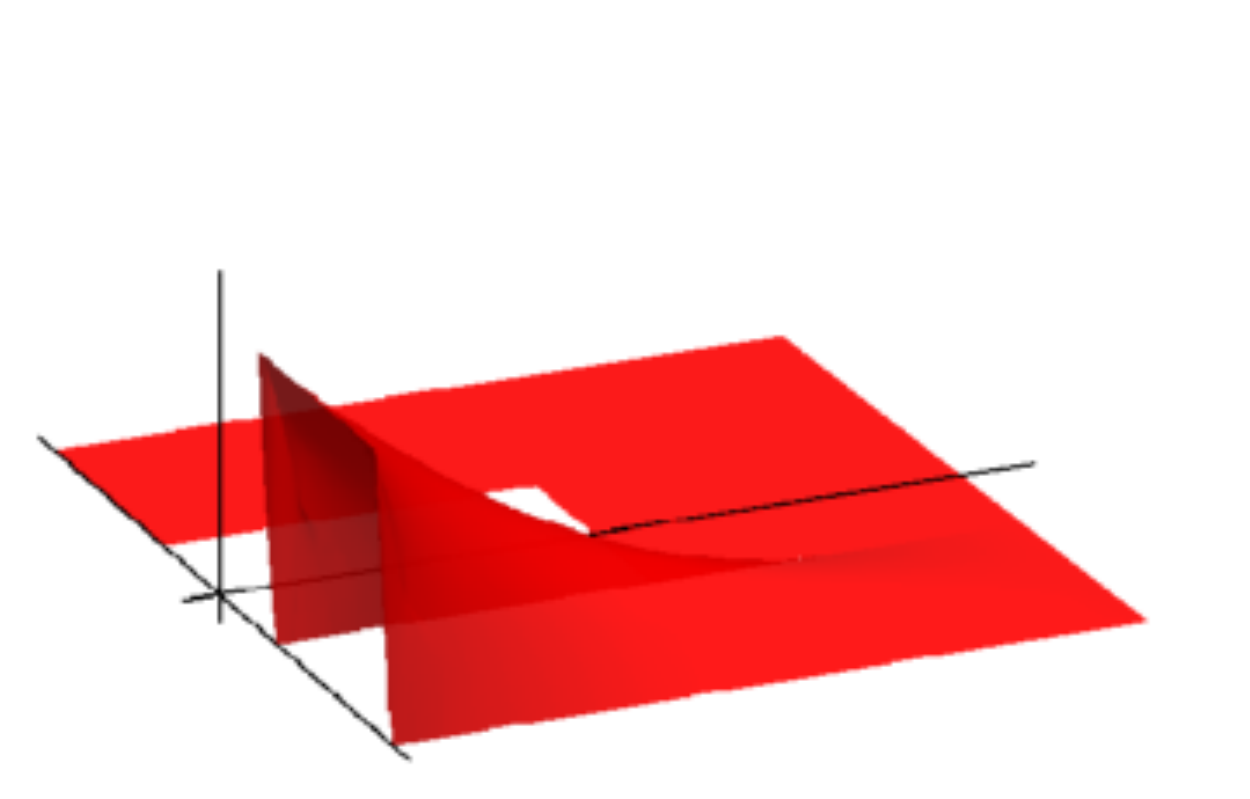}}
		\caption{Solution $\psi^2_{\mathcal{D}{}_{y'}}$ subordinate to $\mathcal{D}_{y'}$
			\label{SubFig6}}
	\end{subfigure}
	
	\caption{\it The three barrier types developed in Section \ref{SubsecTheBarriers}. }\label{FigBarriers}
\end{figure}

We shall require a barrier on a region that intersects the degenerate boundary along two different segments.
Specifically, let $\mathcal{D}_{y'}$ be the domain
\begin{eqnarray}
	\begin{aligned}
		\mathcal{D}_{y'}
		\;=\;\mathcal{R}_{2,2y'}\,\setminus\,\overline{\mathcal{R}_{1,y'}}.
	\end{aligned} \label{DomainTwoIntsWithBoundary}
\end{eqnarray}
The closure of this region has two line segments that intersect the degenerate boundary: $l_1=\{y=0,x\in[-2,-1]\}$ and $l_2=\{y=0,x\in[1,2]\}$.
We define two barriers: $\psi^1_{\mathcal{D}_{y'}}$ which is $1$ along $l_1$ and $0$ along $l_2$ and $\psi^2_{\mathcal{D}_{y'}}$ which is $1$ along $l_2$ and $0$ along $l_1$, and both are zero at the non-generate part of $\partial\mathcal{D}_{y'}$.
See Figures \ref{SubFigRectDom} and \ref{SubFig6}.
\begin{lemma}[Two barriers subordinate to $\mathcal{D}_{y'}$] \label{LemmaBarrierD}
	There is a function $\psi^1_{\mathcal{D}_{y'}}$ satisfying $y\triangle\psi-\partial_y\psi=0$ in $\mathcal{D}_{y'}$ with the boundary data that $\psi^1_{\mathcal{D}_{y'}}$ equals $1$ on $l_1$ and equals zero on all other boundary points of $\mathcal{D}_{y'}$.
	
	There is a function $\psi^2_{\mathcal{D}_{y'}}$ satisfying $y\triangle\psi-\partial_y\psi=0$ in $\mathcal{D}_{y'}$ with the boundary data that $\psi^2_{\mathcal{D}_{y'}}$ equals $1$ on $l_2$ and equals zero on all other boundary points of $\mathcal{D}_{y'}$.
\end{lemma}
\begin{proof}
	This follows after using the exhaustion method from Lemma \ref{LemmaBarrierR}.
	To recap, we solve the Dirichlet problem on the ``clipped'' region $\mathcal{D}\cap\{\delta>0\}$, and send $\delta\searrow0$.
	Then one must prove $\psi^1_{\mathcal{D}_{y'}}$ or $\psi^2_{\mathcal{D}_{y'}}$ converges to a solution with the correct boundary values.
	This is achieved by once again constructing a lower barrier at the boundary as was done in (\ref{EqnLowerSubBarrierDef}).
\end{proof}

{\bf Remark}. 
The barrier $\psi_{\mathcal{R}}$ is used to prove the ``one-component lemma'' \ref{LemmaBoxLemma}, which is crucial for global injectivity.
The barrier $\psi^1_{\mathcal{D}_{y'}}$ is used to prove our second ``one-component lemma'' \ref{LemmaTwoBoxLemma}, which is crucial for global surjectivity.
The only one of our barriers that was explicit, $\psi_{x_0,y_0}$, will not be used any further.

\subsection{The bijectivity zone near $\partial\Sigma^2$.} \label{EqnBijectivityZone}

The behaviors of the coordinate systems are most tightly constrained at the boundaries $\partial\Sigma^2$ and $\{y=0\}$.
A simple argument using the classical Hopf Lemma \cite{GT} shows that a certain ``bijectivity zone'' $\Omega$ of $z$ must extend inward from the boundary some small way; see Figure \ref{FigBijZone}.
The starting point for global bijectivity is establishing the bijectivity of $z$ on this small zone.
\noindent\begin{figure}[h!]
	\centering
	\includegraphics[scale=0.35]{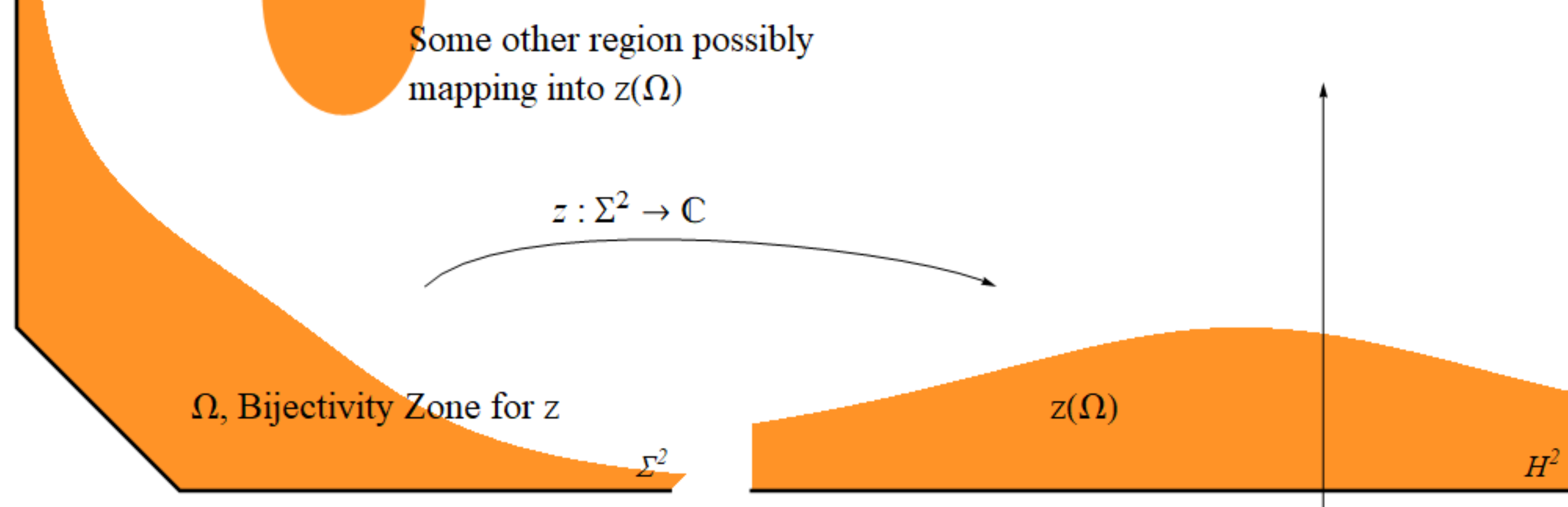}
	\caption{\it Depiction of the ``bijectivity zone for $z$'' established in Lemmas \ref{LemmaPreBijectiveZone}, \ref{LemmaPiecewise}, and \ref{LemmaBijectiveNearBoundary}.
	Later, Lemma \ref{LemmaCloseBijectivityZInv} will remove the possibility---depicted in the upper left---of any other region that might also map into $z(\Omega)$.}\label{FigBijZone}
\end{figure}
\begin{lemma} \label{LemmaPreBijectiveZone}
	Assume $(\Sigma^2,g_\Sigma)$ satisfies (A)-(F) of the introduction.
	Then $z$ maps $\partial\Sigma^2$ injectively into the boundary $\{y=0\}$.
	There is a neighborhood $\Omega$ of $\partial\Sigma^2$---which we call the ``bijectivity zone''---on which $z$ remains injective.
\end{lemma}
\begin{proof}
	By Hypothesis (C), $g$ is differentiable up to smooth points of $\partial\Sigma^2$ and Lipschitz at corner points.
	The harmonic function $y=\sqrt{\mathcal{V}}$ on $\Sigma^2$ is zero precisely on $\partial\Sigma^2$, by Hypothesis (E), meaning $z:\partial\Sigma^2\rightarrow\{y=0\}$ and $z^{-1}(\{y=0\})\subseteq\partial\Sigma^2$.
	Then because the Laplacian has either smooth or at worst has Lipschitz coefficents (at the corner points), $y$ is smooth everywhere except possibly at corner points where it is Lipschitz.
	Because $y\equiv0$ on $\partial\Sigma^2$ and $y\ge0$ on $\Sigma^2$, the classical Hopf Lemma \cite{GT} states that $|dy|>0$ at the boundary.
	
	Because $z=x+\sqrt{-1}y$ is an analytic function we have $|dx|=|dy|$ on $\Sigma^2$ so in particular $|dx|>0$ on $\partial\Sigma^2$.
	Therefore $z:\partial\Sigma^2\rightarrow\{y=0\}$ is injective.
	This map between 1-dimensional manifolds is smooth on segments of $\partial\Sigma^2$ and continuous at corner points.
	By smoothness $|dz|>0$ on some neighborhood $\Omega'$ of the boundary.
	In particular $z$ is locally injective on $\Omega'$, meaning point $p\in\partial\Sigma^2$ has a precompact neighborhood $\Omega_p\subset\Sigma^2$ which is a semi-disk on which $z:\Omega_p\rightarrow{}z(\Omega_p)$ is injective.
	
	We shall create a subset $\Omega\subset\Omega'$ on which $z$ is injective, by piecing together refinements of the neighborhoods $\Omega_p$.
	This will be tied to an exhaustion of $\partial\Sigma^2$, where $z$ is already known to be injective.
	To build the exhaustion of $\partial\Sigma^2$, cover $\partial\Sigma^2$ with countably many of the semi-disks $\{\Omega_{p_i}\}_i$ in such a way that any compact subset of $\partial\Sigma^2$ intersects just finitely many of these semi-disks.
	Set $L_i=\bigcup_{j=1}^i\Omega_{p_j}\cap\partial\Sigma^2$.
	Then $L_1\subset{}L_2\subset\dots$ is our exhaustion of the boundary.
	
	To create $\Omega$, first set $\Omega_1=\Omega_{p_1}$.
	For an induction argument, assume nested open sets $\Omega_{1}\subset\dots\subset\Omega_i$ have been created so that $L_i\subset\Omega_i$ and so that $z:\Omega_i\rightarrow\overline{H}{}^2$ is injective.
	To create $\Omega_{i+1}$, first set $\Omega'_{i+1}=\Omega_i\cup\Omega_{p_i}$
	Certainly $L_{i+1}\subset\Omega'_{i+1}$ and $z$ is locally injective on $\Omega_{i+1}$, but possibly it is no longer globally injective.
	To fix this, set $\Omega_{i+1}=\Omega_i\cup\left(\Omega_{p_i}\setminus{}z^{-1}\big(\overline{z(\Omega_i)}\big)\right)$.
	Now $z:\Omega_{i+1}\rightarrow\mathbb{C}$ is certainly injective and we retain the nesting $\Omega_i\subseteq\Omega_{i+1}$.	
	However, we might have removed too much: $\Omega_{i+1}$ is still open, but we might have removed points of $L_{i+1}$.
	
	To rule this out, we use the injectivity of $z$ on $\partial\Sigma^2$ itself.
	If some point $p\in{}L_{i+1}\setminus\overline{L_i}$ was removed, this means $z(p)\in\overline{z(\Omega_i)}$.
	By continuity and the fact that $z^{-1}(\{y=0\})\subset\partial\Sigma^2$ (which is condition (\textit{ii}) above), necessarily $z(p)\in{}z(L_i)$.
	But $z$ is injective along the boundary, so $p\in{}L_i$, contradicting the fact that $p\in{}L_{i+1}\setminus\overline{L_i}$.
	Therefore the open set $\Omega_{i+1}$ still contains $L_{i+1}$.
	
	Because $z$ is injective on each $\Omega_i$, it is injective on the open set $\Omega=\bigcup_i\Omega_i$.
	Because $\Omega_i$ contains $L_i$, certainly $\Omega$ contains $\partial\Sigma^2$.
	This concludes the proof.
\end{proof}

\begin{lemma}[Piecewise linearity at the boundary] \label{LemmaPiecewise}
	Assume $(\Sigma^2,g_\Sigma)$ satisfies (A)-(F) of the introduction.
	The map $z:\partial\Sigma^2\rightarrow\{y=0\}$ is surjective.
	Both this map and its inverse, the ``outline map'' $(\varphi^1,\varphi^2):\{y=0\}\rightarrow\partial\Sigma^2$, are piecewise linear with finitely many Lipschitz points.
\end{lemma}
\begin{proof}
	Pushing forward the momentum functions $\varphi^1$, $\varphi^2$ along the injective map $z:\Omega\rightarrow{}z(\Omega)$, the equation $d(\mathcal{V}^{-1/2}J_\Sigma{}d\varphi^i)=0$ of Proposition \ref{PropDivOfVarphi} becomes $y(\varphi^i_{xx}+\varphi^i_{yy})-\varphi^i_y=0$.
	By Hypotheses (C) and (D), at or near segments the metric and momentum functions are smooth.
	Then, because $y=0$ on $\partial\Sigma^2$, we have $\varphi^i_y=0$ along the image of any segment in $\overline{H}{}^2$.
	Then near $y=0$ the expression $\frac1y\varphi^i_y$ is a difference quotient, so by smoothness we have $\lim_{y\searrow0}\frac1y\varphi^i_y=\varphi^i_{yy}$.
	From the equation $\varphi^i_{xx}+\varphi^i_{yy}-\frac1y\varphi^i_y=0$ we obtain
	\begin{eqnarray}
		0
		\;=\;\lim_{y\searrow0}
		\left(
		\varphi^i_{xx}+\varphi^i_{yy}-\frac{1}{y}\varphi^i_y\right)
		\;=\;\lim_{y\searrow0}\varphi^i_{xx}
		\;=\;\varphi^i_{xx}(x,0) \label{EqnDoubleDerivAlongBoundary}
	\end{eqnarray}
	at boundary segments---therefore along the segment $l_i$ there are constants $c_i$, $d_i$ so that $\varphi^i(x,0)=c_ix+d_i$.
	At corners this no longer holds, but we still have that the $\varphi^i$ are continuous.
	Therefore $x\mapsto\varphi^i(x,0)$ is piecewise linear with finitely many Lipschitz points.
	
	Because $(\varphi^1,\varphi^2):\{y=0\}\rightarrow\partial\Sigma^2$ is piecewise linear, its inverse $z:\partial\Sigma^2\rightarrow\{y=0\}$ is piecewise linear and also has finitely many Lipschitz points.
	Because $z:\partial\Sigma^2\rightarrow\{y=0\}$ is both injective and piecewise linear with finitely many Lipschitz points, it is surjective.
\end{proof}
\begin{lemma}[The Bijectivity Zone for $z$] \label{LemmaBijectiveNearBoundary}
	Assume $(\Sigma^2,g_\Sigma)$ obeys (A)-(E).
	Then a neighborhood $\Omega$ of $\partial\Sigma^2$ exists where $z:\Omega\rightarrow{}z(\Omega)$ is bijective, and $z(\Omega)$ is a neighborhood of $\{y=0\}$.
	At the boundary, both $z:\partial\Sigma^2\rightarrow\{y=0\}$  and its inverse, the ``outline map'' $(\varphi^1,\varphi^2):\{y=0\}\rightarrow\partial\Sigma^2$, are bijective and piecewise linear.
\end{lemma}
\begin{proof}
	The first statement is just Lemma \ref{LemmaPiecewise}.
	For the second statement, Lemma \ref{LemmaPreBijectiveZone} says there is a neighborhood $\Omega$ of $\partial\Sigma^2$ on which $z:\Omega\rightarrow\overline{H}{}^2$ is a bijection onto its image.
	Lemma \ref{LemmaPiecewise} says this image contains all of $\{y=0\}$, so $z:\Omega\rightarrow{}z(\Omega)$ is a bijection between a neighborhood of $\partial\Sigma$ and a neighborhood of $\{y=0\}$.
\end{proof}

\subsection{The bijectivity zone near $\{y=0\}$.} \label{SubSecSecondBijZone}

The map $z:\Omega\rightarrow{}z(\Omega)$ is a bijection onto its image, as we now know.
But $z^{-1}$ restricted to $z(\Omega)$ is not necessarily a bijection because it might not be single-valued; see Figure \ref{FigBijZone}.
We must prove that nothing {\it except} the region $\Omega$ maps to $z(\Omega)$.
This is done using the first of our two ``one component lemmas,'' Lemma \ref{LemmaBoxLemma}, which states that if a domain $\mathcal{R}$ intersects $\partial\Sigma^2$, then $z^{-1}(\mathcal{R})$ has only one component.
\begin{lemma}[The One-Component Lemma] \label{LemmaBoxLemma}
	Assume $\Sigma^2$ satisfies hypotheses (A)-(F).
	Consider any domain $\mathcal{R}\subset\overline{H}{}^2$ that satisfies conditions (1)-(3) from \S\ref{SubsecTheBarriers}.
	Then the pre-image $z^{-1}\left(\mathcal{R}\right)\subset\Sigma^2$ has exactly one component.
\end{lemma}
\begin{proof}
	Let $\Omega\subset\Sigma^2$ be the ``bijectivity zone'' of $z$, the neighborhood of $\partial\Sigma^2$ guaranteed by Lemma \ref{LemmaBijectiveNearBoundary} for which $z:\Omega\rightarrow{}z(\Omega)$ is one-one and onto.
	By the open mapping theorem, $\partial(z(\mathcal{R}))\subseteq{}z(\partial{}\mathcal{R})$ and $z(Int(\mathcal{R}))\subseteq{}Int(z(\mathcal{R}))$.
	\noindent\begin{figure}[h!] 
		\centering
		\includegraphics[scale={0.375}]{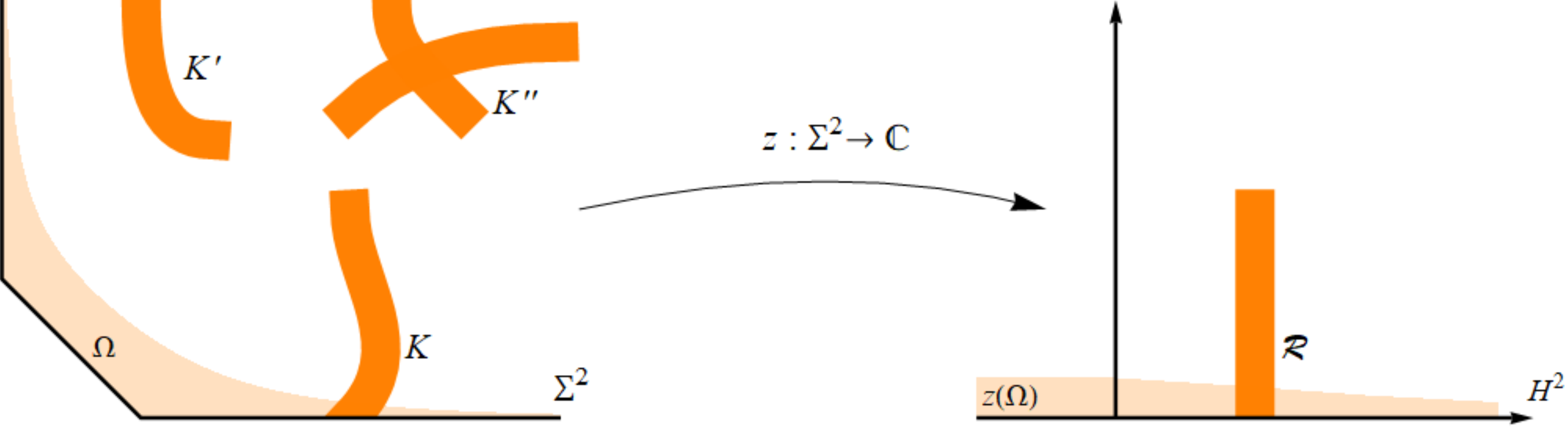}
		\caption{A domain $\mathcal{R}\subset\overline{H}{}^2$ and a possible pre-image with multiple components.
		Exactly one pre-image component $K$ intersects $\partial\Sigma^2$, so $\Psi=\psi_\mathcal{R}\circ{}z$ takes zero boundary values on all others.
		} \label{FigPreIm}
	\end{figure}
	
	For a proof by contradiction, assume two or more distinct components $z^{-1}(\mathcal{R})$ exist, which we call $K$, $K'$, $K''$, $\dots$ (see Figure \ref{FigPreIm}).
	Because $z:\Omega\rightarrow{}z(\Omega)$ is a bijection, exactly one of these components can intersect $\Omega$; we call this component $K$.
	We have the function $\psi_{\mathcal{R}}$ of Lemma \ref{LemmaBarrierR}, so on $z^{-1}(\mathcal{R})$ we have
	\begin{equation}
		\Psi\;=\;\psi_{\mathcal{R}}\circ{}z.
	\end{equation}
	Because the function $\psi_{\mathcal{R}}$ satisfies $y(\psi_{xx}+\psi_{yy})-\psi_y=0$ on $\mathcal{R}\subset\mathbb{C}$, the function $\Psi$ satisfies the elliptic equation $d(\mathcal{V}^{-\frac12}J_\Sigma{}d\Psi)=0$ on $z^{-1}(\mathcal{R})$.
	On $z^{-1}(\mathcal{R})$, the functions $\Psi$, $\varphi^1$, $\varphi^2$ satisfy the degenerate elliptic equation from Proposition \ref{PropDivOfVarphi}:
	\begin{eqnarray}
		\begin{aligned}
			&d\left(\mathcal{V}^{-\frac12}J_\Sigma{}d\Psi\right)=0
			\quad\text{and}\quad
			d\left(\mathcal{V}^{-\frac12}J_\Sigma{}d\varphi^i\right)=0.
		\end{aligned} \label{EqnsEllipticBarrierMoments}
	\end{eqnarray}
	Finally, because none of the components except $K$ intersect the boundary $\partial\Sigma^2$, on each component $K',K'',\dots$ the function $\Psi$ has zero boundary values.
	
	Now we can explain the idea of the proof.
	When $\Sigma^2$ is not a half-plane, we may assume $\Sigma^2$ lies within the quarter-plane (eg. Figures \ref{FigTwoTransforms} or \ref{FigLeBruns}), and therefore
	\begin{equation}
		\varphi\;=\;\varphi^1+\varphi^2. \label{EqnNewMomComb}
	\end{equation}
	is non-negative and all of its sub-levelsets $\{\varphi\le{}c\}$ are compact.
	Let $K'$ be any component of $z^{-1}(\mathcal{R})$ except $K$.
	Then on $K'$, has zero boundary values and $0\le\Psi<1$, whereas $\varphi$ has \textit{positive} boundary values and gets unboundedly large far away.
	Thus the maximum principle shows $\varphi$ dominates not only $\Psi$, but any multiple $C\Psi$ of $\Psi$.
	Clearly this is impossible, so the component $K'$ does not exist.
	
	When $\Sigma^2$ is a half-plane, a non-negative momentum function such as (\ref{EqnNewMomComb}) with compact sub-levelsets does not exist, so we make a different argument specially adapted to that case. \\
	
	\noindent\underline{\it Argument that $K'$ does not exist, in the case $\Sigma^2$ is not a half-plane.}
	
	After an affine transformation of the plane we can assume $\Sigma^2$ is in the first quadrant so that the function $\varphi$ of (\ref{EqnNewMomComb}) is non-negative and has compact sub-levelsets.
	For an open-closed argument, let $\mathcal{S}\subset[0,\infty)$ be the set of values for which $\varphi$ dominates $C\Psi$:
	\begin{eqnarray}
		\mathcal{S}\;=\;
		\big\{\,
		C\in[0,\infty)\;\big|\; C\Psi\,<\,\varphi\;\,\text{on}\,\;K'
		\,\big\}.
	\end{eqnarray}
	We shall prove that $\mathcal{S}$ is open, closed, and contains $0$.
	Then because $\mathcal{S}=[0,\infty)$, we have $\infty<\varphi(p)$ on every $p\in{}K'$, a contradiction.
		
	\;{\it{i}}) \underline{$\mathcal{S}\neq\varnothing$}. Because $0<\varphi$ except at $(0,0)$ which is not in $K'$, $0\in\mathcal{S}$.
	
	{\it{ii}}) \underline{$\mathcal{S}$ is open}. Assume $C\in\mathcal{S}$ and let $\{C_i\}$ be a sequence with $C_i\rightarrow{}C$.
	To prove openness, we show that eventually $C_i\in\mathcal{S}$ for large $i$.
	
	Passing to a subsequence, assume $C_i<C+1$.
	If $C_i\notin\mathcal{S}$ then by definition there is at least one point $p_i\in{}K'$ with $C_i\Psi(p_i)\ge\varphi(p_i)$.
	If the set $\{p_i\}\subset\Sigma^2$ has any cluster points, say $p_\infty$ is a cluster point, then by continuity we must have $C\Psi(p_\infty)\ge\varphi(p_\infty)$, an impossibility because we assumed $C\in\mathcal{S}$.
	Thus $\{p_i\}$ has no cluster points, which means that $p_i$ eventually leaves {\it every} compact set, including the compact set $\{\varphi\le{}C+1\}$.
	This is impossible because we chose the sequence $\{p_i\}$ so that $\varphi(p_i)\le{}C_i\Psi(p_i)<C+1$.
	Therefore no such sequence $\{p_i\}$ exists, and we conclude $C_i\in\mathcal{S}$ for sufficiently large $i$.
	
	{\it{iii}}) \underline{$\mathcal{S}$ is closed}. Assume $C_i\in\mathcal{S}$ and $C_i\rightarrow{}C$; we show $C\in\mathcal{S}$.
	Because $C_i\in\mathcal{S}$ by definition $C_i\Psi<\varphi$ on $K'$, so in the limit $C\Psi\le\varphi$ on $K'$.
	But this inequality is strict on $\partial{}K'$, so by the maximum principle $C\Psi<\varphi$ on $K'$.
	Thus $C\in\mathcal{S}$. \\
	
	\noindent\underline{\it Argument that $K'$ does not exist, in the case $\Sigma^2$ is a half-plane.}
	
	After possible affine recombination of the functions $\varphi^1$, $\varphi^2$, we may assume the polygon is the upper half-plane $\Sigma^2=\{\varphi^2\ge0\}$.
	Again take $\Psi=\psi_{\mathcal{R}}\circ{z}$ and restrict the domain to $K'$.
	As before, $0\le\Psi<1$ on $K'$ and we have boundary values $\Psi=0$ on $\partial{}K'$.
	Let $\mathcal{S}$ be
	\begin{eqnarray}
		\mathcal{S}\;=\;\left\{\,
		C\ge0\;\big|\;C\Psi<{}\varphi^2\;on\;K'\,
		\right\}
	\end{eqnarray}
	Certainly $0\in\mathcal{S}$.
	The argument that $\mathcal{S}$ is closed is {\it precisely} the same as it was above except with $\varphi^2$ in place of $\varphi$---this argument relies only on the maximum principle and that $\varphi^2>0$ on $K'$.
	
	To establish the openness of $\mathcal{S}$ assume $C$ is in the closure of $\mathcal{S}$ and let $\{C_i\}$ be any sequence with $C_i\rightarrow{C}$.
	Certainly we may assume $\max_i{}C_i<C+1$.
	We must show that eventually $C_i\in\mathcal{S}$.
	
	\noindent\begin{figure}[h]
		\centering
		\includegraphics[scale={0.45}]{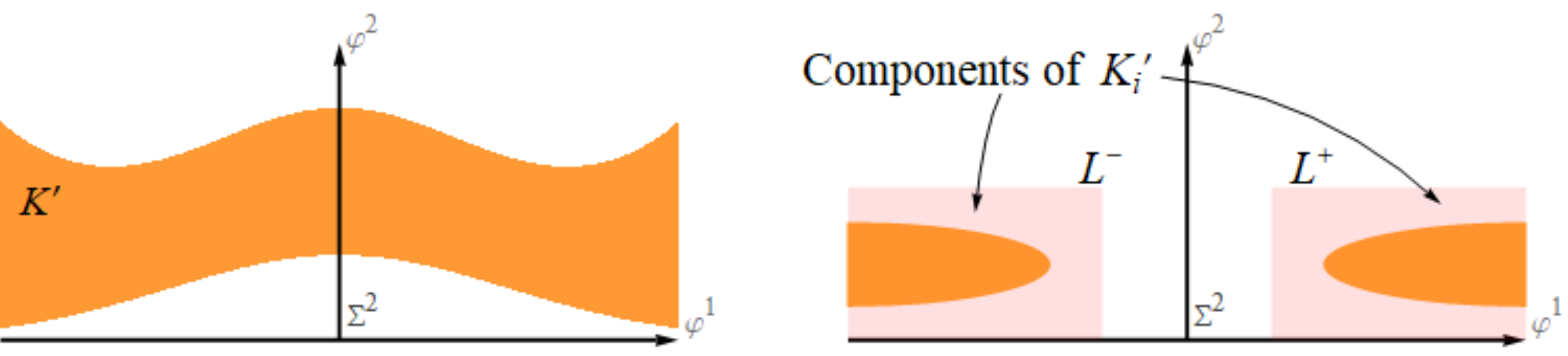}
		\caption{
			On the left, the domain $K'$ of $\Psi$.
			On the right, the domain $K_i'$ of $\Psi_i$ which, for large $i$, must lie within $L^{-}\cup{}L^{+}$.
		} \label{FigKAndHalfStrips}
	\end{figure}
	
	For each $i$, create a new function
	\begin{eqnarray}
		\begin{aligned}
			&\Psi_i\;\triangleq\;C_i\Psi\,-\,\varphi^2 \quad \text{on the domain} \\
			&K_i'\;\triangleq\;\{\,\Psi_i\;\ge\;0\,\}  \;\subseteq\;K'
		\end{aligned}
	\end{eqnarray}
	so that $K'_i$ is the set on which $C_i\Psi$ dominates $\varphi^2$.
	Because $C_i\notin\mathcal{S}$, we know $C_i\Psi$ dominates $\varphi^2$ somewhere so there is at least one point $p_i\in{}K'_i$.
	Each of these $p_i$ is in the strip $0<\varphi^2<C+1$ because $\varphi^2(p_i)\le{}C_i\Psi(p_i)<C+1$.
	Next, $p_i$ has no cluster points, for if $p_\infty$ were a cluster point then by continuity $C\Psi(p_\infty)\ge\varphi^2(p_\infty)$, contradicting $C\in\mathcal{S}$.
	Because there are no cluster points, any such sequence $\{p_i\}\subset{}K'\subset\Sigma^2$ eventually leaves all compact sets and in particular leaves the rectangular region $\{\varphi^1\in[-1,1],\,\varphi^2\in[0,C+1]\}$.
	Thus, as in Figure \ref{FigKAndHalfStrips}, $K'_i$ is eventually in the union of the two half-strips
	\begin{eqnarray}
		\begin{aligned}
		&L^+\;=\;\{\varphi^2\in(0,C+1),\,\varphi^1\;>\;1\;\}, \\
		&L^-\;=\;\{\varphi^2\in(0,C+1),\,\varphi^1\;>\;-1\;\}.
		\end{aligned}
	\end{eqnarray}
	
	To finish the argument we use the other momentum coordinate $\varphi^1$ as a barrier over top of $L^+$ and $-\varphi^1$ as a barrier over $L^-$.
	Certainly $\Psi_i|_{\partial{}K'_i\cap{}L^+}=0$ and $0\le\Psi_i<\Psi<1$.
	Given $\epsilon>0$, we have that $\epsilon\varphi^1|_{\partial{}K'_i\cap{}L^+}>0$ and that $\epsilon\varphi^1$ gets unboundedly large far away, so by the maximum principle $\epsilon\varphi^1$ bounds $\Psi_i$ from above on $K'_i\cap{}L^+$.
	Similarly $-\epsilon\varphi^1$ always bounds $\Psi_i$ from above on $K'_i\cap{}L^-$.
	Sending $\epsilon\searrow0$ shows that $\Psi_i\equiv0$, so $K'_i$ is empty.
	We conclude that $C_i\in\mathcal{S}$.
\end{proof}

\begin{lemma}[The Bijectivity Zone for $z^{-1}$] \label{LemmaCloseBijectivityZInv}
	Let $\Omega$ be the neighborhood of $\partial\Sigma^2$ on which $z:\Omega\rightarrow{}z(\Omega)$ is a bijection.
	Then $z^{-1}$ is single-valued on $z(\Omega)$.
	
	Consequently $z:\Omega\rightarrow{}z(\Omega)$ and $z^{-1}:z(\Omega)\rightarrow\Omega$ are biholomorphisms between a neighborhood $\Omega$ of $\partial\Sigma^2$ and a neighborhood $z(\Omega)$ of $\partial\overline{H}{}^2$.
\end{lemma}
\begin{proof}
	Shrinking $\Omega$ if necessary, we may assume $\Omega$ and $z(\Omega)$ are not only open but connected.
	Picking $z_0\in{}z(\Omega)$, we must show that $z^{-1}(z_0)$ consists of a single point.
	Let $\gamma$ be a path in $z(\Omega)$ from $z_0$ to $\{y=0\}$ and let $\mathcal{R}$ be a neighborhood of $\gamma$; shrinking $\mathcal{R}$ if necessary we may assume $\overline{\mathcal{R}}\subset{}z(\Omega)$.
	Because $z:\Omega\rightarrow{}z(\Omega)$ is injective, certainly $z^{-1}(\mathcal{R})$ has at least one component $K$ and $K\subset\Omega$.
	By the One-Component Lemma, this is the only component.
	Now $z^{-1}(\mathcal{R})$ has a single component, which is inside the ``bijectivity zone'' of Lemma \ref{LemmaPreBijectiveZone}, meaning $z:z^{-1}(\mathcal{R})\rightarrow\mathcal{R}$ is injective.
	This the pre-image of $z_0$ is unique.
\end{proof}

\subsection{Global injectivity} \label{SubsecInjectivity}

Global injectivity of $z:\Sigma^2\rightarrow\overline{H}{}^2$ essentially follows from the One-Component Lemma, although the possible presence of ramification points complicates the argument.
We did not have to worry about this in the case of the neighborhood $\Omega$ of $\partial\Sigma^2$ because we selected it specifically to be a neighborhood where $|dy|>0$, as is guaranteed by the Hopf Lemma.

\begin{lemma} \label{LemmaRamificationNoAccum}
	The set of ramification points of the complex variable $z:\Sigma^2\rightarrow\overline{H}{}^2$ has no accumulation points.
\end{lemma}
\begin{proof}
	By classical analytic continuation there is no {\it interior} accumulation of ramification points.
	The remaining possibility is that ramification points accumulate near $\partial\Sigma^2$.
	But this is impossible by Lemma \ref{LemmaBijectiveNearBoundary}.
\end{proof}

\begin{proof}[Proof of Proposition \ref{ThmInjectivityOfZ}, global injectivity of $z$.]
	For an argument by contradiction, suppose distinct points $p_0$, $q_0$ have $z(p_0)=z(q_0)$.
	Possibly $p_0$ or $q_0$ is a ramification point, but if so, we can slightly adjust their locations to ensure neither is a ramification point and still retain $z(p_0)=z(q_0)$.
	
	Draw a path $\gamma:[0,1]\rightarrow\Sigma^2$ from any point in $\partial\Sigma^2$ to $p_0$; because ramification points are sparse (Lemma \ref{LemmaRamificationNoAccum}) we can avoid them.
	The image path $z\circ\gamma:[0,1]\rightarrow\overline{H}{}^2$ goes from the boundary $\{y=0\}$ to the common location $z(p_0)=z(q_0)$.
	Even though $z\circ\gamma$ contains no ramification points, it might still self-intersect.
	The first step is to improve the choice of the path $\gamma$, so that $z\circ\gamma$ does not self-intersect.
	
	Because $z^{-1}$ is single-valued on $z(\Omega)$, for some small $\epsilon>0$ the restricted path $z\circ\gamma|_{[0,\epsilon)}$ intersects no other part of the path $z\circ\gamma|_{[0,1]}$.	
	Then let $T_1\in(\epsilon,1)$ be the first value where $z\circ\gamma\big|_{[0,T_1]}$ ceases to be non-self intersecting; in particular there is some strictly smaller $T_0\in[\epsilon,T_1)$ for which $(z\circ\gamma)(T_1)=(z\circ\gamma)(T_0)$.
	We have now found new points $p_1=\gamma(T_0)\neq{}\gamma(T_1)=q_1$ with $z(p_1)=(z\circ\gamma)(T_0)=(z\circ\gamma)(T_1)=z(q_1)$.
	The new path $\gamma:[0,T_0]\rightarrow\Sigma^2$ is non-self-intersecting (because $T_0<T_1$), but its terminal point $p_1=\gamma(T_0)$ retains the property that $z(p_1)=z(q_1)$ for some $q_1\ne{}p_1$.
	
	\noindent
	\begin{figure}[h!] 
		\vspace{-0.1in}
		\includegraphics[scale={0.5}]{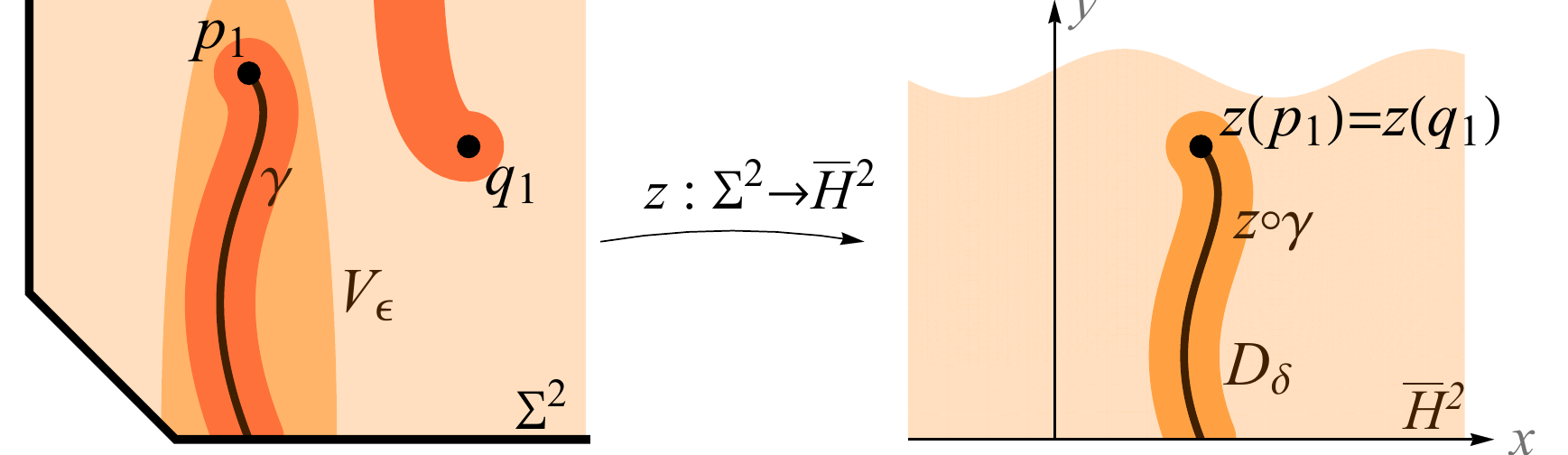}
		\caption{Illustration for the proof of Lemma \ref{ThmInjectivityOfZ}.
			If $z(p_1)=z(q_1)$, then drawing $D_\delta$ around $z\circ\gamma$ for sufficiently small $\delta$, we can separate its pre-image $z^{-1}(D_\delta)$ into distinct components, one each containing $p_1$ and $q_1$.
			This violates the One-Component Lemma.
		}
		\label{FigInjProof}
	\end{figure}
	
	Next we create a tiny neighborhood $D_\delta$ of the path $z\circ\gamma$, with the aim of using the one-component lemma.
	Let $D_\delta\subset\overline{H}{}^2$ be the neighborhood
	\begin{eqnarray}
		D_\delta
		\;=\;
		\{z'\in\overline{H}{}^2
		\;\big|\;
		\text{some $t\in[0,T_0]$ exists so } |(z\circ\gamma)(t)-z'|<\delta
		\}.
	\end{eqnarray}
	That is, $D_\delta$ is the $\delta$-neighborhood of the path $z\circ\gamma$, as measured in the complex coordinate.	
	Certainly $D_\delta$ has compact closure.
	Also $\bigcap_{\delta>0}D_\delta$ is just the path itself.
	
	Next, in the $\Sigma^2$ polygon, let $V_\epsilon$ be the $\epsilon$-neighborhood, as measured in the $g_\Sigma$-metric, around $\gamma$.
	Because $\gamma$ intersects no ramification points, we can choose choose $\epsilon>0$ so small that $z:V_\epsilon\rightarrow\overline{H}{}^2$ also contains no ramification points (by Lemma \ref{LemmaRamificationNoAccum}).
	Because $p_1\ne{}q_1$, we may choose $\epsilon$ so small that $p_1\in{}V_\epsilon$ but $q_1\notin{}V_\epsilon$.
	
	Given $\epsilon>0$ there exists $\delta>0$ so that $D_\delta\subset{}z(V_\epsilon)$; this is because $z(V_\epsilon)$ is precompact (by continuity) and contains $z\circ\gamma$ whereas $\bigcap_{\delta>0}D_\delta={}z\circ\gamma$; see Figure \ref{FigInjProof}.
	Equivalently, at least one component of $z^{-1}(D_\delta)$ lies within $V_\epsilon$.	
	By the One-Component Lemma there is exactly one component of $z^{-1}(D_\delta)$, and as we have just seen this lies in $V_\epsilon$.
	But $q_1\notin{}V_\epsilon$, contradicting $z(q_1)=z(p_1)$.
\end{proof}

\subsection{Global surjectivity} \label{SubsecSurjectivity}

The proof of surjectivity is expressed in Figure \ref{FigSurjProof}.
The idea is that if $z_0$ is any point {\it not} in the image of $z:\Sigma^2\rightarrow\overline{H}{}^2$, using the second of our ``one-component'' lemmas, Lemma \ref{LemmaTwoBoxLemma}, we are able to draw a circle around $z_0$ consisting of points that {\it are} in the image of $z$.
Then by Lemma \ref{LemmaDiskAroundPBPoint}, the ``disk lemma,'' $z_0$ is also in the image of $z$.

\begin{lemma}[The One-Component Lemma for $\mathcal{D}_{y'}$] \label{LemmaTwoBoxLemma}
	Assume $\Sigma^2$ satisfies hypotheses (A)-(F) and let $\mathcal{D}_{y'}$ be any domain of the form (\ref{DomainTwoIntsWithBoundary}) (see Figure \ref{SubFigRectDom}).
	
	Then the pre-image $z^{-1}\left(\mathcal{D}_{y'}\right)\subset\Sigma^2$ has exactly one component.
\end{lemma}
\textbf{Remark.} Referring to Figure \ref{FigSurjProof}, this lemma rules out the third picture.
\begin{proof}
	This proof is essentially the same as the proof of Lemma \ref{LemmaBoxLemma}, so we just give the outline emphasizing the differences.
	Consider the function $\psi^2_{\mathcal{D}_{y'}}$ that is $1$ on the boundary segment $l_2=\{x\in(1,2),\,y=0\}$ and zero on all other boundary points.
	Set $\Psi=\psi^2_{\mathcal{D}_{y'}}\circ{}z$ so that $\Psi:z^{-1}(\mathcal{D}_{y'})\rightarrow[0,1]$.

	Letting $\Omega\subset\Sigma^2$ be the ``bijectivity zone'' where $z:\Omega\rightarrow{}z(\Omega)$ and its inverse are bijections between a neighborhood of $\partial\Sigma^2$ and a neighborhood of $\{y=0\}$, we certainly have at least one component $K$ of $z^{-1}(D_{y'})$ so that $z(K)$ contains the segment $l_2$.
	
	If there is some other component $K'\subset{}z^{-1}(\mathcal{D}_{y'})$, then restricting $\Psi$ to $K'$, we have $\Psi=0$ on $\partial{}K'$, $0\le\Psi<1$ on $K'$, and $\Psi$ satisfies the elliptic equation of (\ref{EqnsEllipticBarrierMoments}).
	Now we are \textit{precisely} in the situation of the proof of Lemma \ref{LemmaBoxLemma} directly after (\ref{EqnsEllipticBarrierMoments}).
	One proceeds, word for word, from that lemma.
\end{proof}

\noindent
\begin{figure}[h!] 
	\vspace{-0.30in}
	\includegraphics[scale={0.39}]{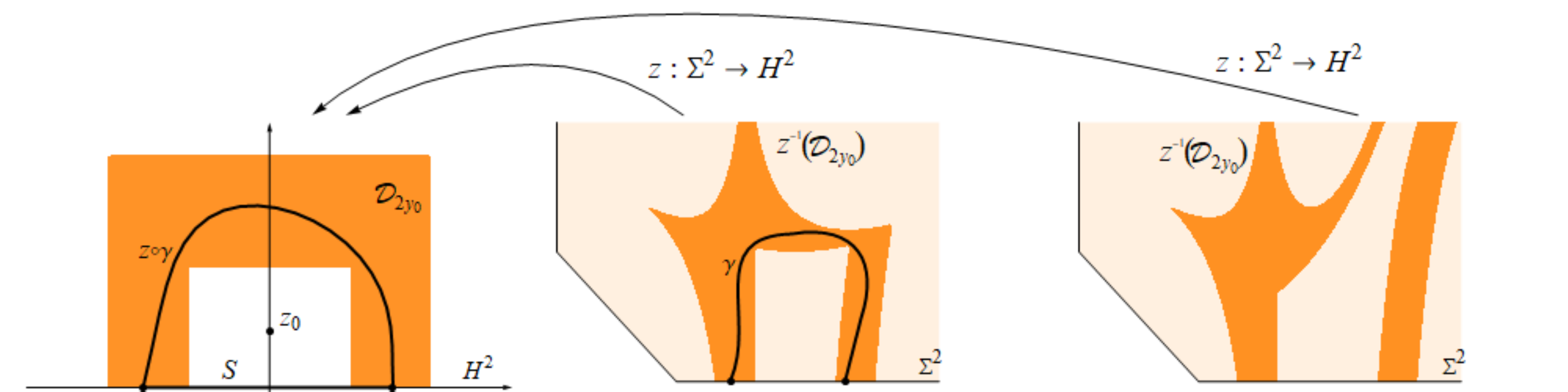}
	\caption{Illustration for the proofs of Lemma \ref{LemmaTwoBoxLemma} and Proposition \ref{ThmSurjectivityOfZ}.
		Lemma \ref{LemmaTwoBoxLemma} rules out the third picture, where the pre-image $z^{-1}(\mathcal{D}_{2y_0})$ is disconnected.
		This means that a loop can be drawn within the pre-image.
		The ``Disk Lemma'' then ensures that all points within this loop are also in the pre-image, completing the proof of Proposition \ref{ThmSurjectivityOfZ}, surjectivity of $z$.
	}
	\label{FigSurjProof}
\end{figure}

\begin{lemma}[The Disk Lemma] \label{LemmaDiskAroundPBPoint}
	Let $D\subset\overline{H}{}^2$ be any precompact domain homeomorphic to the open disk with $C^{0,1}$ boundary.
	
	If $\partial{D}$ lies within the image $z(\Sigma^2)$, then $D$ lies within the image $z(\Sigma^2)$.
\end{lemma}
\begin{proof}
	Certainly $\partial{D}$ is homeomorphic to a circle.
	Because $z$ is one-to-one and $\partial{D}\subset{}z(\Sigma^2)$, $z^{-1}(\partial{D})$ is also homeomorphic to a circle and in particular is compact.
	The open mapping theorem and continuity of $z$ provides two further facts: first that $\partial\left(z^{-1}(D)\right)\subseteq{}z^{-1}(\partial{}D)$, and second that $D\cap{}z(\Sigma^2)$ open set.
	
	For the contradiction we assume $\partial{}D\subset{}z(\Sigma^2)$, but there are points in $D$ that are not in $z(\Sigma^2)$.
	Because $D\cap{}z(\Sigma^2)$ is open we can find points $z_i=z(p_i)$ in $D\cap{}z(\Sigma^2)$ converging to some $z_\infty\notin{}z(\Sigma^2)$.
	If the sequence $p_i\in\Sigma^2$ has a cluster point $p_\infty$ then by continuity $z(p_i)\rightarrow{}z_\infty=z(p_\infty)$, contradiction $z_\infty\notin{}z(\Sigma^2)$.
	We conclude that $p_i$ is a divergence sequence.
	Therefore the pre-image $z^{-1}(D)$ is unbounded (meaning it is not contained in any compact subregion).
	
	We can show its complement $\Sigma^2\setminus{}z^{-1}(D)$ is also unbounded.
	The ``bijectivity zone'' of Lemma \ref{LemmaCloseBijectivityZInv} asserts $z:\Omega\rightarrow{}z(\Omega)$ is a bijection.
	Because $D$ and therefore $D\cap{}z(\Omega)$ is precompact, we have that $z^{-1}(D)\cap\Omega$ is precompact by bijectivity.
	Therefore its complement $\Omega\setminus{}z^{-1}(D)$ is unbounded, so $\Sigma^2\setminus{}z^{-1}(D)$ is unbounded.
	
	Both $z^{-1}(D)$ and its complement $\Sigma^2\setminus{}z^{-1}(D)$ are unbounded, and so the boundary $\partial(z^{-1}(D))$ is unbounded.
	But $\partial(z^{-1}(D))\subset{}z^{-1}(\partial{}D)$, so by continuity $\partial{}D$ is unbounded, contradicting the fact that $\partial{}D$ is $C^{0,1}$-equivalent to a circle.
\end{proof}

\begin{proof}[Proof of Proposition \ref{ThmSurjectivityOfZ}, surjectivity of $z$.]
	Let $z_0$ be any point in $H^2$.
	After possibly translating in the $x$-direction we can assume $z_0$ has coordinates $z_0=(0,y_0)$ for some $y_0>0$.
	Then consider the region $\mathcal{D}_{2y_0}$ (discussed in Subsection \ref{SubsecTheBarriers}), given by
	\begin{eqnarray}
		\mathcal{D}_{2y_0}
		\;=\;\mathcal{R}_{2,4y_0}\,\setminus\,\mathcal{R}_{1,2y_0}.
	\end{eqnarray}
	This is depicted in the first image of Figure \ref{FigSurjProof}.
	By construction, $z_0\notin\mathcal{D}_{2y_0}$.
	By Lemma \ref{LemmaTwoBoxLemma}, $z^{-1}(\mathcal{D}_{2y_0})$ has a single component (the second image of Figure \ref{FigSurjProof}).
	
	As $z^{-1}(\mathcal{D}_{2y_0})$ has just one component by Lemma \ref{LemmaTwoBoxLemma}, and because it intersects $\partial\Sigma^2$ in two locations---the line segments $l_1$ and $l_2$--- we can draw a path $\gamma:[0,1]\rightarrow{}z^{-1}(\mathcal{D}_{2y_0})$ so that $\gamma(0)\in{}l_1$ and $\gamma(1)\in{}l_2$.
	Then the path $(z\circ\gamma):[0,1]\rightarrow{}D$ has $(z\circ\gamma)(0)$ in the line segment $\{x\in[-1,-\frac12],y=0\}$ and $(z\circ\gamma)(1)$ in the line segment $\{x\in[-1,-\frac12],y=0\}$.
	Closing the path by letting $S$ be the line segment in $\partial{}H^2$ connecting $(z\circ\gamma)(0)$ and $(z\circ\gamma)(1)$, the set
	\begin{eqnarray}
		C\;=\;(z\circ\gamma)[0,1]\,\cup\,S
	\end{eqnarray}
	is topologically a circle; this is depicted in the first and second images of Figure \ref{FigSurjProof}.
	
	By construction the circle $C$ lies entirely within $z(\Sigma^2)$.
	It therefore bounds a disk which contains $z_0$.
	Thus Lemma \ref{LemmaDiskAroundPBPoint} says $z_0\in{}z(\Sigma^2)$.
\end{proof}

\section{Classification of polygon metrics} \label{SecAnalytic}

The classification works by an ``outline matching'' process, using the labels $s_i$ on the boundary segments $l_i\subset\partial\Sigma^2$ to create an ``outline map'' $\Psi:\{y=0\}\rightarrow\partial\Sigma^2$ that is piecewise linear on each segment.
Then we use a certain repository of known momentum functions that is rich enough to reproduce any such outline.
We compare these constructed functions $\tilde\varphi{}^1$ and $\tilde\varphi{}^2$ to the already-existing momentum functions $(\varphi^1,\varphi^2)$ on $\Sigma^2$, and find that they are identical on the boundary.
The functions $\varphi^1-\tilde\varphi{}^1$ and $\varphi^2-\tilde\varphi{}^2$ are zero on $\{y=0\}$.
Then the Liouville Theorem from \cite{Web2} says $\varphi^i$ differs from the constructed solution $\tilde{\varphi}^i$ by at worst a multiple of $y^2$, completing the classification.

First we must clarify the relationship between the labels and the momentum functions.
Along $\{y=0\}$ the map $x\mapsto(\varphi^1(x,0),\varphi^2(x,0))$ is piecewise linear, so the direction vector $\frac{\partial}{\partial{}x}=\varphi^1_x\frac{\partial}{\partial\varphi^1}+\varphi^2_x\frac{\partial}{\partial\varphi^2}$ has constant length as measured in coordinates.
This is the boundary label: along the segment $l_i$
\begin{equation}
	s_i\;=\;\sqrt{\left(\varphi^1_x\right)^2+\left(\varphi^2_x\right)^2}. \label{EqnMarkings}
\end{equation}
This is the reason we sometime call the label the ``parameterization speed.''
See Figures \ref{FigGeneralCase}, \ref{FigParCase}, or \ref{FigHPCase}.
To see this is constant along $l_i$, because $\varphi^i_{xx}=0$ by (\ref{EqnDoubleDerivAlongBoundary}),
\begin{equation}
	\frac{\partial{}s_i}{\partial{}x}
	\;=\;\frac{\varphi^1_x\varphi^1_{xx}+\varphi^2_x\varphi^2_{xx}}{\sqrt{\left(\varphi^1_x\right)^2+\left(\varphi^2_x\right)^2}}
	\;=\;0.
\end{equation}
We show this ``parameterization speed'' is identical to our previous interpretation, equation (\ref{EqnInterpretationOfMarkings}), which interprets the labels in terms of the action fields on $M^4$.
After affine transformation of $\Sigma^2$, we may assume the segment $l_i$ lies along the $\varphi^1$-axis; then by definition $\varphi^1_x=0$, so that $s_i=|\varphi^1_x|$.
Let $\gamma(t)$ be a unit-speed geodesic with $\gamma(0)\in\{y=0\}$ that moves perpendicularly off the segment in the $y$-direction.
Near $t=0$, $|\mathcal{X}^1|=O(1)$, $|\mathcal{X}^2|=O(t)$, and $\left<\mathcal{X}^1,\mathcal{X}^2\right>=O(t^2)$; therefore $y=\sqrt{|\mathcal{X}^1|^2|\mathcal{X}^2|^2-\left<\mathcal{X}^1,\mathcal{X}^2\right>{}^2}=|\mathcal{X}^1||\mathcal{X}^2|+O(t^2)$.
Then
\begin{equation}
	\begin{aligned}
	s_i&\;=\;|\varphi^1_x|
	\;=\;\left|\left<\frac{\partial}{\partial{x}},\,\nabla\varphi^1\right>\right|
	\;=\;\left|\left<\frac{\partial}{\partial{x}},\,J_4\mathcal{X}^1\right>\right|
	\;=\;\left|\frac{\partial}{\partial{x}}\right|\left|\mathcal{X}^1\right|
	\end{aligned}
\end{equation}
where the last equality follows because $\mathcal{X}^1$ is a Killing field along a totally geodesic submanifold (the zero-set of the other Killing field $\mathcal{X}^2$), so its length only in the direction of the segment, so it is parallel to $\frac{\partial}{\partial{}x}$.
But $\left|\frac{\partial}{\partial{x}}\right|=\frac{1}{|\nabla{}x|}=\frac{1}{|\nabla{}y|}$, so
\begin{equation}
	\begin{aligned}
		s_i&
		\;=\;\frac{1}{|\nabla{}y|}\left|\mathcal{X}^1\right|
		\;=\;\frac{1}{\frac{\partial}{\partial{t}}\left(|\mathcal{X}^1||\mathcal{X}^2|\right)}\left|\mathcal{X}^1\right|
		\;=\;\frac{1}{\frac{\partial}{\partial{t}}|\mathcal{X}^2|}
	\end{aligned}
\end{equation}
where we used $y=|\mathcal{X}^1||\mathcal{X}^2|+O(t^2)$ and that $|\mathcal{X}^2|=O(t)$.
This recovers (\ref{EqnInterpretationOfMarkings}), our previous interpretation.

We indicate this section's milestones.
In Section \ref{SubsecOutlineMatching}, we match momentum functions to any outline, in a boundary-matching process essentially the same as that from \cite{AS}.
In Section \ref{SubsecImprLiouville} we create a slightly improved version of one of the Liouville theorems of \cite{Web2}; this is the essential result that allows for the classification.
In Sections \ref{SubsecClassGen}, \ref{SubsecClassPar}, and \ref{SubsecClassHP} we apply the Liouville theorem to classify all variations of the momentum functions found in Section \ref{SubsecOutlineMatching}.

In Section \ref{SubsecLipzchitzRemark} we make a technical comment concerning the possibility of Lipschitz points of $\varphi^1$, $\varphi^2$ that lie on an edge instead of a corner.
As expected, we find that Lipschitz points internal to an edge produces a curvature singularity, so create an unrealistic model of K\"ahler reduction.

\subsection{Outline matching} \label{SubsecOutlineMatching}

An ``outline map'' is any piecewise linear map from $\{y=0\}$ into the $\varphi^1$-$\varphi^2$ plane that has finitely many Lipshitz points.
In the following theorem we construct maps $(\varphi^1,\varphi^2):\overline{H}{}^2\rightarrow\Sigma^2$ that agree with the outline map when restricted to the boundary of the half-plane $\overline{H}{}^2$.
This theorem is basically due to Abreu and Sena-Dias \cite{AS} who did ``Case I.''

\begin{lemma} \label{LemmaOutlineLemma}
	Given a closed, non-compact labeled polygon $\Sigma^2$ with vertex points $p_1,\dots,p_d$ and labels $s_0,\dots,s_d$, let $\Phi:\{y=0\}\rightarrow\mathbb{R}^2$ be an outline map with parameterization speed $s_i$ along the edge $l_i$.
	
	Then there exist functions $\varphi^1,\varphi^2:\overline{H}{}^2\rightarrow\mathbb{R}$---which we construct explicitly in the proof---with the following properties:
	\begin{itemize}
		\item[{\it{i}})] Restricted to $\{y=0\}$, the map $(\varphi^1,\varphi^2):\{y=0\}\rightarrow\mathbb{R}^2$ is equal to $\Phi$.
		\item[{\it{ii}})] If the outline is the boundary of a closed, non-compact, convex polygon $\Sigma^2$, then the map $(\varphi^1,\varphi^2):\overline{H}{}^2\rightarrow\Sigma^2$ is bijective and non-singular.
		\item[{\it{iii}})] $y(\varphi^i_{xx}+\varphi^i_{yy})-\varphi^i_y=0$.
	\end{itemize}
\end{lemma}
\noindent\begin{figure}[h!] 
	\centering
	\vspace{-0.0in}
	\hspace{-0.5in}
	\includegraphics[scale={0.5}]{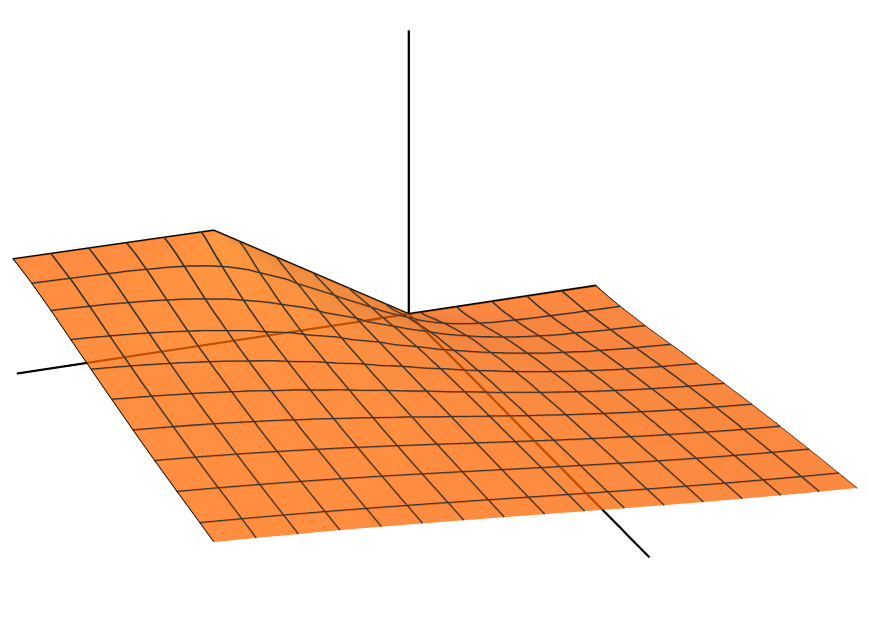}
	\caption{The building block $\varphi$ used for boundary-matching.}
	\hspace{0.5in} \label{FigBuildingBlock}
\end{figure}

\noindent{\bf Remark}. The basic building block for outline matching is the function of the form
\begin{eqnarray}
	\begin{aligned}
	\varphi(x,y)&=
	\frac{1}{2(x_{i+1}-x_i)}\left[\left(x-x_i+\sqrt{(x-x_i)^2+y^2}\right)\right. \\
	&\quad\quad\quad\quad\quad\quad\quad\quad
	-\left.\left(x-x_{i+1}+\sqrt{(x-x_{i+1})^2+y^2}\right)\right].
	\end{aligned} \label{EqnInterpFunction}
\end{eqnarray}
which indeed solves $y\triangle\varphi-\varphi_y=0$ on the upper half-plane.
Crucial for boundary-matching, on $\{y=0\}$ the function is piecewise linear:
\begin{eqnarray}
	\varphi(x,0)
	\;=\;
	\begin{cases}
		\quad\; 0, & x\;\le\;x_i \\
		\frac{x-x_i}{x_{i+1}-x_i} , & x_i\;<\;x\;<\;x_{i+1} \\
		\quad\; 1, & x_{i+1}\;\le\;x
	\end{cases} \label{EqnInterpolation}
\end{eqnarray}
which is the linear interpolation between $0$ and $1$ taking place along the segment $x\in[x_i,x_{i+1}]$.
Summing various functions of the form (\ref{EqnInterpolation}), it is a simple matter to construct any parameterized, piecewise linear map $\{y=0\}\rightarrow\mathbb{R}$ desired, and to use pairs of such functions to create any outline map that we wish.
Then (\ref{EqnInterpFunction}) immediately provides momentum functions on $\overline{H}{}^2$ with this outline.

\begin{proof}
	The polygon has $d+1$ many faces, parameterized at speeds $s_0,\dots,s_d$, and $d$ many vertex points $\vec{p}_1,\dots,\vec{p}_d$ in the $\varphi^1$-$\varphi^2$ plane.
	We take $\vec{p}_i=(m_i,n_i)$.
	Along the line $\{y=0\}$ in $\overline{H}{}^2$ we let the first Lipschitz point be $x_1=0$, and from this we can create the remaining Lipschitz points:
	\begin{eqnarray}
		x_{i+1}\;=\;x_i\,+\,\frac{|\vec{p}_{i+1}-\vec{p}_i|}{s_i}. 
		\label{EqnSpeed}
	\end{eqnarray}
	The proof naturally breaks into three cases: the case that the outline contains parallel rays, the case that the outline is just a single line, and the general case.
	
	\noindent\begin{figure}[h!] 
		\centering
		\vspace{-0.0in}
		\includegraphics[scale={0.425}]{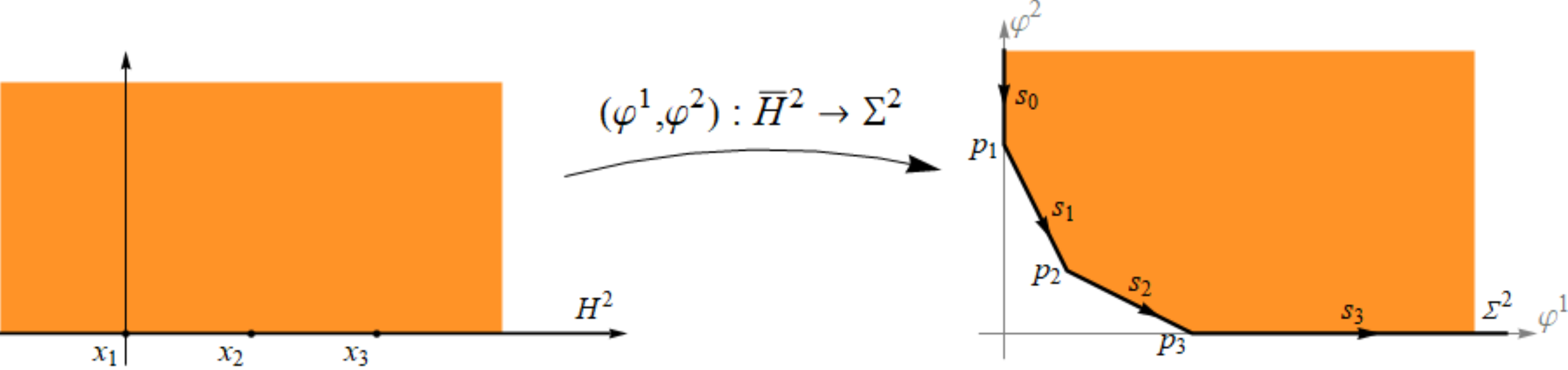}
		\caption{Outline matching in Case I.
		Input data are pivot points $p_1,\dots,p_d$ in the $(\varphi^1,\varphi^2)$-plane and parameterization speeds $s_0,\dots,s_d$ where $s_i\in(0,\infty)$.
		Output is the momentum variables $\varphi^1$, $\varphi^2$ as functions of $x$ and $y$.}
		\hspace{0.5in} \label{FigGeneralCase}
	\end{figure}
	
	\noindent\underline{\it Case I: The general case}.
	
	This is the case where the polygon is not the half-plane, and also has no parallel rays.
	After an affine transformation of the $\varphi^1$-$\varphi^2$ plane, we can assume the two terminal rays lie along the positive $\varphi^1$-axis and the positive $\varphi^2$-axis; see Figure \ref{FigGeneralCase}.
	The momentum functions are
	\begin{eqnarray*}
		\begin{aligned}
			\varphi^1
			&\;=\;
			-\frac{s_0}{2}\left(x-\sqrt{x^2+y^2}\right) \\
			&\quad\quad
			+\sum_{i=1}^{d-1}
			\frac{m_{i+1}-m_i}{2(x_{i+1}-x_i)}
			\left(
			x_i-x_{i+1}+\sqrt{(x-x_i)^2+y^2}
			-\sqrt{(x-x_{i+1})^2+y^2}
			\right)
		\end{aligned}
	\end{eqnarray*}
	and
	\begin{eqnarray*}
		\begin{aligned}
			\varphi^2
			&\;=\;
			\sum_{i=1}^{d-1}
			\frac{n_{i+1}-n_i}{2(x_{i+1}-x_i)}
			\left(
			x_{i+1}-x_i+\sqrt{(x-x_i)^2+y^2}
			-\sqrt{(x-x_{i+1})^2+y^2}
			\right) \\
			&\quad\quad
			+\frac{s_d}{2}\left((x-x_d)+\sqrt{(x-x_d)^2+y^2}\right)
		\end{aligned} \label{EqnMomGeneralCase}
	\end{eqnarray*}
	To verify that the parameterizations along the bounding line $\{y=0\}$ agree with the outline map, we note that the terms of the form $\sqrt{(x-x_i)^2+y^2}$ become $|x-x_i|$ at $y=0$, and we easily verify the tangent vector is
	\begin{eqnarray}
		\begin{aligned}
		\frac{d}{dx}
		=\begin{cases}
			-s_0\frac{\partial}{\partial\varphi^1},
			& x\,<\,0 \\
			\frac{m_{i+1}-m_i}{x_{i+1}-x_i}\frac{\partial}{\partial\varphi^1}+
			\frac{n_{i+1}-n_i}{x_{i+1}-x_i}\frac{\partial}{\partial\varphi^2},
			&x\in(x_i,x_{i+1}),i\in\{1,...,d-1\} \\
			s_d\frac{\partial}{\partial\varphi^2},
			& x\,>\,x_d.
		\end{cases}
		\end{aligned}
	\end{eqnarray}
	along the parameterized path $x\mapsto(\varphi^1(x,0),\varphi^2(x,0))$.
	We remark that from (\ref{EqnSpeed}) we have $s_i=(x_{i+1}-x_i)/|\vec{p}_{i+1}-\vec{p}_i|$ so we see that the parameterization speed is $|d/dx|=s_i$ for $x\in(x_i,x_{i+1})$, as promised. \\
	
	\noindent\begin{figure}[h!] 
		\centering
		\vspace{-0.0in}
		\hspace{-0.6in}
		\includegraphics[scale={0.4}]{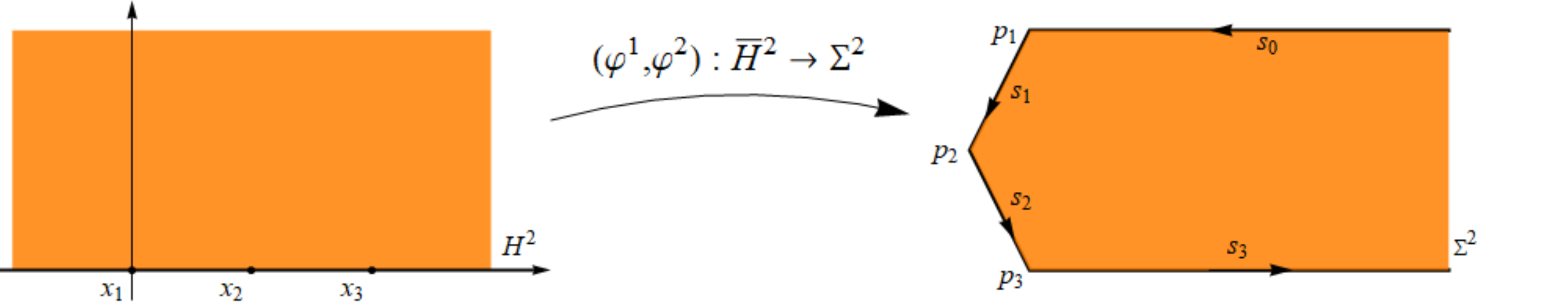}
		\caption{Outline matching in Case II.}
		\hspace{0.5in} \label{FigParCase}
	\end{figure}
	
	\noindent\underline{\it Case II: The outline contains parallel rays}.
	
	In this case the momentum functions are
	\begin{eqnarray*}
		\begin{aligned}
			&\varphi^1
			\;=\;
			\frac{s_0}{2}\left(x-\sqrt{x^2+y^2}\right) \\
			&\quad\quad
			+\sum_{i=1}^{d-1}
			\frac{m_{i+1}-m_i}{2(x_{i+1}-x_i)}
			\left(
			x_{i+1}-x_i+\sqrt{(x-x_i)^2+y^2}
			-\sqrt{(x-x_{i+1})^2+y^2}
			\right) \\
			&\quad\quad
			+\frac{s_d}{2}\left((x-x_d)+\sqrt{(x-x_d)^2+y^2}\right)
		\end{aligned} \label{EqnParMomFuns}
	\end{eqnarray*}
	and
	\begin{eqnarray*}
		\begin{aligned}
			&\varphi^2
			\;=\;
			\sum_{i=1}^{d-1}
			\frac{n_{i+1}-n_i}{2(x_{i+1}-x_i)}
			\left(
			x_{i+1}-x_i+\sqrt{(x-x_i)^2+y^2}
			-\sqrt{(x-x_{i+1})^2+y^2}
			\right)
		\end{aligned}
	\end{eqnarray*}
	To see that the outline is correct, we restrict to the axis $\{y=0\}$, and again use that terms of the form $\sqrt{(x-x_i)^2+y^2}$ become $|x-x_i|$, to find directional derivatives
	\begin{eqnarray}
		\begin{aligned}
			\frac{d}{dx}
			=\begin{cases}
				-s_0\frac{\partial}{\partial\varphi^1},
				& x\,<\,0 \\
				\frac{m_{i+1}-m_i}{x_{i+1}-x_i}\frac{\partial}{\partial\varphi^1}+
				\frac{n_{i+1}-n_i}{x_{i+1}-x_i}\frac{\partial}{\partial\varphi^2},
				& x\in(x_i,x_{i+1}),i\in\{1,...,d-1\} \\
				s_d\frac{\partial}{\partial\varphi^1},
				& x\,>\,x_d.
			\end{cases}
		\end{aligned}
	\end{eqnarray}
	We remark that from (\ref{EqnSpeed}) we have $s_i=(x_{i+1}-x_i)/|\vec{p}_{i+1}-\vec{p}_i|$ so we see that the parameterization speed is $|d/dx|=s_i$ for $x\in(x_i,x_{i+1})$, as promised. \\
	
	\noindent\begin{figure}[h!] 
		\centering
		\vspace{-0.0in}
		\hspace{-0.6in}
		\includegraphics[scale={0.4}]{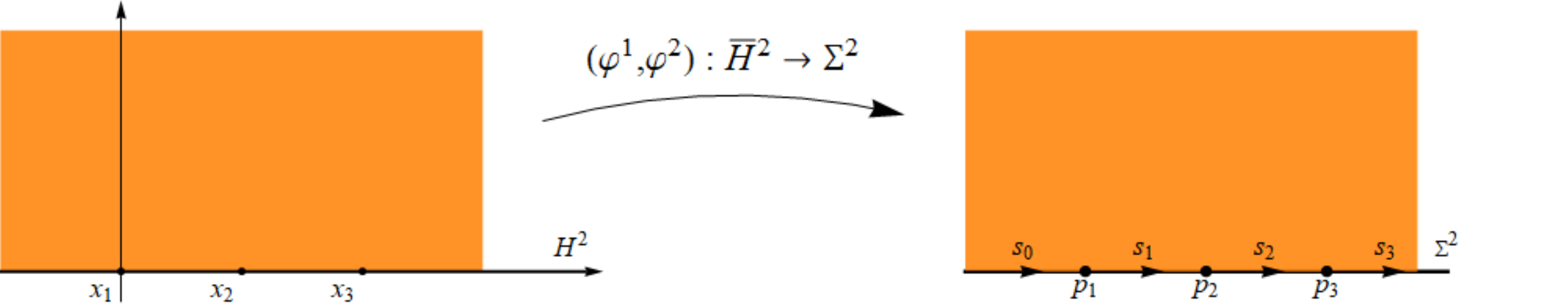}
		\caption{Outline matching in Case III: the ``polygon'' is the half-plane.
			Pivot points will occur only where the parameterization speed changes; see Section \ref{SubsecLipzchitzRemark}.}
		\hspace{0.5in} \label{FigHPCase}
	\end{figure}
	
	\noindent\underline{\it Case III: The ``polygon'' is the half-plane}.
	
	After an affine transformation of the $\varphi^1$-$\varphi^2$ plane we can assume that the ``polygon'' $\Sigma^2$ is the half-plane $\{\varphi^2\ge0\}$.
	We take the momentum functions to be
	\begin{eqnarray*}
		\begin{aligned}
			&\varphi^1
			\;=\;
			\frac{s_0}{2}\left(x-\sqrt{x^2+y^2}\right) \\
			&\quad\quad
			+\sum_{i=1}^{d-1}
			\frac{m_{i+1}-m_i}{2(x_{i+1}-x_i)}
			\left(
			x_{i+1}-x_i+\sqrt{(x-x_i)^2+y^2}
			-\sqrt{(x-x_{i+1})^2+y^2}
			\right) \\
			&\quad\quad
			+\frac{s_d}{2}\left((x-x_d)+\sqrt{(x-x_d)^2+y^2}\right),
			\quad \text{and} \\
			&\varphi^2
			\;=\;\frac12y^2.
		\end{aligned} \label{EqnMomHalfPlane}
	\end{eqnarray*}
	To verify the parameterizations along the bounding line $\{y=0\}$, we use the technique of the previous two cases to verify that
	\begin{eqnarray}
		\begin{aligned}
		\frac{d}{dx}
		=\begin{cases}
		s_0\frac{\partial}{\partial\varphi^1},
		& x\,<\,0 \\
		\frac{m_{i+1}-m_i}{x_{i+1}-x_i}\frac{\partial}{\partial\varphi^1},
		& x\in(x_i,x_{i+1}),\,i\in\{1,\dots,d-1\} \\
		s_d\frac{\partial}{\partial\varphi^1},
		& x\,>\,x_d.
		\end{cases}
	\end{aligned}
	\end{eqnarray}
	along the parameterized path $x\mapsto(\varphi^1(x,0),\varphi^2(x,0))$.
	We note that (\ref{EqnSpeed}) implies $s_i=\frac{m_{i+1}-m_i}{x_{i+1}-x_i}$, and so the parameterization speed is what was promised. \\
	
	\noindent\underline{Proof that $(\varphi^1,\varphi^2):\overline{H}{}^2\rightarrow\Sigma^2$ is bijective and non-singular.}
	
	We have constructed a map $(\varphi^1,\varphi^2):\overline{H}{}^2\rightarrow\Sigma^2$ that agrees with the outline map when restricted to $\{y=0\}$.
	But there is no indication so far that the image is $\Sigma^2$, that the map is surjective, or that it is non-singular.
	
	These facts are proved by Abreu and Sena-Dias in Theorem 4.1 of \cite{AS}.
	Their proof is technically intricate, and the present author has found no appreciable simplification so we think it is best to refer to their paper.
	
	There is a translation issue between their paper and ours: theirs is the Legendre transform of ours.
	The volumetric normal coordinates $(x,y)$ of this paper are the same as the coordinates $(H,r)$ of \cite{AS}.
	One determines the so-called {\it symplectic potential}, which is a convex function ${\bf{u}}:\Sigma^2\rightarrow\mathbb{R}$ satisfying
	\begin{eqnarray}
		\frac{\partial^2{\bf{u}}}{\partial\varphi^i\partial\varphi^j}
		\;=\;\left<\frac{\partial}{\partial\varphi^i},\,\frac{\partial}{\partial\varphi^i}\right>
		\;=\;g_{\Sigma,ij}.
	\end{eqnarray}
	Executing a Legendre transform using the convex function ${\bf{u}}$, the polygon $\Sigma^2$ in the $\varphi^1$-$\varphi^2$ plane is transformed to the $\xi_1$-$\xi_2$ plane by
	\begin{eqnarray}
		\xi_1\;=\;\frac{\partial{\bf{u}}}{\partial\varphi^1}, \quad
		\xi_2\;=\;\frac{\partial{\bf{u}}}{\partial\varphi^2}
	\end{eqnarray}
	which, we remark, is the equation $d{\bf{u}}=\xi_1dx_i+\xi_1dx_i$ in the statement of Theorem 4.1 of \cite{AS}.
	The image of $\Sigma^2$ is the entirety of the $\xi_1$-$\xi_2$ plane.
	Under the Legendre transform $(\varphi^1,\varphi^2)\rightarrow(\xi_1,\xi_2)$, the degenerate-elliptic equation $y\triangle\varphi^i-\varphi^i_y=0$ transforms into the degenerate-elliptic equation
	\begin{eqnarray}
		y\triangle\xi_i\,+\,(\xi_i)_y\;=\;0
	\end{eqnarray}
	which has the opposite sign on the $\partial_y$-term; this equation is the primary consideration of \cite{AS}, as opposed to $y\triangle\varphi^i-(\varphi^i)_y=0$ which is our primary consideration.
	
	The explicit transformation from our expressions of $\varphi^1$, $\varphi^2$ in terms of $x$, $y$ to the corresponding expressions of $\xi_1$, $\xi_2$ in terms of $x$, $y$ are actually given within the proof of Theorem 4.1 of \cite{AS} (unfortunately most of their expressions are not labeled, but are about a page below their equation (7)).
	Having made this translation from our framework to the Abreu-Sena-Dias framework, the Abreu-Sena-Dias proof goes through without change.
\end{proof}

{\bf Remark}. The outline matching in the proof above does not require the outline be convex, although the resulting map $(\varphi^1,\varphi^2):\overline{H}{}^2\rightarrow\Sigma^2$ will be singular if not.
See \ref{SubsecPathExamples} for some pathological examples of this kind.

\subsection{The Improved Liouville theorem} \label{SubsecImprLiouville}

\begin{theorem}[Liouville theorem, {\it cf.} Corollary 1.11 of \cite{Web2}] \label{ThmOrigLiouville}
	Assume $\varphi\in{}C^0(\overline{H}{}^2)\cap{}C^2(H^2)$ is non-negative and solves
	\begin{eqnarray}
		y(\varphi_{xx}+\varphi_{yy})\,-\,\varphi_y\;=\;0
	\end{eqnarray}
	in $H^2$, with boundary condition $\varphi(x,0)=0$ on $\{y=0\}$.
	
	Then $\varphi=C_1y^2$ for some constant $C_1\ge0$.
\end{theorem}

\begin{theorem}[Improved Liouville theorem] \label{ThmImprLiouville}
	Assume $\varphi\in{}C^0(\overline{H}{}^2)\cap{}C^2(H^2)$ solves
	\begin{eqnarray}
		y(\varphi_{xx}+\varphi_{yy})\,-\,\varphi_y\;=\;0
	\end{eqnarray}
	in $H^2$, and assume $\varphi=0$ on the boundary line $\{y=0\}$.
	Further assume $\varphi$ is bounded from below by
	\begin{eqnarray}
		\varphi\;>\;-A\,-\,B|x|\,-\,Cy^\delta
	\end{eqnarray}
	for constants $A,B,C\in\mathbb{R}^{\ge0}$ and $\delta\in[0,2)$.
	
	Then $\varphi\ge0$, and indeed $\varphi=C_1y^2$ for some $C_1\ge0$.
\end{theorem}
\begin{proof}
	We use lower barrier functions to progressively improve the lower bound, finally arriving at $\varphi\ge0$.
	
	\underline{\it Step I}: {\it Improving the lower bound to $\varphi\ge-B^2\sqrt{(x-x_0)^2+y^2}$.}
	
	Consider the function $\underline\varphi$ given by
	\begin{eqnarray}
		\underline{\varphi}\;\triangleq\;
		-M_0-M_1\sqrt{(x-x_0)^2+y^2}\,-\,M_2y^2,
	\end{eqnarray}
	which solves $y(\underline{\varphi}{}_{xx}+\underline{\varphi}{}_{yy})-\underline{\varphi}{}_y=0$.
	We take $M_0,M_2>0$, $M_1>B$, and $x_0\in\mathbb{R}$.
	Our aim is to show that $\underline{\varphi}<\varphi$ even as $\epsilon\searrow0$ and $M_2\searrow0$. 
	A simple computation shows that if
	\begin{eqnarray}
		x\notin\left[-\frac{A-M_0-M_1|x_0|}{M_1-B},\,\frac{A-M_0+M_1|x_0|}{M_1-B}\right],
	\end{eqnarray}
	then
	\begin{eqnarray}
		-M_0-M_1|x-x_0|\le-A-B|x|
	\end{eqnarray}
	(this is where we use $M_1>B$).
	Assuming also $y>(C/M_2)^{2-\delta}$ then
	\begin{eqnarray}
		\begin{aligned}
			\underline\varphi
			&\;=\;-M_0-M_1\sqrt{(x-x_0)^2+y^2}-M_2y^2 \\
			&\;\le\;-A-B|x|-Cy^\delta
		\end{aligned}
	\end{eqnarray}
	In other words, outside of the compact region
	\begin{eqnarray*}
		\mathcal{R}=\left\{(x,y)
		\,\Big|\,
		-\frac{A-M_0-M_1|x_0|}{M_1-B}\le{}x\le\frac{A-M_0+M_1|x_0|}{M_1-B}, \; 0\le{}y\le\left(\frac{C}{M_2}\right)^{2-\delta}
		\right\},
	\end{eqnarray*}
	we have $\underline\varphi<\varphi$.
	Finally, by assumption $\varphi$ is continuous and equals 0 on $\{y=0\}$.
	Since $0<\underline{\varphi}$ on $\{y=0\}$ there is some neighborhood $\Omega$ of $\{y=0\}$ on which $\underline\varphi<\varphi$.
	Therefore $\underline\varphi<\varphi$ except possibly on the compact region $\mathcal{R}\setminus\Omega$.
	But the operator $y\triangle-\partial_y$ is uniformly elliptic on $\mathcal{R}\setminus\Omega$, so by the maximum principle we have $\underline\varphi<\varphi$ on all of $\overline{H}{}^2$.
	Finally let $M_0,M_2\searrow0$ and $M_1\searrow{}B$. \\
	
	\noindent\underline{\it Step II}: {\it Improving the lower bound to $\varphi\ge-By$.}
	
	From Step I we have $\varphi\ge-B\sqrt{(x-x_0)^2+y^2}$.
	Therefore
	\begin{eqnarray}
		\varphi\;\ge\;\sup_{x_0\in\mathbb{R}}-B\sqrt{(x-x_0)^2+y^2}
	\end{eqnarray}
	The fact that $\sup_{x_0\in\mathbb{R}}-B\sqrt{(x-x_0)^2+y^2}=-By$ gives the result. \\
	
	\noindent\underline{\it Step III}: {\it Improving the lower bound to $\varphi\ge0$.}
	
	From Step II, $\varphi+\epsilon{}y^2\ge0$ outside the strip $y\in[0,B\epsilon^{-1}]$.
	To show that also $\varphi+\epsilon{}y^2\ge0$ on the strip, we use the barrier
	\begin{equation}
		\begin{aligned}
		&\underline{\varphi}
		\;=\;f(y)g(x), \quad \text{with} \quad
		 g(x)=\cosh\left(\frac12B^{-1}\epsilon{}y_{1,1}\cdot{}x\right) \quad \text{and} \\
		&\hspace{0.5in}f(y)=\frac{\pi}{4}B^{-1}\epsilon{}y_{1,1}\cdot{}y\,\,Y_1\left(\frac12B^{-1}\epsilon{}y_{1,1}\cdot{}y\right)
		\end{aligned}
	\end{equation}
	where $Y_1$ is the Bessel function of the second kind and $y_{1,1}\approx2.20$ is its first zero.
	One easily verifies $y(\underline{\varphi}{}_{xx}+\underline{\varphi}{}_{yy})-\underline{\varphi}{}_y=0$ and $\underline{\varphi}(x,0)=-\cosh(x\cdot{}B^{-1}\epsilon{}y_{1,1})$.
	The function $f(y)$ was chosen specifically so $f(0)=-1$ and $f(y)\le-1$ on $[0,B\epsilon^{-1}]$.
	
	We use the maximum principle to prove that, for any $\delta>0$, we have $\varphi>\delta\underline{\varphi}$ on the strip.	
	Because $f(y)\le-1$ on $[0,B\epsilon^{-1}]$, we have
	\begin{equation}
		\delta\underline{\varphi}(x,y)\;\le\;-\delta\cosh\left(\frac12B^{-1}\epsilon{}y_{1,1}\cdot{}x\right)
	\end{equation}
	on $y\in[0,B\epsilon^{-1}]$.
	In particular $\varphi\ge\delta\underline{\varphi}$ on $\{y=0\}$ and $\{y=B\epsilon^{-1}\}$.
	From step II $\varphi+\epsilon{}y^2\ge-By+\epsilon{}y^2\ge-\frac14B^2\epsilon^{-1}$, so on the two segments $x=\pm\frac{2B}{\epsilon{}y_{1,1}}\cosh^{-1}(\frac{B^2}{4\epsilon\delta})$ we also have $\underline{\varphi}\le\varphi$.
	In other words, $\underline{\varphi}\le\varphi$ on the boundary of the rectangle
	\begin{equation}
		\left\{y\in[0,B\epsilon^{-1}],\quad x\in\left[-\frac{2B\epsilon^{-1}}{y_{1,1}}\cosh^{-1}\left(\frac{B^2}{4\epsilon\delta}\right),\frac{2B\epsilon^{-1}}{y_{1,1}}\cosh^{-1}\left(\frac{B^2}{4\epsilon\delta}\right)\right]
		\right\},
	\end{equation}
	so by the maximum principle $\underline{\varphi}<\varphi$ on the entirety of the rectangle.
	Now sending $\delta\searrow0$ gives $\varphi+\epsilon{}y^2\ge0$ on the entire strip, and therefore on the entire half-plane.
	Finally sending $\epsilon\searrow0$ gives $\varphi\ge0$ on $\overline{H}{}^2$.
\end{proof}

\subsection{The classification in the general case} \label{SubsecClassGen}

We prove the classification theorem in the ``general case'' of Theorem \ref{ThmIntroGeneralPoly}, the case $\Sigma^2$ is not the half-plane and does not have parallel rays.
After an affine recombination of $\mathcal{X}^1$-$\mathcal{X}^2$ we may assume $\Sigma^2$ has one terminal ray along the $\varphi^1$-axis and the other along the $\varphi^2$-axis.
This constitutes Case I from the proof of Lemma \ref{LemmaOutlineLemma}; see Figure \ref{FigGeneralCase}.

The outline of the classification proof is as follows.
From the metric polygon $(\Sigma^2,g_\Sigma)$ we create isothermal coordinates $(x,y)$, which map $\Sigma^2$ bijectively onto the upper half-plane.
The moment functions $\varphi^1,\varphi^2$ now exist as $\varphi^i:\overline{H}{}^2\rightarrow\mathbb{R}$.
Then we create {\it new} moment functions $\tilde{\varphi}^1$, $\tilde\varphi^2$ using the outline-matching procedure from the proof of Lemma \ref{LemmaOutlineLemma}.
Because the functions $\varphi^1$ and $\tilde{\varphi}^1$ (resp. $\varphi^2$ and $\tilde{\varphi}^2$) agree along the boundary line $\{y=0\}$, we have
\begin{eqnarray}
	\varphi^i(x,0)\,-\,\tilde\varphi{}^i(x,0)\;=\;0 \quad on \quad \{y=0\}.
\end{eqnarray}
We can assume $\Sigma^2$ is the first quadrant, we have $\varphi^i\ge0$.
By construction, $\tilde\varphi^i$ has linear growth at worst.
Therefore $\varphi-\tilde{\varphi}\ge-A+B\sqrt{x^2+y^2}$.
This allows us to use the Improved Liouville theorem, Theorem \ref{ThmImprLiouville}, to obtain
\begin{eqnarray}
	\begin{aligned}
		&\varphi{}^1(x,y)\,-\,\tilde\varphi{}^1(x,y)\;=\;C_1y^2, \quad \text{and} \\
		&\varphi{}^2(x,y)\,-\,\tilde\varphi{}^2(x,y)\;=\;C_2y^2
	\end{aligned}
\end{eqnarray}
on $\overline{H}{}^2$, where $C_1$ and $C_2$ are non-negative constants.
Thus $\varphi^1$, $\varphi^2$ are specified up to a 2-parameter family of possible variations.
We have proven the following.
\begin{theorem}[{\it cf.} Theorem \ref{ThmIntroGeneralPoly}] \label{ThmGeneralCase}
	Assume $(\Sigma^2,g_\Sigma)$ is a metric polygon obeying (A)-(F) that is neither the half-plane nor has parallel rays, and has $d$ many vertices.
	
	If boundary data is specified (as discussed in \S\ref{SubsecBoundaryConds}) then $g_\Sigma$ is a member of a 2-parameter family of possible metrics.
	If boundary data is not specified and the polygon has $d$ many vertices, then $g_\Sigma$ is a member of a ($d$+3)-parameter family of possible metrics.
\end{theorem}

\subsection{Classification in the case $\partial\Sigma^2$ has parallel rays} \label{SubsecClassPar}

Next we examine the classification problem in the case $\partial\Sigma^2$ has parallel rays.
In this case we can make an affine change of coordinates so that the outline $\partial\Sigma^2$ has terminal rays parallel to the $\varphi^1$-axis.
This constitutes Case II from the proof of Lemma \ref{LemmaOutlineLemma}, and is depicted in Figure \ref{FigParCase}.

The proof is similar to the proof in the general case: after outline matching, we determine that the two moment variables have indeterminacy up to summands of the form $Cy^2$.
The difference is that, in this case, the $\varphi^2$ function must remain bounded, as demanded by the fact that the polygon lies within a strip.
This means that $\varphi^2$ cannot be modified by adding any multiple of $y^2$, so the indeterminacy of the pair $(\varphi^1,\varphi^2)$ is reduced by one degree of freedom.
We have proven the following.
\begin{theorem}[{\it cf.} Theorem \ref{ThmIntroParallelPoly}] \label{ThmParallelRayCase}
	Assume $(\Sigma^2,g_\Sigma)$ is a metric polygon obeying (A)-(F) and $\Sigma^2$ has parallel rays and $d$ many vertices.
	
	If boundary data is specified (as discussed in \S\ref{SubsecBoundaryConds}) then $g_\Sigma$ is a member of a 1-parameter family of possible metrics.
	If boundary data is not specified and the polygon has $d$ many vertices, $g_\Sigma$ is a member of a ($d$+2)-parameter family of possible metrics.
\end{theorem}

\subsection{Classification in the case $\Sigma^2$ is the half-plane} \label{SubsecClassHP}

There is a serious technical issue that separates this case from the other two.
In those cases, it could be arranged that both momentum functions were positive, and then after boundary matching we could use the Improved Liouville Theorem \ref{ThmImprLiouville}.
However in the half-plane case, it can only be arranged that \textit{one} of the momentum functions is positive.
The other can have no lower bound.

\begin{proposition}[Half-plane polygons, {\it cf.} Theorem \ref{ThmIntroHalfPlanePoly}] \label{PropHalfPlaneFlat}
	Assume the metric polygon $(\Sigma^2,g_\Sigma)$ is the closed half-plane $\Sigma^2=\{\varphi^2\ge0\}$ and obeys (A)-(F).
	After possible affine recombination of $\varphi^1$, $\varphi^2$, there exists a constant $M\ge0$ so that
	\begin{eqnarray}
		\varphi^1\;=\;x+Mxy^2,
		\quad \varphi^2\;=\;\frac12y^2. \label{EqnsHalfPlaneTransitions1}
	\end{eqnarray} 
\end{proposition}
\textbf{Remark.}
	Prior to the possible affine recombination, the raw construction gives constants $C_1,C_2>0$ and $C_3,M_1\ge0$ so that
	\begin{eqnarray}
		\varphi^1\;=\;C_1x+M_1xy^2+C_3y^2, \quad
		\varphi^2\;=\;C_2y^2. \label{EqnsHalfPlaneTransitions2}
	\end{eqnarray} 

\begin{proof}
	After an affine transformation of variables if necessary, we may assume $\Sigma^2=\{\varphi^2\ge0\}$.
	By (C), the functions $\varphi^1$, $\varphi^2$ are $C^\infty(\overline{H}{}^2)$---therefore there are no Lipschitz points along the boundary (see the remark on Lipschitz points along segments in \S\ref{SubsecLipzchitzRemark}).
	
	To take care of the $\varphi^2$ variable, note that since $\varphi^2(0,y)=0$ and $\varphi^2\ge0$ we can use the Liouville theorem of \cite{Web2}, recorded here as Theorem \ref{ThmOrigLiouville}, to guarantee $\varphi^2=C_2y^2$, $C_2>0$.
	Scaling coordinates if we like, we can take $C_2=\frac12$.
	
	The transition matrix from $\left\{\frac{\partial}{\partial{x}},\frac{\partial}{\partial{y}}\right\}$ to $\left\{\frac{\partial}{\partial\varphi^1},\,\frac{\partial}{\partial\varphi^2}\right\}$ is
	\begin{eqnarray}
		A\;=\;
		\left(\begin{array}{cc}
			\frac{\partial\varphi^1}{\partial{}x} & \frac{\partial\varphi^1}{\partial{}y} \\
			0 & y
		\end{array}\right)
	\end{eqnarray}
	so by (\ref{EqnMetricConstruction}) the polygon metric is simply
	\begin{eqnarray}
		g_\Sigma
		\;=\;\frac{\partial\varphi^1}{\partial{}x}
		\left(dx\otimes{}dx\,+\,dy\otimes{}dy\right).
	\end{eqnarray}
	The positive definiteness of $g_\Sigma$ now gives $\frac{\partial\varphi^1}{\partial{}x}>0$.
	
	Next, because $\varphi^1$ solves $y(\varphi^1_{xx}+\varphi^1_{yy})-\varphi^1_y=0$, taking a derivative in the $x$ direction shows that $\varphi^1_x$ solves $y((\varphi^1_x)_{xx}+(\varphi^1_x)_{yy})-(\varphi^1_x)_y=0$.
	Also, along the boundary line $\{y=0\}$ necessarily $\varphi^1_x=s_0$, which is the parameterization speed.
	In particular $\varphi^1_x$ is constant along $\{y=0\}$.
	Now consider the function
	\begin{eqnarray}
		\tilde\varphi\;\triangleq\;\frac{\partial\varphi^1}{\partial{}x}-s_0.
	\end{eqnarray}
	We have the following three facts:
	\begin{itemize}
		\item[{\it{i}})] $\tilde\varphi\;=\;0$ along the boundary $\{y=0\}$,
		\item[{\it{ii}})] $\tilde\varphi$ solves $y\left(\tilde\varphi_{xx}+\tilde\varphi_{xx}\right)-\tilde\varphi_y=0$, and
		\item[{\it{iii}})] $\tilde\varphi>-s_0$.
	\end{itemize}
	By the Improved Liouville theorem, Theorem \ref{ThmImprLiouville}, we have that $\tilde\varphi$ is a non-negative multiple of $y^2$.
	Therefore
	\begin{eqnarray}
		\begin{aligned}
		&\frac{\partial\varphi^1}{\partial{x}}\;=\;s_0\,+\,My^2
		\quad \text{for some} \quad M\ge0.
		\end{aligned}
	\end{eqnarray}
	Integrating in $x$ we have $\varphi^1=s_0x+Mxy^2+C(y)$ where $C(y)$ is some function of $y$ alone.
	Then $C(y)$ solves $yC_{yy}-C_y=0$, meaning $C=C_0+C_3y^2$ and so $\varphi^1=C_0+C_1x+M_1xy^2+C_3y^2$ for constants $C_0$, $C_1>0$, and $M_1,C_3\ge0$.
	Translating in the $x$ direction if necessary, we may assume $C_0=0$.
\end{proof}

\subsection{Removal of Lipschitz points that are internal to edges} \label{SubsecLipzchitzRemark}

The image in Figure \ref{FigHPCase} shows a half-plane with three Lipschitz points on its boundary.
Likewise, other types of polygons might have Lipschitz points inside a segment, where the edge's parameterization speed changes but its direction does not.
We show that this is always pathological, as it always generates a curvature singularity.

\begin{lemma}
	Assume $(\Sigma^2,g_\Sigma)$ is a metric polygon with a point $p\in\partial\Sigma^2$ that is not a vertex of the polygon, but is a Lipschitz point of $\varphi^1$ or $\varphi^2$ as functions of $x$, $y$.
	Then $g_\Sigma$ has a curvature singularity at $p$.
\end{lemma}
\begin{proof}
	We have that $p\in\partial\Sigma^2$ is a point along an edge so that, even though there is no direction change, there is a speed change with speeds $s_0$, $s_1$ on either side of $p$.
	
	First we scale the metric and the polygon, and take a limit.
	Letting $n\rightarrow\infty$, we scale the $(\varphi^1,\varphi^2)$-plane by a factor of $n$ and scale $g_\Sigma$ by a factor of $n^2$.
	Taking a pointed Hausdorff limit of the polygon $\Sigma^2$ and taking a pointed Gromov-Hausdorff limit of $g_\Sigma$ we obtain a metric polygon that is a half-plane.
	The boundary values do not change under this simultaneous scaling, so the limiting polygon $(\Sigma^2,g_\Sigma)$ is a half-plane with a single Lipschitz point.
	
	Using the boundary-matching method of the proof of Lemma \ref{LemmaOutlineLemma}, set
	\begin{eqnarray}
		\begin{aligned}
		&\tilde{\varphi}{}^1
		\;=\;
		\frac{s_0}{2}\left(x-\sqrt{x^2+y^2}\right)
		+\frac{s_1}{2}\left(x+\sqrt{x^2+y^2}\right), \quad
		\tilde{\varphi}{}^2\;=\;\frac12y^2
		\end{aligned}
	\end{eqnarray}
	where $s_0\neq{}s_1$ are the labels.
	The boundary-matching method combined with the Improved Liouville Theorem, as executed in Theorem \ref{PropHalfPlaneFlat}, shows that the functions $\varphi^1$, $\varphi^2$ are equal to $\tilde{\varphi}{}^1$, $\tilde{\varphi}{}^2$ up to a summand of $Cy^2$, $C>0$.
	After possible scaling the moment functions and translating the polygon so $p=(0,0)$, we see that
	\begin{eqnarray}
		\begin{aligned}
		&\varphi^1
		\;=\;
		\frac{s_0}{2}\left(x-\sqrt{x^2+y^2}\right)
		+\frac{s_1}{2}\left(x+\sqrt{x^2+y^2}\right)
		+Cxy^2, \\
		&\varphi^2\;=\;\frac12y^2.
		\end{aligned}
	\end{eqnarray}
	Computing $K_\Sigma$ using (\ref{EqnCompKSigma}), we obtain
	\begin{eqnarray}
		\begin{aligned}
		K_\Sigma
		&\;=\;\frac{s_1-s_0}{2}\cdot
		\frac{s_1(x-\sqrt{x^2+y^2})+s_0(x+\sqrt{x^2+y^2})}{\left(s_1(x-\sqrt{x^2+y^2})+s_0(x+\sqrt{x^2+y^2})\right)^3}.
		\end{aligned}
	\end{eqnarray}
	This expression has an infinite singularity at the origin.
	Because this curvature singularity exists after the blow-up process, by continuity it certainly existed before---recall the blowup process systematically {\it decreases} curvature.
	This concludes the proof that $g_\Sigma$ has a curvature singularity at $p$.
\end{proof}

\section{Consideration of the asymptotic conditions} \label{SecAsymptotics}

Our classification requires only that the polygon be topologically closed.
We show this is true under either of the asymptotic conditions
\begin{itemize}
	\item[A1)] The action fields decay slowly (or not at all): $\sqrt{|\mathcal{X}^1|^2+|\mathcal{X}^2|^2}\;>\;C_1r^{-1+\epsilon}$ for some $\epsilon>0$ and $C_1>\infty$, when $r$ is sufficiently large.
	\item[A2)] Curvature decays quickly: $|\Riem|<(2-\epsilon)r^{-2}$, for any $\epsilon>0$, when $r$ is sufficiently large.
\end{itemize}
The boundary is connected (by condition (B)), and must include \textit{some} point of closure, or else by Theorem \ref{ThmIntroCompletePolygon} the manifold is flat.
From the literature, specifically \cite{Delzant} or \cite{McDSal}, we know the local structure near points of closure on $\Sigma^2$: such a point is locally part of a of a vertex or part of a segment.
Then we take a limit along a segment as we approach a non-included point.
Because the map $\Phi:M^4\rightarrow\Sigma^2$ is smooth, this can only occur if we approach infinity in $M^4$ but remain finite in the coordinate plane.
However the value of $\varphi^i$ can be determined by its gradient $|\nabla\varphi^i|=|\mathcal{X}^i|$, so $\varphi^i$ can only remain finite if the norm $|\mathcal{X}^i|$ gets too small too fast.

\subsection{The asyptotic conditions (A1) and (A2)}

\begin{proposition} \label{PropKillingDecay}
	Assume $(M^4,J,g,\mathcal{X}^1,\mathcal{X}^2)$ is a ZSC toric K\"ahler manifold of finite topology, and assume all manifold ends satisfy (A1) or (A2).
	Then its reduction $(\Sigma^2,g_\Sigma)$ is topologically closed.
\end{proposition}
\begin{proof}
	We show that $\Sigma^2$ includes its boundary.
	Conditions (A1) and (A2) involve the distance $r$.	
	It is ordinarily very difficult to find geodesics without further information about the metric.
	But by (D) all edges of $\Sigma^2$ are geodesics.
	The idea is to show that, under (A1) or (A2), if a boundary segment has infinite metric-length, it must have infinite coordinate-length, too.
	Thus the boundary cannot have any excluded points.
	
	First we rule out the case of no edges.
	But then $span\{\mathcal{X}^1,\,\mathcal{X}^2\}$ is everywhere rank 2 so the moment map $\Phi:M^4\rightarrow\Sigma^2$ is a Riemannian submersion.
	Then $(\Sigma^2,g_\Sigma)$ is a complete Riemannian manifold, so by Corollary \ref{ThmIntroM4CompletePolygon} it is flat.
	
	Therefore $span\{\mathcal{X}^1,\,\mathcal{X}^2\}$ has rank $1$ or $0$ somewhere, meaning the volume function $\mathcal{V}=|\mathcal{X}^1|^2|\mathcal{X}^2|^2-\left<\mathcal{X}^1,\mathcal{X}^2\right>{}^2$ has zeros.
	We first establish the \textit{local} structure of any such zero.
	The local structure can be found, for instance, in \cite{Delzant} or in the book \cite{McDSal}.
	To wit, at a point $p\in{}M^4$ if the rank of the distribution lowers to $1$, then a neighborhood of $p$ maps to a half-disk in the $(\varphi^1,\varphi^2)$-plane with the image of $p$ being on the closed half-disk boundary.
	If the rank reduces to $0$, a neighborhood of $p$ reduces to a wedge with the image of $p$ being at the vertex.
	Therefore at least some edges exist.
	Using (A1) or (A2) we prove that any such edge is closed.
	
	
	Pick a line segment $l\subset\Sigma^2$; after an affine transformation we may assume $l$ is oriented along the $\varphi^1$-axis.
	Over such a line segment is a 2-dimensional submanifold $\Phi^{-1}(l)=L^2\in{}M^4$ on which the Killing field $\mathcal{X}^1$ is not identically zero.
	The momentum function $\varphi^1$ satisfies $\nabla\varphi^1=J\mathcal{X}^1$.
	The manifold $L^2\subset{}M^4$ is a totally geodesic, holomorphic submanifold (as it is a zero-set of the holomorphic killing field $\mathcal{X}^2$).
	See Figure \ref{FigDepictL2}.
	\begin{figure}[h]
		\includegraphics[scale=0.5]{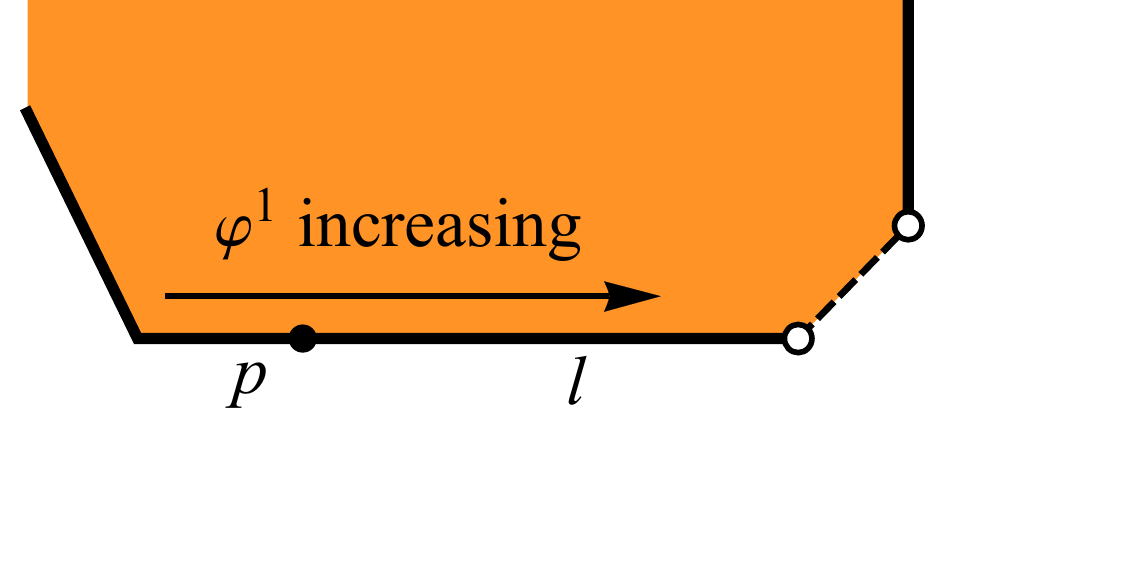}	\includegraphics[scale=0.5]{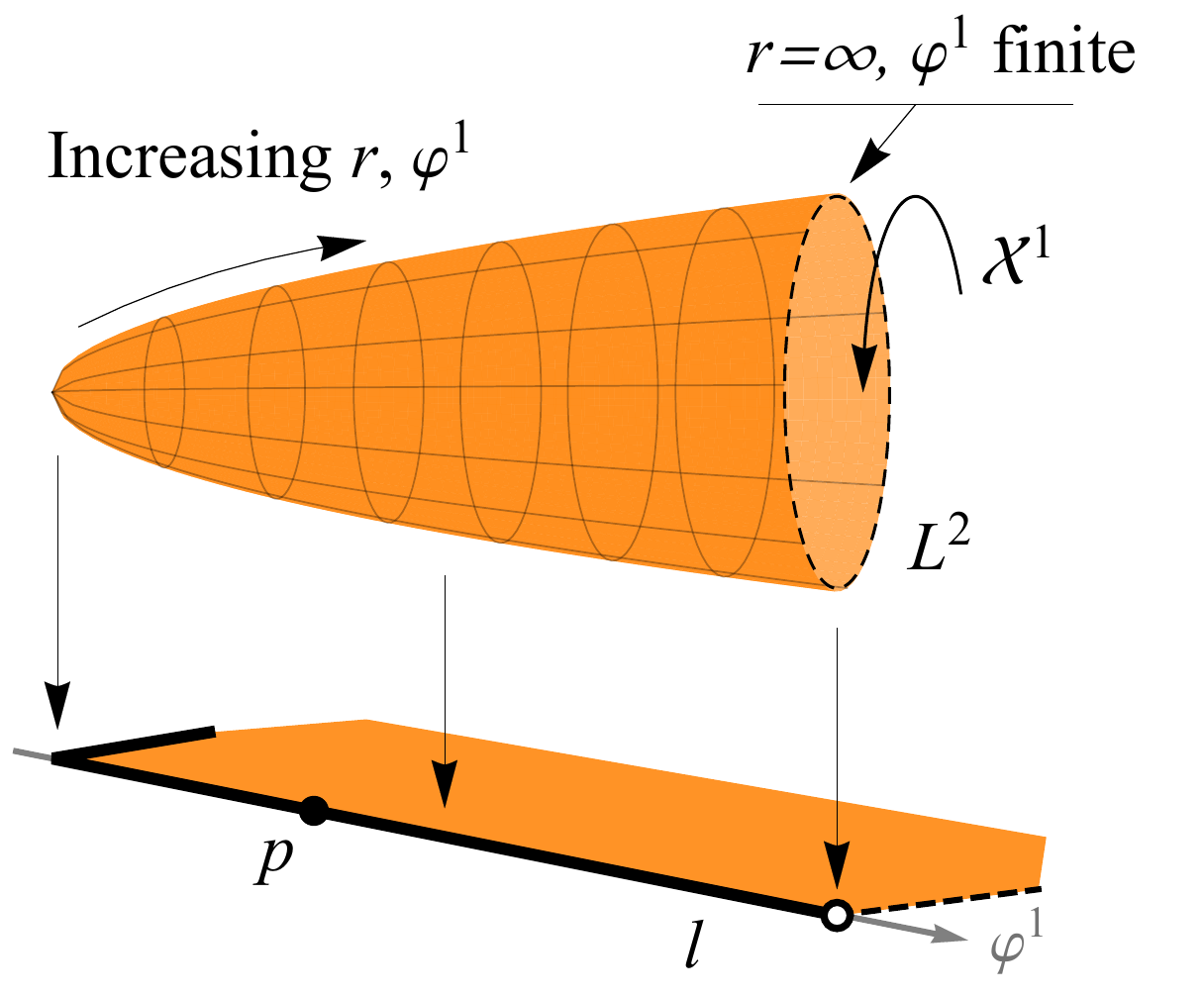}
		\caption{A non-closed segment $l$ and its overlying 2-manifold $L^2=\Phi^{-1}(l)$.
		An open point is at metric infinity, but finite $\varphi^1$.}
		\label{FigDepictL2}
	\end{figure}
	
	If the segment $l\subset\Sigma^2$ is not closed and if a point $p$ moves along $l$ toward a point of closure, then by continuity, all points in the pre-image $\Phi^{-1}(p)$ must have no cluster points, meaning the set $\Phi^{-1}(p)$ must travel infinitely far away.
	Because $l$ and $L^2$ are totally geodesic, $l$ and $L^2$ must both be unbounded in the metric sense.
	
	Either of the hypotheses (A1) or (A2) can show this to be impossible.
	If (A1) is true, then $|\mathcal{X}^1|>\min\{C_0,C_1r^{-1+\epsilon}\}$.
	Because $L^2$ is holomorphic and $\mathcal{X}^1$ is a Killing field on $L^2$, also $\nabla\varphi^1=J\mathcal{X}^1$ is tangent to $L^2$.
	Thus $\nabla\varphi^1=\frac{\partial\varphi^1}{\partial{}r}dr$ where $r$ is any rotationally-invariant distance function along $L^2$.
	Then using $|\nabla\varphi^1|=|\mathcal{X}^1|>C_1r^{-1+\epsilon}$ we have
	\begin{eqnarray}
		\begin{aligned}
		&\frac{\partial\varphi^1}{\partial{}r}
		\;=\;\nabla\varphi^1
		\;\ge\;C_1r^{-1+\epsilon}, \quad \text{so} \quad
		\varphi^1(r)\;=\;O(r^{\epsilon})
		\end{aligned}
	\end{eqnarray}
	Since $\epsilon>0$ we have $\varphi^1\rightarrow\infty$ as $r\rightarrow\infty$, the sought-for contradiction.
	
	Next assume (A2) holds.
	Again $\mathcal{X}^1$ is a Killing field on $L^2$, so in particular it obeys the Jacobi equation
	\begin{equation}
		\nabla_{\frac{\partial}{\partial{}r}}\nabla_{\frac{\partial}{\partial{}r}}\mathcal{X}^1
		+\Riem\left(\mathcal{X}^1,\frac{\partial}{\partial{}r}\right)\frac{\partial}{\partial{}r}\;=\;0.
	\end{equation}
	Because $|\mathcal{X}^1|>0$ (or else there would reach a vertex and the segment would be closed) and because $|\Riem|<(2-\epsilon)r^{-2}$, the usual Rauch comparison says that $|\mathcal{X}^1|>f$ where $f$ is a non-negative solution of the comparison Riccati equation $f_{rr}-(2-\epsilon)r^{-2}f=0$.
	This gives lower boundary $f=C_0r^{\frac{1-\sqrt{9-4\epsilon}}{2}}$ for $|\mathcal{X}^1|$.
	Then because $|\nabla\varphi^1|=|\partial\varphi^1/\partial{}r|$ on $L^2$ (as $\varphi^1$ only changes in the radial direction, not in the direction of the Killing field), we have
	\begin{equation}
		\left|\frac{\partial\varphi^1}{\partial{}r}\right|
		\;=\;|\mathcal{X}^1|
		\;\ge\;C_0r^{\frac{1-\sqrt{9-4\epsilon}}{2}}.
	\end{equation}
	Integrating, we have $\varphi^1=O(r^{\frac{3-\sqrt{9-4\epsilon}}{2}})$.
	The exponent $\frac{3-\sqrt{9-4\epsilon}}{2}$ is positive, so again $\varphi^1\rightarrow\infty$ as $r\rightarrow\infty$, giving the desired contradiction.
\end{proof}

\begin{proposition}
	Assume $(M^4,J,g,\mathcal{X}^1,\mathcal{X}^2)$ is a ZSC toric K\"ahler manifold, and that all ends are asymptotically spheroidal.
	Then its metric reduction $(\Sigma^2,g_\Sigma)$ is topologically closed.
	Its polygon $\Sigma^2$ does not have parallel rays, and is not the half-plane.
\end{proposition}
\begin{proof}
	$\Sigma^2$ is topologically closed by Proposition \ref{PropKillingDecay}.
	What remain is to understand the moment diagram at infinity.
	But the level-sets of the distance function are $\mathbb{S}^3$, and these level-sets inherit the toric structure.
	But any isometric toric structure on $\mathbb{S}^3$ has two singular orbits.
	These singular orbits are not the zero-set of any one single vector field, and therefore the rays of $\partial\Sigma^2$ they represent are not parallel. 
	(This asymptotic structure is in the first image of in Figure \ref{FigALEFGH}).
\end{proof}

\begin{figure}[h]
	\includegraphics[scale=0.5]{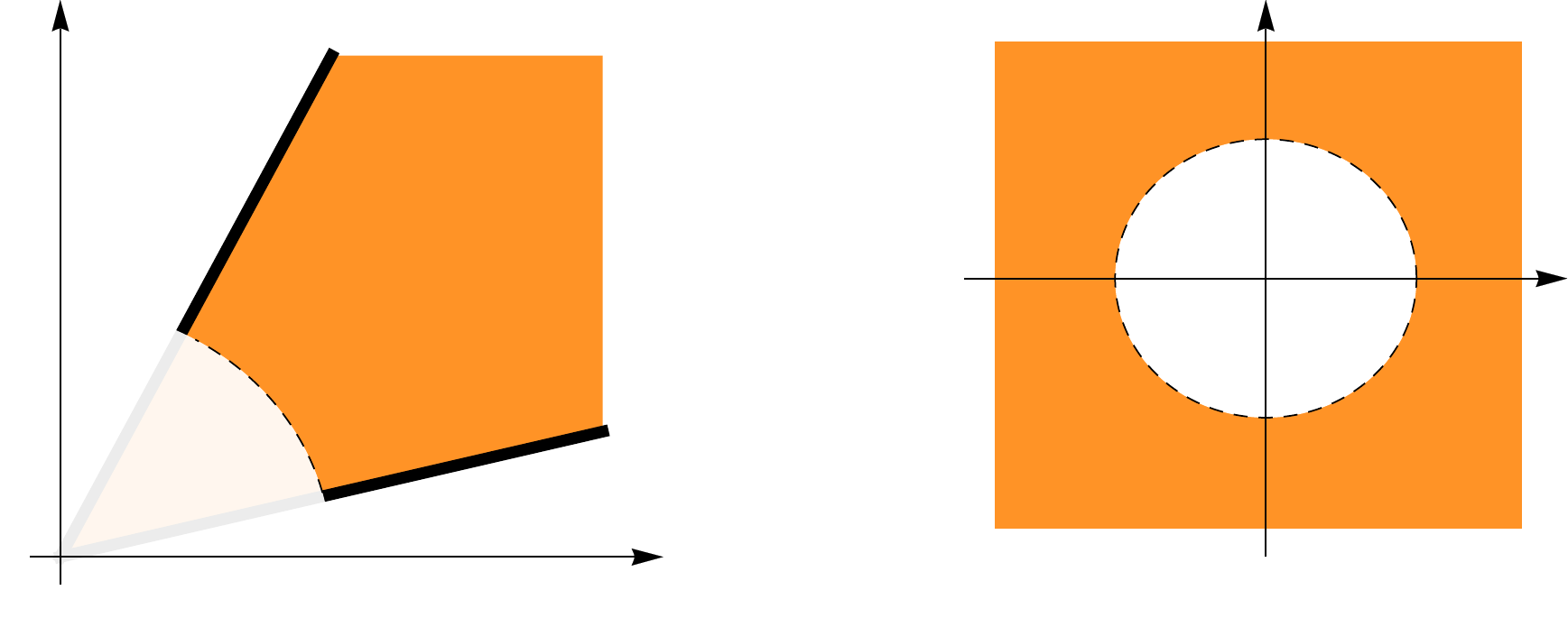}
	\caption{
		On the left is the ``asymptotically spherical'' structure, which is asymptotically a wedge.
		On the right is the ``asymptotically toroidal'' structure, which has no edges and asymptotically fills the plane. }
	\label{FigALEFGH}
\end{figure}

\begin{proposition} \label{PropToroidalFlat}
	Assume $(M^4,J,g,\mathcal{X}^1,\mathcal{X}^2)$ is a toric K\"ahler manifold with non-negative scalar curvature, and at least one end is asymptotically toroidal.
	Then its metric reduction $(\Sigma^2,g_\Sigma)$ is both topologically closed and geodesically complete.
	Consequently, $(M^4,g)$ is a flat manifold.
\end{proposition}
\begin{proof}
	$\Sigma^2$ is topologically closed by Proposition \ref{PropKillingDecay}.
	What remains is to understand the moment diagram at infinity.
	
	But the level-sets of the distance function have a free action of $T^2$, so far enough away, the polygon $\Sigma^2$ has no edges.
	Because the polygon is convex and closed, it has no edges at all, and is therefore complete.
	Therefore the metric is flat by Theorem \ref{ThmRTwoClassification}).
\end{proof}

\begin{corollary}[{\it cf.} Corollary \ref{CorIntroOneEnded}] \label{CorOneEnded}
	Assume $(M^4,J,g,\mathcal{X}^1,\mathcal{X}^2)$ is a scalar-flat toric K\"ahler manifold of finite topology that satisfies asymptotic condition (A1) or (A2).
	Let $(\Sigma^2,g_\Sigma)$ be its reduction.
	
	Then $(\Sigma^2,g_\Sigma)$ is either an infinite closed strip or else satisfies conditions (A)-(F), and there are five possibilities:
	\begin{itemize}
		\item[{\it{i}})] $(M^4,g)$ is flat,
		\item[{\it{ii}})] $\Sigma^2$ is an infinite closed strip in the $\varphi^1$-$\varphi^2$ plane,
		\item[{\it{iii}})] $(M^4,g)$ is the exceptional half-plane instanton,
		\item[{\it{iv}})] $\Sigma^2$ has parallel rays, and for given boundary values the metric belongs to a 1-parameter family of possibilities,
		\item[{\it{v}})] or $M^4$ is asymptotically spheroidal (never toroidal), and is
		\begin{itemize}
			\item[\textit{a})] Asymptotically locally Euclidean---and for any set of labels there is precisely one such metric---or
			\item[\textit{b})] Asymptotically spheriodal and asymptotically equivalent to a Taub-NUT, chiral Taub-NUT, or exceptional Taub-NUT; after labels are determined, the metric belongs to a 2-parameter family of possibilities.
		\end{itemize}
	\end{itemize}
\end{corollary}
\begin{proof}
	We verify that $\Sigma^2$ obeys conditions (A)-(F).
	By Proposition \ref{PropKillingDecay} the polygon is closed; this verifies condition (A).
	For condition (B), by convexity, if the polygon boundary were disconnected the polygon could only be the strip, so condition (B) holds.
	Conditions (C), (D), and (E) follow directly from the fact that $(\Sigma^2,g_\Sigma)$ is the reduction of a K\"ahler manifold.
	Condition (C) is the fact that $\Phi:M^4\rightarrow\Sigma^2$ is generically a Riemannian submersion.
	Condition (D), the assertion that $\nabla\varphi^i/|\nabla\varphi^i|$ is covariant constant along polygon edges, is simply the fact that the pre-image of any edge in $(M^4,g_\Sigma)$ is totally geodesic, as it is the zero-set of a Killing field.
	Condition (E) is the fact that $[\nabla\varphi^1,\nabla\varphi^2]=0$ on $M^4$ and therefore on $\Sigma^2$.
	Condition (F), the pseudo-ZSC condition, follows from $(M^4,g)$ being ZSC.
	
	The fact that the boundary numbers $s_0,\dots,s_d$ on $\Sigma^2$ are all uniquely determined follows either the definition in \S\ref{SubsecBoundaryConds} or at the beginning of \S\ref{SecAnalytic}.
	
	Since $\Sigma^2$ is topologically closed and has finitely many facets (because $M^4$ has finite topology), there are four possibilities.
	First $\Sigma^2$ has no edges, so $(\Sigma^2,g_\Sigma)$ is flat.
	Second $\Sigma^2$ is the closed half-plane, so $(\Sigma^2,g_\Sigma)$ is a half-plane instanton, as shown in Section \ref{SubsecClassHP}.
	
	Third $\Sigma^2$ might have parallel rays.
	We may assume the rays are parallel to the $\varphi^1$-axis.
	Then the level-sets $\varphi^1=Const$ for large $Const$ have an action of $\mathcal{X}^1$ with no zeros so this action splits off a circle factor, and an action of $\mathcal{X}^2$ with disconnected zero locus so the other factor is $\mathbb{S}^2$.
	The asymptotic structure is therefore $\mathbb{R}\times\mathbb{S}\times\mathbb{S}^2$.
		
	Finally $\Sigma^2$ may be the general case.
	In this case $\Sigma^2$ is asymptotically a wedge, as shown in the first image in Figure \ref{FigALEFGH}.
	Consider level-sets of a buseman $r$; because the rays are not parallel there is no single Killing field that has zeros along both rays.
	The structure of the level-sets, therefore, is a solid torus times a segment, with torus cycles identified to a point over top of the two segment endpoints.
	Thus each level set of $r$ (for $r$ sufficiently large) is a union of two solid tori, so is a lens space.
	This is the asymptotically spheroidal case.
	
	The assertion that $(M^4,J,g)$ belongs to either a 1- or 2-parameter family of possibilities, depending on $\Sigma^2$, now follows from Theorems \ref{ThmIntroGeneralPoly}, \ref{ThmIntroParallelPoly}, and \ref{ThmIntroHalfPlanePoly}.
	
	Finally we check the asymptotic possibilities for case (\textit{v}).
	In the quarter-plane, $\varphi^1$ and $\varphi^2$ proceeding to $\infty$ is the same as $x^2+y^2$ proceeding to $\infty$.
	Referring to Lemma \ref{LemmaOutlineLemma} and adding on the possible multiples of $y^2$, the momentum functions asymptotically approach
	\begin{equation}
		\begin{aligned}
		&\varphi^1\approx-\frac{s_0}{2}\left(-x+\sqrt{x^2+y^2}\right)\,+\,C_1y^2 \\
		&\varphi^2\approx\frac{s_d}{2}\left(x+\sqrt{x^2+y^2}\right)\,+\,C_2y^2
		\end{aligned}
	\end{equation}
	This is because all of the intermediate interpolation functions proceed to a constant so cease to affect the metric, and only the first and last summands continue to make any difference when $x^2+y^2$ are large.
	But, up to constant multiples, these are the momentum equations from (\ref{EqnMomTaubNut}), the Taub-NUT, Generalized Taub-NUTs, and Exceptional Taub-NUT metrics.
	Asymptotically, therefore, the metrics in case (\textit{v}) approach one of these.
\end{proof}

\subsection{The Insufficiency of the ALE-ALF-ALG-ALH schema} \label{SubSecInsufficiency}

The advantage of the concepts ``asymptotically spherical'' and ``asymptotically toroidal'' over the traditional ALE-F-G-H classification is that the traditional classification requires manifold ends to be modeled on fiber bundles over standard base manifolds, and these models are much too inflexible to represent the natural diversity of typical manifold ends.
For example, the chiral Taub-NUT metrics (discovered by Donaldson) have quadratically decaying curvature and 3-dimensional blowdowns.
But quotienting out by the collapsing field at infinity yields (at best!) a cone over an orbifold version of $\mathbb{S}^2$, whereas the ALF model requires the quotient to be a manifold.
Therefore, no matter how closely a chiral Taub-NUT may resemble a standard Taub-NUT in the larger 4-manifold setting, it will never be close to a model end as prescribed in the ALF model.

To describe the Cherkis-Kapustin framework briefly, a {\it model end} is a product $[R,\infty)\times{}N^3$ where $N^3$ is the total space of a fiber bundle $\pi:N^3\rightarrow{}B$ with fiber $F$, and which has metric
\begin{eqnarray}
	\begin{aligned}
	g_{Model}\;=\;dr^2\,+\,r^2g_B\,+\,g_F \label{EqnALEFGHModel}
	\end{aligned}
\end{eqnarray}
A manifold end fits into the schema (see \cite{CK}, \cite{Etesi}) if it ``close'' in an appropriate asymptotic sense to a finite quotient of one of the following models:
\begin{itemize}
	\item[{\it{i}})] {\it ALE} if $N^3=B^3=\mathbb{S}^3$, the fiber is a point, and $g_B$ is the standard round metric.
	\item[{\it{ii}})] {\it ALF} if $N^3=\mathbb{S}^3$ and $\pi:N^3\rightarrow{}B^2$ is the Hopf fibration, with fiber $F^1=\mathbb{S}^1$.
	The metric $g_B$ is the round metric on the 2-sphere of radius $1/2$ and the fiber metric gives $S^1$ a radius of 1.
	\item[{\it{iii}})] {\it ALG} if $N^3$ is a torus bundle over a circle (in other words, it is a solvemanifold) and $\pi:N^3\rightarrow{}B^1$ is the projection onto its base circle.
	The metric $g_B$ gives a circle and $g_F$ gives a flat 2-torus.
	\item[{\it{iv}})] {\it ALH} if $N^3$ is the 3-torus and $\pi:N^3\rightarrow{}B^0$ is the projection onto a point; the metric $g_F$ gives a flat 3-torus.
\end{itemize}
The paper \cite{Web3} explains, in exacting detail, the asymptotics of the chiral Taub-NUT metrics.
The asymptotics are ALF-like in many respects, but definitely cannot be modeled in the form of (\ref{EqnALEFGHModel}).
The basic reason is that the level-sets $N^3$ might not have a collpasing field at infinity that can possibly be modeled on a fiber-bundle over a manifold, or even a fibration over a manifold.
The base of the projection $\pi:N^3\rightarrow{}B$ is often an orbifold, and a ``bad'' orbifold at that.
Worse, the fiber $F$ might have 2-dimensional closure (eg. an irrational subgroup of the torus), so $B$ does not have a manifold structure of any kind.
One could quotient by the closure of the fibers, but then the base would be a closed line segment which is also not part of the model.

That these situations actually happen very naturally gives a concrete illustration of the basic insufficiency of the ALF model, and the need for the terms ``asymptotically spherical'' (if not ``asymptocially toroidal,'' for which there are no examples in the scalar-flat K\"ahler setting as we have proved in Corollary \ref{CorRTwoClassification} and Proposition \ref{PropToroidalFlat}).
Some very naturally occurring metrics with close resemblance to the classic Taub-NUT simply cannot be placed into a form that fits the ALF model.

\section{Examples} \label{SecExamples}

\subsection{Non-pathological examples} \label{SubsecNonPathExamples}

\subsubsection{The LeBrun metrics on $\mathcal{O}(-k)$} \label{SubsubsecOk}

Let $e_1,e_2,e_3$ be the left-invariant frames and $\sigma^1,\sigma^2,\sigma^3$ the corresponding left-invariant coframes on $\mathbb{S}^3$, normalized in the usual way so $d\sigma^i=-\epsilon^i{}_{jk}\sigma^j\wedge\sigma^k$.
From \cite{LebOK} a scalar-flat metric on $\mathcal{O}(-k)$ is
\begin{equation*}
	g=\frac{1}{\left(1-\frac{1}{r^2}\right)\left(1+\frac{k-1}{r^2}\right)}dr^2
	+r^2\left(1-\frac{1}{r^2}\right)\left(1+\frac{k-1}{r^2}\right)(\sigma^1)^2
	+r^2\Big((\sigma^2)^2+(\sigma^3)^2\Big)
\end{equation*}
where $r\in[1,\infty)$.
The K\"ahler form is
\begin{equation}
	\omega\;=\;-rdr\wedge\sigma^1+r^2\sigma^2\wedge\sigma^3. \label{LeBKahlForm}
\end{equation}
From $d(\sigma^2\wedge\sigma^3)=0$ and $d\sigma^1=-2\sigma^2\wedge\sigma^3$, we easily verify $d\omega=0$.
To find commuting Killing fields we must recall the Euler coordinates on $\mathbb{S}^3$ and corresponding polar coordinates on $\mathbb{R}^4\approx\mathbb{C}^2$, given by
\begin{equation}
	(r,\psi,\theta,\varphi)
	\;\longmapsto\;\left(
	r\cos(\theta/2)e^{-\frac{i}{2}(\psi+\varphi)},\;
	r\sin(\theta/2)e^{-\frac{i}{2}(\psi-\varphi)}
	\right) \label{EqnEulerCoords}
\end{equation}
with ranges $r\in[0,\infty)$, $\psi,\varphi\in[0,4\pi)$, and $\theta\in[0,\pi]$.
We choose toric structure
\begin{equation}
	\mathcal{X}^1\;=\;2\frac{\partial}{\partial\psi}, \quad
	\mathcal{X}^2\;=\;2\frac{\partial}{\partial\varphi}. \label{EqnOkKilling}
\end{equation}
In terms of coordinates, the left-invariant 1-forms are
\begin{equation}
	\begin{aligned}
	&\sigma^1=\frac12(d\psi+\cos(\theta)d\varphi), \quad
	\sigma^2=\frac12(\sin(\psi)d\theta-\cos(\psi)\sin(\theta)d\varphi), \\
	&\sigma^3=\frac12(\cos(\psi)d\theta+\sin(\psi)\sin(\theta)d\varphi).
	\end{aligned}
\end{equation}
With these expressions we can plug the fields (\ref{EqnOkKilling}) in to (\ref{LeBKahlForm}) to get
\begin{equation}
	\begin{aligned}
	&d\varphi^1
	\;=\;\omega\left(2\frac{\partial}{\partial\psi},\;\cdot\,\right)
	\;=\;rdr\;=\;d\left(\frac12r^2\right) \\
	&d\varphi^2
	\;=\;\omega\left(2\frac{\partial}{\partial\varphi},\;\cdot\,\right)
	\;=\;r\cos(\theta)\,dr-\frac12r^2\sin(\theta)\,d\theta
	\;=\;d\left(\frac12r^2\cos\theta\right)
	\end{aligned}
\end{equation}
therefore our symplectic coordinates are $(\varphi^1,\varphi^2,\theta_1,\theta_2)$ where $\varphi^1=\frac12r^2$, $\varphi^2=\frac12r^2\cos\theta$, $\theta_1=\frac12\psi$, and $\theta_2=\frac12\varphi$.
The coordinate ranges are $\varphi^1\ge1/2$, $-\varphi^1\le\varphi^2\le\varphi^1$, and $\theta_1,\theta_2\in[0,2\pi)$; see Figure \ref{FigLeBruns}.
The polygon metric is then
\begin{equation}
	(g_\Sigma^{ij})
	\;=\;\left(\begin{array}{cc}
		\frac{(2\varphi^1-1)(2\varphi^1+(k-1))}{2\varphi^1} & \frac{(2\varphi^1-1)(2\varphi^1+(k-1))\varphi^2}{2(\varphi^1)^2} \\
		\frac{(2\varphi^1-1)(2\varphi^1+(k-1))\varphi^2}{2(\varphi^1)^2} & 2\varphi^1+\frac{(2\varphi^1(k-2)-(k-1))(\varphi^2)^2}{2(\varphi^1)^3}
	\end{array}\right).
\end{equation}
The volumetric normal coordinates, using $y=\sqrt{\det{}g_{\Sigma^{ij}}}$ and (\ref{EqnJSigma}) for $J_\Sigma$, are
\begin{equation}
	y\;=\;\frac{1}{\varphi^1}\sqrt{(2\varphi^1-1)(2\varphi^1+k-1)((\varphi^1)^2-(\varphi^2)^2)},
	\quad
	x\;=\;\frac12(k-2)\frac{\varphi^2}{\varphi^1}+2\varphi^2.
\end{equation}
Inverting this gives the $\varphi^i$ as functions of $x$ and $y$:
\begin{equation}
	\begin{aligned}
	\varphi^1
	&\;=\;
	\frac12
	\left(
	1\;-\;\frac{k}{2}
	\;+\;\frac12\left(\sqrt{\left(x-\frac{k}{2}\right)^2+y^2}
	+\sqrt{\left(x+\frac{k}{2}\right)^2+y^2}\right)
	\right), \\
	\varphi^2
	&\;=\;\frac12\left(
	x+\left(\frac12-\frac1k\right)\left(
	\sqrt{\left(x-\frac{k}{2}\right)^2+y^2}
	-\sqrt{\left(x+\frac{k}{2}\right)^2+y^2}
	\right)
	\right).
	\end{aligned}
\end{equation}
One may check directly that $y(\varphi^i_{xx}+\varphi^i_{yy})-\varphi^i_y=0$.
The ``oultine map'' is the restriction of $(\varphi^1,\varphi^2)$ to $\{y=0\}$.
We compute
\begin{equation*}
	(\varphi^1,\varphi^2)\Big|_{y=0}=
	\begin{cases}
		\left(\frac14(2-k-2x),\;\frac14(-2+k+2x)\right),\quad & x<-\frac{k}{2} \\
		\left(\frac12,\;\frac1kx\right),\quad & -\frac{k}{2}\le{}x\le\frac{k}{2} \\
		\left(\frac14(2-k+2x),\;\frac14(2-k+2x)\right),\quad & x>\frac{k}{2}
	\end{cases}
\end{equation*}
\noindent\begin{figure}[h!] 
	\vspace{-0.125in}
	\centering
	\includegraphics[scale={0.375}]{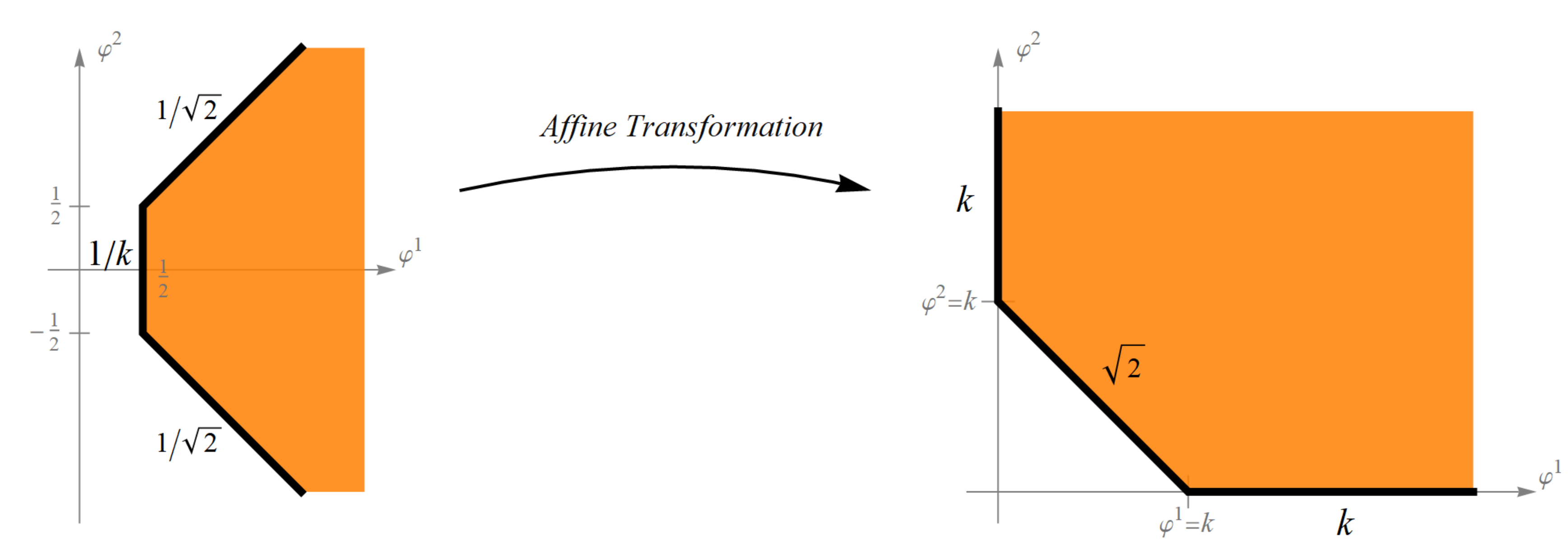}
	\caption{
		Polygons and computed labels for the LeBrun metrics on $\mathcal{O}(-k)$, before and after an affine change of variables.
	}
	\label{FigLeBruns}
\end{figure}

Finally we compute the ``parameterization speeds'' or labels.
Using (\ref{EqnMarkings}),
\begin{equation}
	\begin{aligned}
		&s
		\;=\;\sqrt{(\varphi^1_x)^2+(\varphi^2_x)^2}\Big|_{y=0}
		\;=\;\begin{cases}
			\frac{1}{\sqrt{2}} \quad & x<-\frac{k}{2} \\
			\frac{1}{k} & -\frac{k}{2}<x<\frac{k}{2} \\
			\frac{1}{\sqrt{2}} \quad & x>\frac{k}{2}
		\end{cases}
	\end{aligned}
\end{equation}
Finally note that this polygon is not in the right position for the outline matching of Section \ref{SubsecOutlineMatching}, in the sense that its terminal rays are not parallel to the first quadrant axes.
This is rectified with the affine transformation
\begin{equation}
	\left(\begin{array}{c}
		\varphi^1 \\\varphi^2
	\end{array}\right)
	\;\longmapsto\;
	\left(\begin{array}{cc}
		k & -k \\
		k & k
	\end{array}\right)
	\left(\begin{array}{c}
		\varphi^1 \\\varphi^2
	\end{array}\right).
\end{equation}
Lastly we see how to use the polygon markings to reconstruct the manifold.
We can use either polygon of Figure \ref{FigLeBruns}.
We choose the polygon on the right.
Choose a torus action with $\mathcal{X}_1$, $\mathcal{X}_2$ being standard generators with coordinate ranges $\theta_1,\theta_2\in[0,2\pi{}k)$.
The diagonal action, with its zero-set being the diagonal segment, has range $[0,2\sqrt{2}\pi{}k)$.
Referring to the discussion in Section \ref{SubsecBoundaryConds}, the cone angle along each ray is the range divided by the label, so we see cone angle $2\pi$ along the two rays, but cone angle $2\pi{}k$ along the diagonal segment.
This is rectified by taking a $k$-to-1 torus quotient along the diagonal.
Each of the generators still has range $[0,2\pi{}k)$ so the cone angles along the rays remains $2\pi$, but now the diagonal segment's cone angle is also $2\pi$, so we have a manifold.

\subsubsection{Polygons for $\mathbb{S}^2\times\mathbb{H}^2$} \label{SubsecSTimesH}

Up to isometry there is essentially just one Killing field on $\mathbb{S}^2$ (the rotational field), but on the hyperbolic plane $\mathbb{H}^2$ there are three: the elliptic field (which fixes a point), the parabolic field (which fixes a point on the ideal circle), and the hyperbolic field (which fixes two points on the ideal circle).
The parabolic and hyperbolic fields have no zeros, so if we think of these fields as coming from actions of $\mathbb{R}^1$ on $\mathbb{H}^2$ then we can mod-out by integer translations to obtain $\mathbb{S}^1$ action on $\mathbb{H}^2/\mathbb{Z}$; 

Given the elliptic field on $\mathbb{H}^2$, the moment functions for the product metric are
\begin{eqnarray}
	\begin{aligned}
	&\varphi^1=\frac12\left(
	-1+\sqrt{x^2+y^2}
	+\sqrt{(x-1)^2+y^2}
	\right) \\
	&\varphi^2=\frac{1}{2}\left(
	1-\sqrt{x^2+y^2}+\sqrt{(x-1)^2+y^2}
	\right).
	\end{aligned}
\end{eqnarray}
Given the parabolic field, the moment functions for the product metric are
\begin{eqnarray}
	\begin{aligned}
	&\varphi^1=\frac12\left(
	\sqrt{x^2+y^2}
	\right), \quad
	\varphi^2=\frac{1}{2}\left(
	1+\frac{x}{\sqrt{x^2+y^2}}
	\right).
	\end{aligned}
\end{eqnarray}
The hyperbolic field, on the other hand, produces a unique situation.
Its polygon is the infinite strip and, notably, its volumetric function $z:\Sigma^2\rightarrow\overline{}H{}^2$ is ramified, having a 2-to-1 branch point.
On any subdomain of $\Sigma^2$ on which $z$ is injective we can still perform the constructions of this paper, and obtain moment functions
\begin{eqnarray}
	\begin{aligned}
	&\varphi^1
	\;=\;\frac{1}{\sqrt{2}}\sqrt{1-\left(x^2+y^2\right)^2+\sqrt{\left(1-\left(x^2+y^2\right)^2\right)^2\,+\,4y^2}} \\
	&\varphi^2
	\;=\;\frac{1}{\sqrt{2}}\sqrt{-1+\left(x^2+y^2\right)^2+\sqrt{\left(1-\left(x^2+y^2\right)^2\right)^2\,+\,4y^2}}.
	\end{aligned}
\end{eqnarray}
These moment functions are singular.
The map $(\varphi^1,\varphi^2):\overline{H}{}^2\rightarrow\Sigma^2$, with these moment functions, brings $\overline{H}{}^2$ to precisely one-half of the polygon.
\noindent\begin{figure}[h!] 
	\centering
	\vspace{-0.0in}
	\hspace{-0.6in}
	\includegraphics[scale={0.4}]{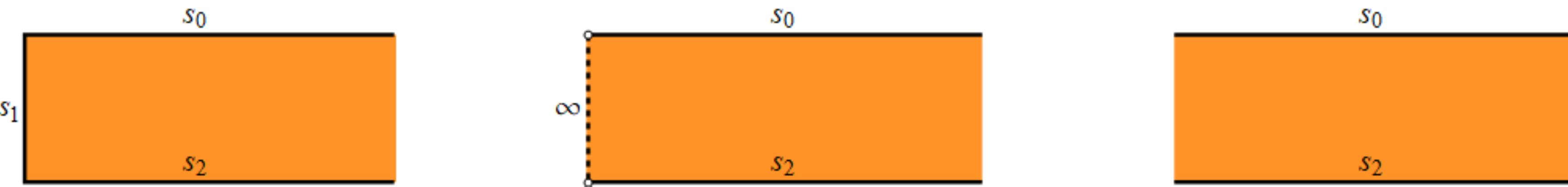} \\
	\hspace{0.1in}
	\includegraphics[scale={0.36}]{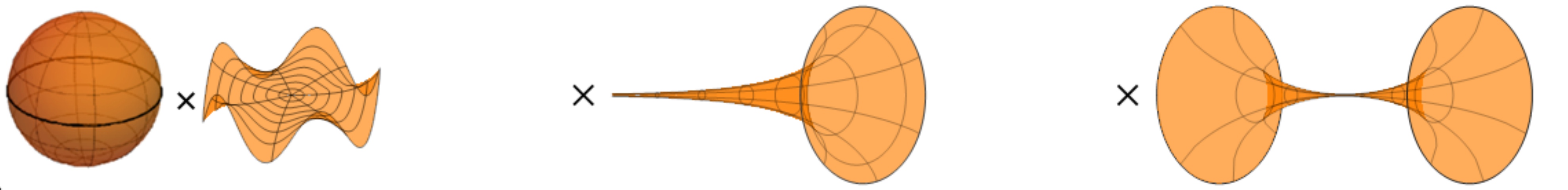}
	\caption{Three polygons that produce scalar-flat metrics on $\mathbb{S}^2\times\mathbb{H}^2$, along with a depiction of the three corresponding rotational structures on $\mathbb{H}^2$: elliptic, parabolic, and hyperbolic.
	}
	\label{FigThreeStrips}
\end{figure}

This construction gives rise to orbifold and conifold metrics as well, which we obtain by adjusting the speeds.
Setting $s_0,s_2>0$ and using 
\begin{eqnarray}
	\begin{aligned}
	&\varphi^1=\frac{s_0}{2}\left(
	-x+\sqrt{x^2+y^2}\right)
	+\frac{s_2}{2}\left(
	x-1+\sqrt{(x-1)^2+y^2}
	\right) \\
	&\varphi^2=\frac12\left(1+\frac{x}{\sqrt{x^2+y^2}}\right).
	\end{aligned}
\end{eqnarray}
the polygon is the semi-open half-strip, depicted in the second image of Figure \ref{FigThreeStrips} or in Figure \ref{FigOrbifolds}, but with parameterization speeds $s_0$, and $s_2$.
When $s_0=s_2$ is rational, the construction gives a ``good orbifold'' on $\mathbb{S}^2$; when $s_0\ne{}s_2$ the construction can give a ``bad orbifold.''
When $s_0$ or $s_2$ are irrational, the construction gives a conifold.

\noindent\begin{figure}[h!] 
	\centering
	\vspace{-0.0in}
	\hspace{-0.0in}
	\includegraphics[scale={0.4}]{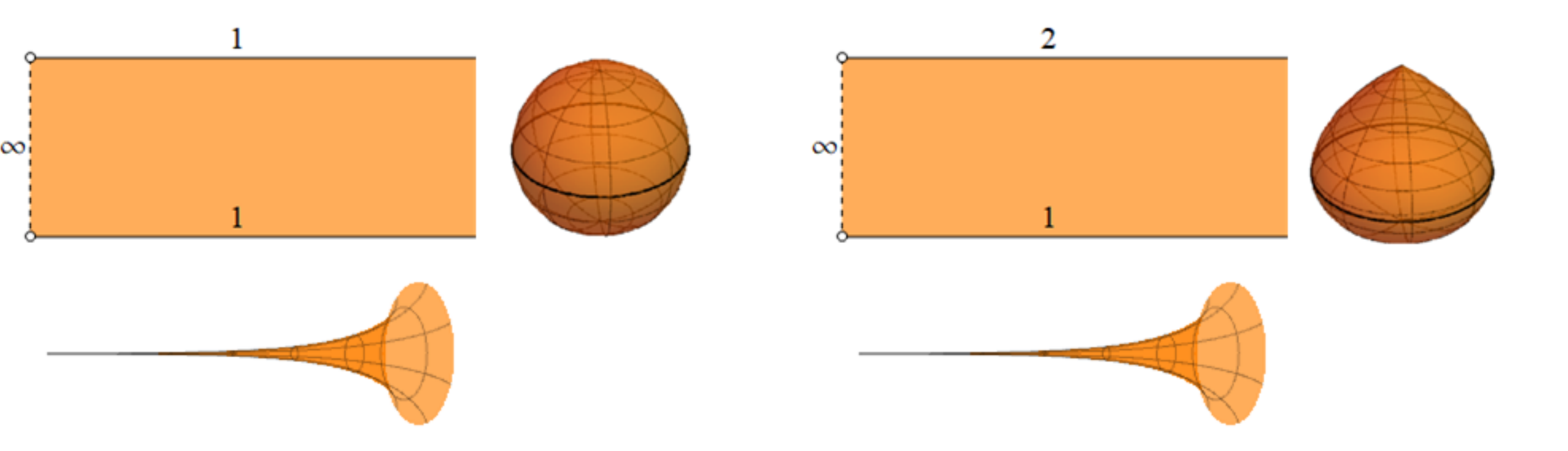}
	\caption{
		Illustration of how polygon labels indicate topology.
		A cross product of smooth $\mathbb{S}^2$ with $\mathbb{H}^2$, and a cross product of a bad orbifold with $\mathbb{H}^2$.
		The latter is a topological product, but, in the scalar-flat case, is never a metric product.
	}
	\label{FigOrbifolds}
\end{figure}


\subsubsection{A two-ended instanton} \label{SubsubsecTwoEnded}

Consider the moment functions
\begin{eqnarray}
	\begin{aligned}
		\varphi^1&\;=\;\frac{1}{2}\left(x+\sqrt{x^2+y^2}\right)
		+\frac12\left(-1+\frac{x}{\sqrt{x^2+y^2}}\right)
		+\frac{M(1-k)}{2}y^2 \\
		\varphi^2&\;=\;\frac{1}{2}\left(-x+\sqrt{x^2+y^2}\right)
		+\frac12\left(1-\frac{x}{\sqrt{x^2+y^2}}\right)
		+\frac{M(1+k)}{2}y^2
	\end{aligned} \label{EqnMomsTwoEndedFull}
\end{eqnarray}
where $M\ge0$ and $k\in[-1,1]$ are constants.
Each term $x(x^2+y^2)^{-\frac12}$ creates a ``jump'' singularity for the moment functions along the line $\{y=0\}$, and produces a non-closed segment on the polygon boundary, which should be considered a segment of infinite parameterization speed.
Figure \ref{FigTwoEndedPlots} depicts these functions and the semi-open polygon they generate.

\noindent\begin{figure}[h!] 
	\centering
	\vspace{-0.0in}
	\hspace{-0.3in}
	\includegraphics[scale={0.4}]{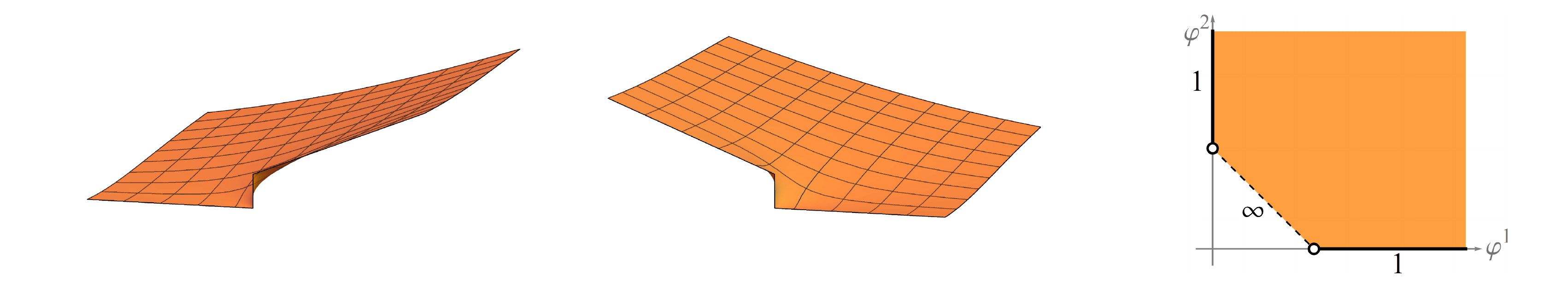}
	\caption{Graphs of $\varphi^1$ and $\varphi^2$ over the upper half-plane, and the image of the map $(x,y)\mapsto(\varphi^1,\varphi^2)$.
	}
	\label{FigTwoEndedPlots}
\end{figure}
A full exploration of this two-ended manifold $(M^4,g,J,\mathcal{X}^1,\mathcal{X}^2)$ would require a lot of space, so we state the overall properties without recording the justifications.

If $M=0$, one end is asymptotically Euclidean (AE) and the other end is cuspidal, whose structure is similar to the cusp-like middle image in Figure \ref{FigThreeStrips}.
This manifold is scalar-flat and has bounded sectional curvature.
It is not Ricci-flat for any values of $M\ge0$, $k\in[-1,1]$.
It has finite energy; in the case $M=0$ we have $\int|\Riem|^2dVol=96\pi^2$.

If $M\ne0$ then this manifold is still two-ended and complete.
Sectional curvature remains bounded.
The cuspidal end remains cuspidal, but when $k\in(-1,1)$ the AE end becomes spheroidal with cubic volume growth---it is technically not ALF unless $k=0$.
The energy $\int|\Riem|^2$ remains finite.
In the $k=-1,1$ case, the end is no longer ALF-like.
It has quartic volume growth, and its asymptotic structure is that of the exceptional Taub-NUT \cite{Web3}.

\subsubsection{A reduced manifold $(\Sigma^2,g_\Sigma)$ which is not a polygon} \label{SubsubsecNonpolygon}

Consider the function
\begin{eqnarray}
	{\bf{u}}(\varphi^1,\varphi^2)\;=\;-\frac12\log\left(1-\left(\varphi^1\right)^2-\left(\varphi^2\right)^2 \right) \label{EqnDefOfUu}
\end{eqnarray}
defined on the open disk $\Sigma^2=\{\left(\varphi^1\right)^2+\left(\varphi^2\right)^2<1\}$.
Interpreting this as a symplectic potential, we obtain the ``polygon'' metric $g_{\Sigma,ij}={\bf{u}}_{ij}d\varphi^i\otimes{}d\varphi^j$ where ${\bf{u}}_{ij}$ given by ${\bf{u}}_{ij}\triangleq\frac{\partial^2{\bf{u}}}{\partial\varphi^i\partial\varphi^j}$.
Explicitly this is
\begin{eqnarray}
	{\bf{u}}_{ij}
	\;=\;
	\left(\begin{array}{cc}
		\frac{1+\left(\varphi^1\right)^2-\left(\varphi^2\right)^2}{\left(1-\left(\varphi^1\right)^2-\left(\varphi^2\right)^2\right)^2} & \frac{\varphi^1\varphi^2}{\left(1-\left(\varphi^1\right)^2-\left(\varphi^2\right)^2\right)^2} \\
		\frac{\varphi^1\varphi^2}{\left(1-\left(\varphi^1\right)^2-\left(\varphi^2\right)^2\right)^2} & \frac{1-\left(\varphi^1\right)^2+\left(\varphi^2\right)^2}{\left(1-\left(\varphi^1\right)^2-\left(\varphi^2\right)^2\right)^2}
	\end{array}\right).
\end{eqnarray}

\noindent\begin{figure}[h!] 
	\centering
	\vspace{-0.0in}
	\hspace{-0.3in}
	\includegraphics[scale={0.4}]{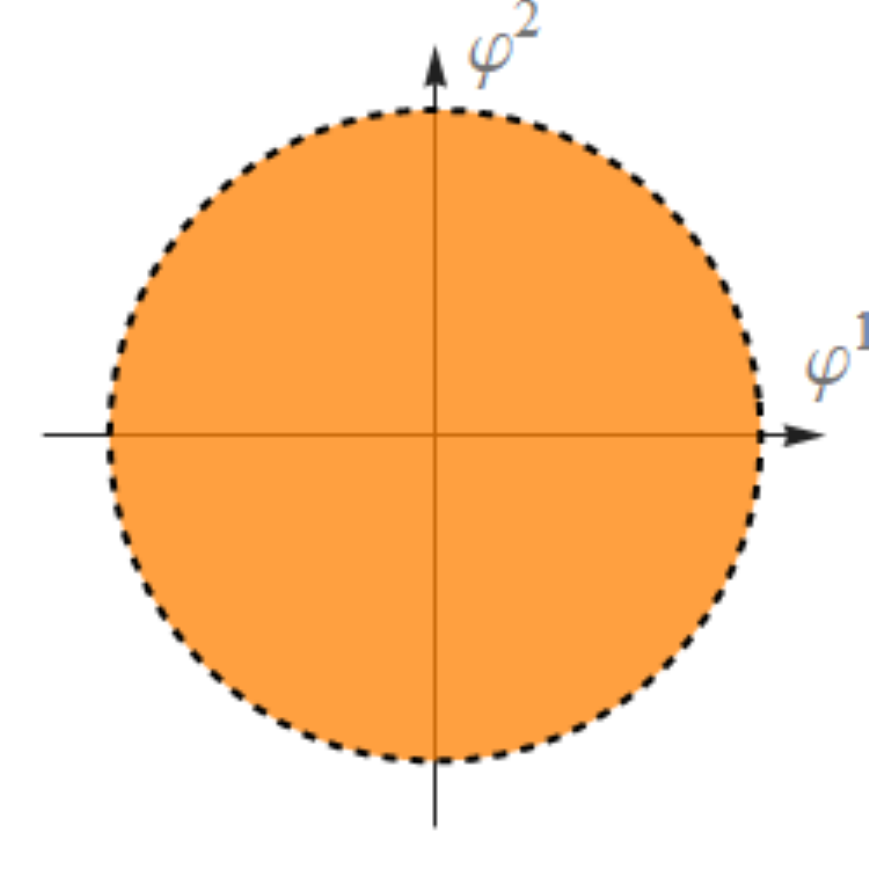}
	\caption{The reduction $\Sigma^2$ of a complete K\"ahler manifold with bounded curvature, whose symplectic potential is given by (\ref{EqnDefOfUu}).
		This is not a polygon.
	}
	\label{FigDiskPolygon}
\end{figure}

The reduced manifold $(\Sigma^2,g_\Sigma)$ of Figure \ref{FigDiskPolygon} with potential ${\bf{u}}$ is a complete 2-manifold, and is the reduction of a complete toric K\"ahler manifold $(M^4,J,g,\mathcal{X}^1,\mathcal{X}^2)$ on which the symplectomorphic fields are never colinear or zero.
Setting $({\bf{u}}^{ij})\triangleq({\bf{u}}_{ij})^{-1}$, the Abreu equation (which is (10) of \cite{Ab1}) gives scalar curvature $R=-\frac12\frac{\partial^2{\bf{u}}^{ij}}{\partial\varphi^i\partial\varphi^j}$.
The scalar curvature of $(M^4,g)$ is not signed; at the origin $R=+8$ and asymptotically $R$ decreases to $-3$.
Computing the full Riemann tensor, one finds that all sectional curvatures are bounded.

\subsection{Pathological examples} \label{SubsecPathExamples}

\subsubsection{Non-convex polygons} \label{SubsubsecNonConvex}
The outline-matching programme of Lemma \ref{LemmaOutlineLemma} does not require that the polygons be convex.
Indeed we can create non-convex shapes and even shapes with self-intersecting outlines as in Figure \ref{FigPath1}.
Although we will not prove it here, all such shapes create pathologies.
In particular their polygon metrics will always be singular.
\noindent\begin{figure}[h!] 
	\centering
	\includegraphics[scale={0.4}]{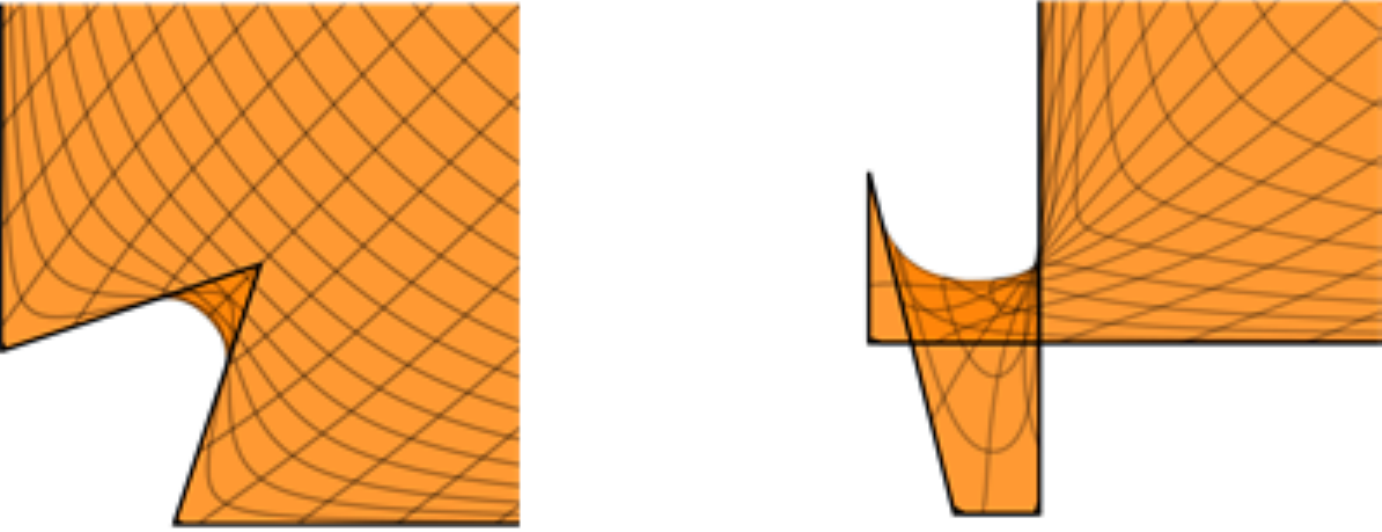}
	\caption{Examples of outline matching with pathological results: moment functions with non-convex outline and moment functions with self-intersecting outline.
	}
	\label{FigPath1}
\end{figure}

\subsubsection{Removed Edges}

To extend the theory of this paper to the case of non-closed polygons, one would have to contend with many potential pathologies.
Removing boundary segments or rays can result in a non-pathological instanton, as the 2-ended instanton of \ref{SubsubsecTwoEnded}, but can create pathologies.
\noindent\begin{figure}[h!] 
	\centering
	\includegraphics[scale={0.6}]{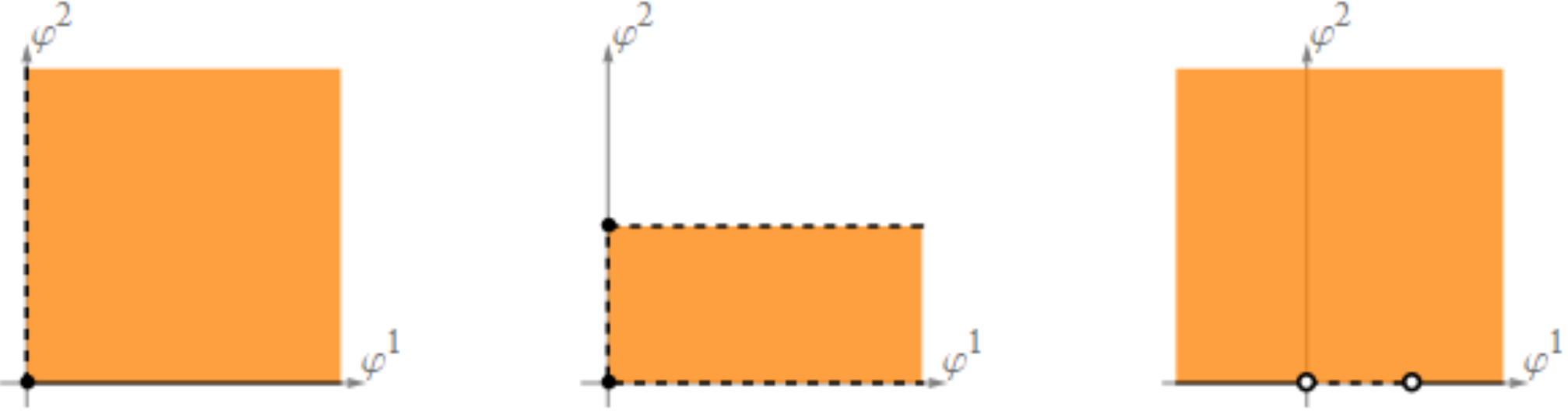}
	\caption{Pathological examples.
	}
	\label{FigPath2}
\end{figure}

Moment functions for the quarter-plane with a removed ray are
\begin{eqnarray}
	\begin{aligned}
		\varphi^1\;=\;x+\sqrt{x^2+y^2}, \quad
		\varphi^2=y^2.
	\end{aligned}
\end{eqnarray}
Moment functions for pathological half-strip are
\begin{eqnarray}
	\begin{aligned}
		\varphi^1\;=\;y^2, \quad
		\varphi^2\;=\;\frac12\left(1-\frac{x}{\sqrt{x^2+y^2}}\right).
	\end{aligned}
\end{eqnarray}
Moment functions for the half-plane with a removed segment are
\begin{eqnarray}
	\begin{aligned}
	\varphi^1\;=\;x+\frac{x}{\sqrt{x^2+y^2}}, \quad
	\varphi^2\;=\;y^2.
	\end{aligned}
\end{eqnarray}
These examples do have associated scalar-flat instantons $(M^4,g,J,\mathcal{X}^1,\mathcal{X}^2)$, but all three have curvature singularities.

\end{document}